\documentclass[reqno]{amsart}

\usepackage{hyperref}
\hypersetup{
	colorlinks=true,
	citecolor=blue,
	linkcolor=blue,
	filecolor=magenta,      
	urlcolor=cyan,
}

\usepackage{adjustbox}

\usepackage[margin=1.3in]{geometry}
\usepackage[normalem]{ulem}
\usepackage{setspace}
\usepackage{thmtools}
\usepackage[capitalize,noabbrev,nameinlink]{cleveref}
\usepackage{amsmath, amsthm, amssymb, graphicx, dsfont, stmaryrd, extarrows, enumitem, mathtools}
\usepackage[T1]{fontenc}
\usepackage{textcomp}
\usepackage{lmodern}
\usepackage{microtype} 
\usepackage{amscd}
\usepackage{mathrsfs}
\usepackage{tikz-cd}
\usetikzlibrary{matrix,patterns}
\usetikzlibrary{decorations.markings}
\tikzset{negated/.style={
        decoration={markings,
            mark= at position 0.5 with {
                \node[transform shape] (tempnode) {\scalebox{2}{$\times$}};
            }
        },
        postaction={decorate}
    }
}
\usepackage[colorinlistoftodos]{todonotes}
\usepackage{color}
\usepackage{bbm}
\usepackage{verbatim}
\usepackage[mathscr]{euscript}
\usepackage{bm}

\numberwithin{equation}{section}

\theoremstyle{plain}
\newtheorem{thm}[equation]{Theorem}
\newtheorem*{thm*}{Theorem}

\newtheorem{thmx}{Theorem}

\crefname{thmx}{Theorem}{Theorems}
\newtheorem{prop}[equation]{Proposition}
\crefname{prop}{Proposition}{Propositions}

\crefname{question}{Question}{Questions}
    
\newtheorem{cor}[equation]{Corollary}       
\newtheorem{lem}[equation]{Lemma}

\theoremstyle{definition} 
\newtheorem{defn}[equation]{Definition} 

\newtheorem{ex}[equation]{Example}
\newtheorem*{ex*}{Example}
\crefname{ex}{Example}{Examples}
\newtheorem{exs}[equation]{Examples}
\newtheorem{rem}[equation]{Remark}   

\newtheorem{nota}[equation]{Notation}

\newtheorem{chunk}[equation]{}
\crefformat{chunk}{(#2#1#3)}
\crefformat{equation}{(#2#1#3)}

\newcommand{\Z}{\mathbb{Z}}
\newcommand{\Q}{\mathbb{Q}}

\newcommand{\Hom}{\mathrm{Hom}}
\newcommand{\iHom}{\underline{\mathrm{Hom}}}

\newcommand{\msf}[1]{\mathsf{#1}}

\newcommand{\mc}[1]{\mathcal{#1}}
\newcommand{\mrm}[1]{\mathrm{#1}}
\newcommand{\mbb}[1]{\mathbb{#1}}
\newcommand{\scr}[1]{\mathscr{#1}}

\newcommand{\p}{\mathfrak{p}}

\newcommand{\m}{\mathfrak{m}}

\newcommand{\B}{\mathscr{B}}
\newcommand{\C}{\mathscr{C}}
\newcommand{\D}{\mathscr{D}}

\newcommand{\K}{\mathcal{K}}

\newcommand{\W}{\mathcal{W}}

\newcommand{\1}{\mathds{1}}

\newcommand{\res}{\mathrm{res}}

\newcommand{\Sp}{\mathrm{Sp}}

\newcommand{\Mod}[1]{\mathrm{Mod}_{#1}}

\newcommand{\CAlg}[1]{\mathrm{CAlg}(#1)}

\newcommand{\Loc}{\operatorname{Loc}}
\newcommand{\thick}{\operatorname{thick}}
\newcommand{\Loctensor}[1]{\operatorname{Loc}_\otimes({#1})}
\newcommand{\thicktensor}[1]{\operatorname{thick}_\otimes({#1})}

\renewcommand{\hat}{\widehat}
\renewcommand{\hat}[1]{\widehat{#1}}
\renewcommand{\bar}{\overline}
\newcommand{\wt}{\widetilde}

\newcommand{\op}{\mrm{op}}

\renewcommand{\phi}{\varphi}

\newcommand{\Pic}{\mathrm{Pic}}
\newcommand{\ev}{\mrm{ev}}

\newcommand{\Gammaf}{\Gamma_{\!f}}
\newcommand{\Lambdaf}{\Lambda_{\!f}}
\newcommand{\Ef}{\mathcal{E}_{\!f}}

\newcommand{\DCh}{\msf{D}_\mrm{Ch}}
\newcommand{\DDG}{\msf{D}_\mrm{DG}}
\newcommand{\EffLift}{\msf{EffLift}}

\usepackage{todonotes}

\title{Duality in tensor-triangular geometry via proxy-smallness}
\begin{document}
    \author{Thomas Peirce}
    \address[Peirce]{Mathematics Institute, University of Warwick, Zeeman Building, Coventry, CV4 7AL, UK}
    \email{Thomas.Peirce@warwick.ac.uk}
	\author{Jordan Williamson}
	\address[Williamson]{Department of Algebra, Faculty of Mathematics and Physics, Charles University in Prague, Sokolovsk\'{a} 83, 186 75 Praha, Czech Republic}
	\email{williamson@karlin.mff.cuni.cz}
    \begin{abstract}
    We make a systematic study of duality phenomena in tensor-triangular geometry, generalising and complementing previous results of Balmer--Dell'Ambrogio--Sanders and Dwyer--Greenlees--Iyengar. A key feature of our approach is the use of proxy-smallness to remove assumptions on functors preserving compact objects, and to this end we introduce proxy-small geometric functors and establish their key properties. Given such a functor, we classify the rigid objects in its associated torsion category, giving a new perspective on results of Benson--Iyengar--Krause--Pevtsova. As a consequence, we show that any proxy-small geometric functor satisfies Grothendieck duality on a canonical subcategory of objects, irrespective of whether its right adjoint preserves compact objects. We use this as a tool to classify Matlis dualising objects and to provide a suitable generalisation of the Gorenstein ring spectra of Dwyer--Greenlees--Iyengar in tensor-triangular geometry. We illustrate the framework developed with various examples and applications, showing that it captures Matlis duality and Gorenstein duality in commutative algebra, duality phenomena in chromatic and equivariant stable homotopy theory, and Watanabe's theorem in polynomial invariant theory.
\end{abstract}
    \subjclass[2020]{18F99, 18G80, 55U35, 55P60, 13D45}
	\maketitle
    
    \setcounter{tocdepth}{1}
\tableofcontents

\newpage
\section{Introduction}
Tensor-triangulated (`tt') categories arise across mathematics: for instance in algebra, geometry, topology, representation theory, and non-commutative geometry~\cite{TTGICM}. The machinery of tensor-triangular geometry allows one to give unified systematic treatments of various phenomena, for instance the classification of thick tensor ideals, or more pertinently for this paper, duality theorems.
The goal of this paper is to develop a framework which is broad enough to capture various duality results in tensor-triangular geometry, building on work of Balmer--Dell'Ambrogio--Sanders~\cite{BDS} and Dwyer--Greenlees--Iyengar~\cite{DGI}, and to demonstrate their applicability in various concrete examples. For example, we give a new conceptual perspective on Watanabe's theorem from polynomial invariant theory, realising it as the shadow of a more general equivariant duality theorem.  

\subsection*{The context}
We work in the setting of rigidly-compactly generated tt-categories (see \cref{sec:ttbackground} for details). A geometric functor $f^*\colon \C \to \D$ between rigidly-compactly generated tt-categories is a coproduct-preserving, strong symmetric
monoidal, triangulated functor. Brown representability ensures that such a geometric functor fits into an adjoint triple
\[\begin{tikzcd}[column sep=2cm]
        \C \ar[r, "f^*", yshift=2.5mm] \ar[r, "f^{!}"', yshift=-2.5mm] & \D. \ar[l, "f_*" description]
    \end{tikzcd}\]
with left adjoints denoted on top. We emphasise that we use geometric notation: for example, one may take $f^*$ to be the (derived) extension of scalars functor along a map of commutative rings, or the restriction functor from $G$-spectra/$G$-representations to $H$-spectra/$H$-representations for $H$ a subgroup of $G$. 

To explain the main goals and results of this paper, we recall the following results of Balmer--Dell'Ambrogio--Sanders~\cite{BDS} where we write $\omega_f = f^!(\1_\C)$. Given a geometric functor $f^*\colon \C \to \D$ between rigidly-compactly generated tt-categories, they prove that the following are equivalent:
 \begin{enumerate}[label=(\roman*)]
        \item $f_*$ preserves compacts;
        \item Grothendieck duality holds for all $X \in \C$, in that there is a natural isomorphism \[\omega_f \otimes f^*X \simeq f^!X;\]
        \item there is an adjoint quintuple $f_! \dashv f^* \dashv f_* \dashv f^! \dashv f_\#$.
    \end{enumerate}
    If these equivalent conditions hold, then we say that $f^*$ has \emph{Grothendieck--Neeman duality}.
    Moreover, if $f^*$ has Grothendieck--Neeman duality, then they prove the equivalence of the following: 
    \begin{enumerate}[label=(\roman*)]
        \item $\omega_f$ is invertible;
        \item the Wirthm\"uller isomorphism holds for all $Y \in \D$ in that there is an invertible object $\omega \in \D$ together with a natural isomorphism \[f_!(\omega^{-1} \otimes Y) \simeq f_*Y;\]
        \item there is an infinite chain of adjoints on both sides:
        \[\cdots \dashv f_! \dashv f^* \dashv f_* \dashv f^! \dashv f_\# \dashv \cdots\]
    \end{enumerate}

They use this to give an abstract viewpoint on a more general form of Grothendieck duality in terms of \emph{external dualising objects}: those objects $X$ for which $\iHom(-,X)$ gives an antiequivalence on a suitable full subcategory. Dwyer--Greenlees--Iyengar~\cite{DGI} have also given a general approach to duality phenomena through the notion of a Gorenstein map $R \to k$ of ring spectra. In this paper, we connect and generalise these two approaches to duality results in a general tt-framework. One key aspect is the use of proxy-smallness as a weakening of the assumption that $f_*$ preserves compact objects in the statement of Grothendieck--Neeman duality. 

\subsection*{Summary of main results} In short, the main results of this paper are the following, which we expand upon and motivate further throughout the rest of the introduction:
\begin{enumerate}
    \item We prove that Grothendieck duality still holds on a canonical subcategory of $\C$ if one relaxes the assumption that $f_*$ preserves compacts, to $f_*(\1_\D)$ being proxy-small, see \cref{thmx:Grothendieckduality}.
    \item Even without Grothendieck--Neeman duality, we show that the invertibility of $\omega_f$ provides a good framework for studying generalised Matlis-type duality statements, see \cref{thmx:external,thmx:dualising}.
    \item We demonstrate that the invertibility of $\omega_f$ provides a good notion of Gorenstein morphisms and duality in tt-geometry which extends the approach of Dwyer--Greenlees--Iyengar to more general settings, see \cref{thmx:duality}. 
\end{enumerate}

\subsection*{Grothendieck duality via proxy-smallness}
The condition that $f_*$ preserves compact objects is rather restrictive. For example, for a local Noetherian ring $R$ with residue field $k$, this condition is that $k$ is a compact $R$-module (i.e., $k$ has finite projective dimension), which holds if and only if $R$ is a regular local ring. Dwyer--Greenlees--Iyengar~\cite{DGI} introduced the notion of proxy-smallness as a weakening of compactness (recall that compact objects are sometimes also called small). In our context, an object $X$ is \emph{proxy-small} if there exists a set of compact objects $\W$ such that $X \in \Loctensor{\W}$ and $\W \subseteq \thicktensor{X}$. Informally, $\W$ is a witness for the fact that the tensor-triangular information of $X$ is seen by compact objects. Returning to the above example, for any local Noetherian ring $R$, the residue field $k$ is always proxy-small as an $R$-module, with the witness set $\W$ given by the Koszul complex on the generators of the maximal ideal $\m$. Therefore for any local Noetherian ring $R$, the residue field $k$ is proxy-small, but $k$ is compact if and only if $R$ is regular.


Throughout the paper, we will often assume that $f_*(\1_\D)$ is proxy-small as a replacement for $f_*$ preserving compacts. A geometric functor $f^*\colon \C \to \D$ whose right adjoint $f_*$ has this property is called \emph{proxy-small}. We lay some basic groundwork on the behaviour of such functors, for example, demonstrating their good base change properties (\cref{prop:changeofbase}). In particular, if $f^*\colon \C \to \D$ is a proxy-small geometric functor, then $f^*$ factors over the category $\Gammaf\C \coloneqq \Loctensor{f_*\1_\D}$ of torsion objects. Although $\C$ is rigidly-compactly generated, the torsion category $\Gammaf\C$ rarely is: the unit object is rigid (a.k.a., strongly dualisable) but not compact in general, so there are more rigid objects than compacts. As such it is of interest to classify the rigid objects in $\Gammaf\C$. Our first main result gives a characterisation of the rigid objects in $\Gammaf\C$ in terms of passage under $f^*$ to $\D$:
\begin{thmx}[{\cref{f*rigid}}]\label{thmx:f*rigid}
    Let $f^*\colon\C\to\D$ be a proxy-small geometric functor. Let $M\in\C$.
    The following are equivalent:
        \begin{enumerate}[label=(\arabic*)]
        \item the map $\rho_{M,N}\colon DM\otimes N\to\iHom(M,N)$ is an $f^*$-isomorphism for all $N \in \C$ (i.e., $f^*(\rho_{M,N})$ is an isomorphism for all $N \in \C$);
        \item $\Gammaf M\in\Gammaf\C$ is rigid;
        \item $\Lambdaf M\in\Lambdaf\C$ is rigid;
        \item $f^*M\in\D$ is rigid.
        \end{enumerate}
\end{thmx}

We note that this is related to the work of Benson--Iyengar--Krause--Pevtsova~\cite{BIKP} on local dualisability and we refer the reader to \cref{ex:BIKPRigid} for discussion.

We say that an object satisfying the equivalent conditions of \cref{thmx:f*rigid} is \emph{$f^*$-rigid}. Using the equivalent characterisations of the previous theorem, we show that if one relaxes the assumption that $f_*$ preserves compacts, to $f^*$ being proxy-small, then Grothendieck duality holds on the subcategory of $\C$ consisting of the $f^*$-rigid objects. This result is interesting in its own right, but is also a key input into our study of more general duality phenomena.
\begin{thmx}[{\cref{f*compactisos}}]\label{thmx:Grothendieckduality}
    Let $f^*\colon \C \to \D$ be a proxy-small geometric functor and let $M \in \C$. If $f^*M$ is rigid, then there is a natural isomorphism
    \[
    f^{!}X \otimes f^*M \simeq f^{!}(X \otimes M)
    \]
    for all $X \in \C$. In particular, if $f^*M$ is rigid there is an isomorphism
    \[
    \omega_f \otimes f^*M \simeq f^{!}M.
    \]
\end{thmx}
In addition to this, one also obtains various other isomorphisms which one would typically only expect to hold for rigid objects, rather than this larger class of \emph{$f^*$-rigid} objects, see \cref{f*compactclosed}. We also note that this provides a partial answer to a question of Sanders~\cite{SandersCompactness}, who showed that one can always force Grothendieck--Neeman duality to hold by taking a finite colocalisation of $\D$, and asks whether it is possible to instead localise $\C$. Our previous result shows that under the proxy-smallness assumption, we obtain Grothendieck duality for free on the thick $\otimes$-ideal of $f^*$-rigid objects.

\subsection*{Dualising objects and the Gorenstein property}
It is already noted in \cite{BDS} that $\omega_f$ can be invertible without Grothendieck--Neeman duality holding. In fact, this happens often:

\begin{ex*} If $f^*$ is extension of scalars along $R \to k$ for a local Noetherian ring $(R,\m,k)$, the functor $f^*$ has Grothendieck--Neeman duality if and only if $R$ is regular, and $\omega_f$ is invertible if and only if $R$ is Gorenstein. In particular, in this case Grothendieck--Neeman duality forces the invertibility of $\omega_f$ (since any regular ring is Gorenstein). However, if $R$ is not regular but is Gorenstein, then $\omega_f$ is invertible without Grothendieck--Neeman duality holding. \end{ex*}

This provides one motivation for undertaking a systematic study of the general behaviour encountered when $\omega_f$ is invertible without Grothendieck--Neeman duality holding, as a suitable form of the Gorenstein condition in tt-geometry. To this end, we say that a geometric functor $f^*\colon \C \to \D$ is \emph{Gorenstein} if $\omega_f = f^!(\1_\C)$ is invertible. This is consistent with the usual definitions in commutative algebra and algebraic geometry, and gives a slight generalisation of the definition of Dwyer--Greenlees--Iyengar~\cite{DGI} in the setting of commutative ring spectra, see \cref{rem:Gorenstein} for further details. In particular, one specific motivation for us was to develop a formal framework for studying Gorenstein \emph{equivariant} ring spectra in line with the non-equivariant setting pioneered in \cite{DGI}. We explain how this Gorenstein property often yields a Gorenstein duality statement later in the introduction.

Although $\omega_f$ plays a distinguished role in Grothendieck duality as demonstrated in \cref{thmx:Grothendieckduality}, from a categorical perspective, the above definition of Gorenstein admits a natural generalisation: we say that an object $M \in \C$ is \emph{Matlis dualising} if $f^!M$ is an invertible object. The terminology arises because of the following result, which shows that Matlis dualising objects give an abstract form of Matlis duality, specialising to the classical statement in local algebra. Moreover, the following can also be interpreted as a characterisation of Matlis dualising objects as a particularly well-behaved class of the external dualising objects studied in \cite{BDS}.

\begin{thmx}[{\cref{Mduality}, \cref{cor-Matlisdualvsextdual}}]\label{thmx:external}
    Let $f^*\colon \C \to \D$ be a geometric functor. If $M \in \C$ is Matlis dualising, then the functor $\iHom(-,M)\colon \C^\op \to \C$ restricts to an equivalence of categories
    \[\iHom(-,M)\colon \thicktensor{f_*(\D^c)}^\op \xrightarrow{~\simeq~} \thicktensor{f_*(\D^c)}.\]
    Moreover, if $\D$ is pure semisimple (see \cref{puresemisimple}) and $f_*$ is conservative, then the converse holds.
\end{thmx}

In the Gorenstein setting, we give a global characterisation of Matlis dualising objects. When specialised to the setting of chromatic homotopy theory, this recovers the approach of Dwyer--Greenlees--Iyengar to Gross--Hopkins duality~\cite{DGIGrossHopkins}. In light of \cref{thmx:external}, the following can also be viewed as a classification of external dualising objects for $\thicktensor{f_*(\D^c)}$ when $f_*$ is conservative.

\begin{thmx}[{\cref{thm:dualising}}]\label{thmx:dualising}
    Let $f^*\colon \C \to \D$ be a proxy-small geometric functor and assume that $\D$ is pure semisimple. If $f^*$ is Gorenstein, the following are equivalent for $M \in \C$:
    \begin{enumerate}
        \item $M$ is Matlis dualising (i.e., $f^{!}M$ is invertible in $\D$);
        \item $f^*M$ is invertible in $\D$;
        \item $\Gammaf M$ is invertible in $\Gammaf\C$;
        \item $\Lambdaf M$ is invertible in $\Lambdaf\C$.
    \end{enumerate}
\end{thmx}
\subsection*{Gorenstein duality and Matlis lifts}
Recall that in local algebra, the Gorenstein condition is equivalent to \emph{Gorenstein duality} for local cohomology: \[H^*_\m(R)\simeq\Sigma^{-\dim(R)}E(k).\]
For augmented algebra spectra, a generalisation of this property is explored by Dwyer--Greenlees--Iyengar~\cite{DGI}, and we generalise the augmented algebra setting to geometric functors which admit sections (\cref{defn:Gorduality}). 

An important subtlety that occurs away from local algebra is that although Gorenstein duality implies the Gorenstein condition, the converse need not hold in general. The failure of the converse is measured by an orientability condition described below, as demonstrated by the following result:
\begin{thmx}[\cref{dualityIffOrientable}]\label{thmx:duality}
    Let $f^*\colon\C\to\D$ be a Gorenstein proxy-small geometric functor with a geometric section $\eta^*$. Then $f^*$ has Gorenstein duality if and only if $\eta^!(\omega_f)$ is an orientable Matlis lift of $\omega_f$.
\end{thmx}
A \emph{Matlis lift} of an object $Y\in\D$ is an object $X\in\C$ such that $f^!X\simeq Y$. If in particular $Y=\omega_f$, by definition $\1_\C$ is a Matlis lift, and we say another Matlis lift $X$ of $\omega_f$ is \emph{orientable} if $\Gammaf X\simeq\Gammaf \1_\C$. 

As a consequence of \cref{thmx:duality}, one is interested in the uniqueness of Matlis lifts in the torsion category. In particular, we want to know which geometric functors are \emph{automatically orientable} in that all Matlis lifts of $\omega_f$ are orientable. This occurs in local algebra (that is, extension of scalars along $R\to k$ is automatically orientable for $(R,\m,k)$ a Noetherian local ring) and explains why the Gorenstein condition and Gorenstein duality coincide in this setting.

We describe two characterisations of Matlis lifts. Firstly, we extend the Morita theory of \cite{DGI}, observing that the data of a Matlis lift is encoded in a module structure over the endomorphism algebra $\Ef\coloneqq \iHom(f_*\1_\D,f_*\1_\D)$. This leads to a full classification of so-called `weak' Matlis lifts in the torsion category. The notion of \emph{weak} Matlis lifts arises from the fact that the right adjoint $f_*$ of $f^*$ need not be conservative in general, unlike in the setting of \cite{DGI}.
\begin{thmx}[\cref{prop:classifylifts}]
    Let $f^*$ be a proxy-small enhanced geometric functor, and let $Y\in\D$. There is a bijection between torsion weak Matlis lifts of $Y$ and right $\Ef$-module structures on $f_*Y$ that are effective (\cref{defn:effective}).
\end{thmx}

The other classification of Matlis lifts, specifically those of $\omega_f$, comes from \cref{thmx:dualising}. Provided $\D$ is pure semisimple and $f^*$ is Gorenstein, as a consequence of Grothendieck duality (\cref{thmx:Grothendieckduality}) we see that $X\in\C$ is a Matlis lift of $\omega_f$ if and only if $f^*X\simeq\1_\D$, and so orientability is controlled by the kernel of a group homomorphism:

\begin{thmx}[\cref{prop:AOandPic}] Let $f^*\colon\C\to\D$ be a proxy-small, Gorenstein, geometric functor, where $\D$ is pure semisimple. Then $f^*$ is automatically orientable if and only if the group homomorphism \[f^*\colon\Pic(\Gammaf\C)\to\Pic(\D)\]
is injective.
\end{thmx}

\subsection*{Applications and outlook}
We end by giving various examples and applications of the theory developed in this paper. In order to give new examples, we prove various ascent and descent results for Matlis dualising objects, Gorenstein geometric functors, proxy-smallness, and orientability (\cref{sec:descent}). 
Throughout the paper we illustrate the general theory with two key examples: local algebra and chromatic homotopy theory, giving new perspectives on various statements in these areas. However, we also investigate some more detailed examples as we now describe.

Given a graded commutative ring $R$, we consider the totalisation functor $T^*$ from the derived category of chain complexes of graded $R$-modules to the derived category of DG-$R$-modules. This is a geometric functor, and as an application of our descent results, we prove that a formal DGA $A$ is Gorenstein in our sense (and the sense of Dwyer--Greenlees--Iyengar) if and only if $H_*A$ is a Gorenstein ring, see \cref{FormalDGAGor}. Our framework also applies to show that a formal DGA is homotopically regular if and only if its homology is regular. As such, these demonstrate that for formal DGAs, there is good interplay between classic algebraic notions and their homotopical enhancements.

Our second application is to polynomial invariant theory, and gives a new perspective on Watanabe's theorem, in particular realising it as the shadow of an equivariant Gorenstein duality statement. Recall that Watanabe's theorem states that for a finite group $G$, if the determinant representation associated to a $kG$-module $V$ is trivial, then the invariant ring $k[V]^G$ is Gorenstein. We show that without the hypothesis on the determinant, one can realise $k[V]$ as being equivariantly Gorenstein (\cref{prop:ourWatanabe}). We then explain how one can recover the original theorem of Watanabe from this by providing a general statement about when fixed points preserve Gorenstein duality (\cref{prop:fixedpoints}). This can be viewed as a blueprint for establishing similar statements regarding the behaviour of fixed points and Gorenstein duality, with a particular view towards equivariant homotopy theory, and we intend to return to this in future work. In particular, this application suggests that this framework would permit us to view various Gorenstein duality statements, such as those of \cite{GreenleesrelGor, PolWilliamson} as the (geometric) fixed points of equivariant Gorenstein duality phenomena.

\subsection*{Acknowledgements}
We are grateful to John Greenlees for his interest and for many helpful conversations and suggestions. We would also like to thank Marek P\'asek for pointing out a simpler argument in \cref{rem:automaticallyrelative}.

TP is supported by the Warwick Mathematics Institute CDT, and is funded by EPSRC grant EP/W523793/1.
JW is supported by the project PRIMUS/23/SCI/006 from Charles University and by the Charles University Research Centre
program No.~UNCE/24/SCI/022. The authors would like to thank the Isaac Newton Institute for Mathematical Sciences, Cambridge, for support and hospitality during the programme `Equivariant homotopy theory in context', where work on this paper was undertaken. This work was supported by EPSRC grant EP/Z000580/1.

\section{Tensor-triangular background}\label{sec:ttbackground}
In this section, we recall some fundamentals regarding tensor-triangulated categories.

\begin{chunk}
    A tensor-triangulated category $\C$ (henceforth tt-category) is a triangulated category together with a closed symmetric monoidal structure which is compatible with the triangulation (see \cite[\S A.2]{HPS}). We write $- \otimes -$ for the tensor product, $\1_\C$ for the tensor unit, and $\iHom(-,-)$ for the internal hom. We write $D = \iHom(-,\1_\C)$ for the functional duality functor. 
\end{chunk}

\begin{chunk}\label{rigiditychunk}
    Recall that an object $X \in \C$ is \emph{rigid} if the canonical map 
    \[DX \otimes Y \to \iHom(X,Y)\]
    is an isomorphism for all $Y \in \C$. (In fact, it suffices to just check for $Y = X$.) Note that in the literature, rigid objects are often called strongly dualisable instead. We will not use such terminology to avoid confusion with dualising objects in the sense of this paper. Rigid objects are \emph{reflexive} in the sense that the canonical map $X \to D^2X$ is an isomorphism. It follows that if $X$ is rigid, then the canonical map \[DY \otimes X \to \iHom(Y,X)\] is also an isomorphism for all $Y \in \C$. We say that $\C$ is \emph{rigidly-compactly generated} if $\C$ is compactly generated, and the subcategories of  rigid objects and compact objects coincide. Note that this implies that the tensor unit $\1_\C$ is compact. We write $\C^c$ for the full subcategory of compact objects. For the majority of the paper we work in the setting of rigidly-compactly generated tt-categories, but occasionally we will need to work with tt-categories which are not.
\end{chunk}

\begin{chunk}
    Let $\C$ be a compactly generated tt-category. A full subcategory $\mathscr{S}$ of $\C$ is \emph{thick} if it is closed under retracts and is triangulated. Moreover it is said to be \emph{localising} if it is thick and closed under coproducts. A \emph{thick $\otimes$-ideal} of $\C$ is a thick subcategory $\mathscr{S}$ which is closed under tensoring with rigid objects of $\C$. A \emph{localising $\otimes$-ideal} of $\C$ is a localising subcategory $\mathscr{S}$ which is closed under tensoring with arbitrary objects of $\C$. We emphasise that in our conventions, thick $\otimes$-ideals of $\C$ need only be closed under tensoring with rigid objects. 
\end{chunk}
\begin{chunk}
    Let $\C$ be a compactly generated tt-category. Given a collection of objects $\mathscr{X}$ of $\C$, we write $\thick(\mathscr{X})$, $\thicktensor{\mathscr{X}}$, $\Loc(\mathscr{X})$, and $\Loctensor{\mathscr{X}}$ for the smallest thick subcategory, thick $\otimes$-ideal, localising subcategory, and localising $\otimes$-ideal of $\C$ which contains $\mathscr{X}$ respectively. We note that if $\C$ is rigidly-compactly generated, then $\thicktensor{\scr{X}} = \thick(\scr{X} \otimes \C^c)$ where $\scr{X} \otimes \C^c = \{X \otimes C \mid X \in \scr{X},~C \in \C^c\}$, and similarly $\Loctensor{\scr{X}} = \Loc(\scr{X} \otimes \C^c)$.
\end{chunk}

\begin{chunk}\label{chunk:geometric}
    Recall that a functor $f^*\colon \C \to \D$ between rigidly-compacted generated tt-categories is \emph{geometric} if it is strong symmetric monoidal, triangulated, and preserves coproducts. Brown representability then implies that $f^*$ has a right adjoint $f_*$. Since $f^*$ is strong monoidal, it preserves rigid(=compact) objects, and hence $f_*$ preserves coproducts. As such, $f_*$ also admits a right adjoint $f^{!}$. Diagrammatically we have
    \[\begin{tikzcd}[column sep=2cm]
        \C \ar[r, "f^*", yshift=2.5mm] \ar[r, "f^{!}"', yshift=-2.5mm] & \D. \ar[l, "f_*" description]
    \end{tikzcd}\]
    We refer the reader to \cref{exs:geometric} for a discussion of several examples of geometric functors arising across various areas. 
\end{chunk}

We now recall some natural isomorphisms which play an important role and use this as an opportunity to set some notation. We refer the reader to \cite[Proposition 2.15]{BDS} for further details.
\begin{chunk}\label{constructionprojformula}
    The adjunction $(f^*,f_*)$ satisfies the projection formula: there is a natural isomorphism \begin{equation}\label{projformula}
        p_{X,Y}\colon X \otimes f_*Y \xrightarrow{\sim} f_*(f^*X \otimes Y)
    \end{equation}
    for all $X \in \C$ and $Y \in \D$. This is constructed as the adjunct to the natural map
    \[f^*(X \otimes f_*Y) \simeq f^*X \otimes f^*f_*Y \xrightarrow{f^*X \otimes \varepsilon} f^*X \otimes Y\] where $\varepsilon$ is the counit of the $(f^*,f_*)$ adjunction. Fixing the $X$ variable in the projection formula \cref{projformula} and taking right adjoints, it follows that there is a natural isomorphism
    \begin{equation}\label{f1ofhom}
        f^{!}\iHom(X,Y) \simeq \iHom(f^*X, f^{!}Y)
    \end{equation}
    for all $X,Y \in \C$.
    On the other hand, fixing the $Y$ variable in the projection formula and taking right adjoints yields the natural isomorphism
    \begin{equation}\label{internaladjunction2}
    \iHom(f_*Y, X) \simeq f_*\iHom(Y, f^{!}X)
    \end{equation}
    for all $X \in \C$ and $Y \in \D$.
\end{chunk}

\begin{chunk}\label{constructioninternaladjunction}
    The adjunction $(f^*,f_*)$ also has an internal form as we now recall. For $M,N \in \C$, there is a natural map $\alpha_{M,N}\colon f^*\iHom(M,N) \to \iHom(f^*M,f^*N)$ which is adjoint to the composite \[f^*M \otimes f^*\iHom(M,N) \simeq f^*(M \otimes \iHom(M,N)) \xrightarrow{f^*\mrm{ev}_{M,N}} f^*N.\] Using this, one constructs a natural isomorphism
    \begin{equation}\label{internaladjunction}
    a_{X,Y}\colon\iHom(X,f_*Y) \xrightarrow{\sim} f_*\iHom(f^*X,Y)
    \end{equation}
    for all $X \in \C$ and $Y \in \D$, as the adjunct to the natural map
    \[f^*\iHom(X,f_*Y) \xrightarrow{\alpha_{X,f_*Y}} \iHom(f^*X, f^*f_*Y) \xrightarrow{\iHom(f^*X,\varepsilon)} \iHom(f^*X,Y).\]
\end{chunk}

\begin{chunk}\label{chunk:enhanced}
    Although most of the results in this paper do not require enhancements, it will occasionally be convenient to use them. This will allow us to avoid separability assumptions when considering modules over commutative algebra objects; recall that the category of modules over a commutative algebra in a tt-category need not inherit a triangulated structure in general unless the algebra is separable, see~\cite{Balmerseparable}.
    An \emph{enhanced tt-category} is a presentably symmetric monoidal stable $\infty$-category whose tensor product commutes with colimits separately in each variable.
    The homotopy category of such an enhanced tt-category is a tt-category with set-indexed coproducts. 
    An enhanced tt-category $\C$ is said to be rigidly-compactly generated if its homotopy category $h\C$ is rigidly-compactly generated. 
    Finally, an enhanced geometric functor $f^*\colon \C \to \D$ is a colimit-preserving symmetric monoidal functor between rigidly-compactly generated enhanced tt-categories. Note that the induced functor on homotopy categories $hf^*\colon h\C \to h\D$ is then a geometric functor in the sense of \cref{chunk:geometric}.
\end{chunk}

\begin{exs}\label{exs:geometric}
We record several examples of geometric functors of interest arising from a variety of fields which we will utilise throughout the paper:
\begin{enumerate}[label=(\roman*)]
    \item Given a commutative ring $R$, the derived category $\msf{D}(R)$ is a rigidly-compactly generated tt-category, with rigid-compact objects the perfect complexes. Given a map of commutative rings $f\colon R \to S$, the (derived) extension of scalars functor $f^* = S \otimes_R -\colon \msf{D}(R) \to \msf{D}(S)$ is a geometric functor, with $f_*$ given by restriction of scalars and $f^! = \Hom_R(S,-)$ given by the (derived) coextension of scalars.
    \item The previous example admits vast generalisation: one could instead take $f$ to be a map of (highly structured) commutative ring spectra, or of commutative equivariant ring spectra, and so on. More precisely, given any rigidly-compactly generated enhanced tt-category $\C$ and any commutative algebra $R$ in $\C$, the category $\Mod{\C}(R)$ of $R$-modules in $\C$, is again a rigidly-compactly generated enhanced tt-category with the subcategory of rigid-compact objects given by $\thick(R \otimes \C^c)$. If $R \to S$ is a map of commutative algebra objects in $\C$, then the extension of scalars functor $f^* = S \otimes_R -\colon \Mod{\C}(R) \to \Mod{\C}(S)$ is an enhanced geometric functor.
    \item Given a quasi-compact quasi-separated scheme $X$, the derived category $\msf{D}_\mrm{qc}(X)$ of $\mathcal{O}_X$-modules with quasi-coherent cohomology is a rigidly-compactly generated tt-category with rigid-compact objects the perfect complexes. Given a map $f\colon X \to Y$ of quasi-compact quasi-separated schemes, the (derived) inverse image functor $f^*\colon \msf{D}_\mrm{qc}(Y) \to \msf{D}_\mrm{qc}(X)$ is a geometric functor, with $f_*$ the (derived) pushforward, and $f^!$ the twisted inverse image functor.
    \item Let $G$ be a compact Lie group. The homotopy category of genuine equivariant $G$-spectra, denoted $\Sp_G$, is a rigidly-compactly generated tt-category with rigid-compact objects the finite $G$-spectra, i.e., those in the thick subcategory generated by the orbits $G/H_+$ as $H$ runs over all closed subgroups of $G$. Given a closed subgroup $H$ of $G$, the restriction functor $f^* = \mrm{res}^G_H\colon \Sp_G \to \Sp_H$ is a geometric functor. The right adjoint $f_*$ is the coinduction functor, and the further right adjoint $f^!$ has the form $S^L \wedge \mrm{res}^G_H(-)$ by the Wirthm\"uller isomorphism, where $L$ denotes the $H$-representation given by the tangent space to $G/H$ at the identity coset $eH$. 
    \item Continuing the previous example, if $G$ is a compact Lie group and $N$ is a normal subgroup of $G$, then the inflation functor $f^* = \mrm{inf}^G_{G/N}\colon \Sp_{G/N} \to \Sp_G$ is a geometric functor. The right adjoint $f_* = (-)^N$ is the categorical $N$-fixed points functor, but the further right adjoint $f^!$ is harder to describe. We refer the reader to \cite[\S 2.1 and \S 2.5]{HKS} and \cite[Remark 8.6]{SandersCompactness} for further discussion about $f^!$. 
    \item Let $G$ be a finite group and $k$ be a field of characteristic dividing the order of $G$. The stable module category $\mrm{StMod}(kG)$ is a rigidly-compactly generated tt-category with compact-rigid objects given by the finite dimensional modules. Given a subgroup $H$ of $G$, the restriction functor $f^* = \mrm{res}^G_H\colon \mrm{StMod}(kG) \to \mrm{StMod}(kH)$ is a geometric functor, with $f_*$ given by the coinduction $\Hom_{kH}(kG,-)$ (which is equivalent to the induction $kG \otimes_{kH} -$), and therefore $f^!$ is again the restriction functor.
\end{enumerate}
\end{exs}

\section{Proxy-small geometric functors}
In this section, we introduce the setup in which the paper operates. We define proxy-smallness in tt-categories and establish its key properties, in particular, its compatibility with geometric functors.

\subsection{Relative proxy-smallness, torsion, and completion}
We first recall the framework of torsion and completion in tt-categories in the sense of Greenlees~\cite{Greenleestate} and Hovey--Palmieri--Strickland~\cite{HPS}. In loc. cit., one builds torsion and completion functors from a set of compact objects; instead, we state this in the setting of proxy-small objects in the sense of~\cite{DGI}, although the proofs remain unchanged.

\begin{defn}\label{defn:proxysmall}
    Let $\C$ be a rigidly-compactly generated tt-category. We say that a set of objects $\mc{K}$ of $\C$ is \emph{proxy-small} if there exists a set of compact objects $\mc{W}$ such that $\mc{W} \subseteq \thicktensor{\mc{K}}$ and  $\mc{K} \subseteq \Loctensor{\mc{W}}$. We call $\mc{W}$ the \emph{witness} for the proxy-smallness of $\mc{K}$, and note this implies $\Loctensor{\mc{K}} = \Loctensor{\mc{W}}$.
\end{defn}

\begin{rem}\label{rem:localringps}
   Proxy-small objects are a weakening of the notion of compact objects (recall that some authors call compact objects small). For instance, given a commutative Noetherian local ring $(R,\m,k)$, the Koszul complex on the generators of $\m$ is a witness for the residue field $k = R/\m$ being proxy-small. In contrast, recall that the residue field $k$ is compact if and only if $R$ is regular. 
\end{rem}

\begin{rem}\label{rem:universalwitness}
    A priori, the definition of proxy-smallness depends on the choice of \emph{some} witness, but one can give an equivalent definition in terms which are intrinsic to $\K$. There is a canonical choice of witness - namely the set $\W_\K=\thicktensor{\K}\cap\C^c$. This is by definition a set of compact objects contained in $\thicktensor{\K}$, and hence, $\K$ is proxy-small if and only if $\Loctensor{\K}=\Loctensor{\thicktensor{\K}\cap\C^c}$. Despite the definition being independent of the choice of witness, we note that choosing a witness is often convenient. 
\end{rem}

\begin{chunk}
    Fix a proxy-small set $\mc{K}$ of objects in $\C$ and define $\Gamma_\K\C = \Loctensor{\mc{K}}$. The inclusion of the localising $\otimes$-ideal $\Loctensor{\mc{K}}$ generated by $\mc{K}$ (equivalently, $\mc{W}$) into $\C$ admits a right adjoint $\Gamma_\mc{K}$:
    \[\begin{tikzcd}
        \Gamma_\K\C \ar[r, hookrightarrow, yshift=1mm] & \C. \ar[l, yshift=-1mm, "\Gamma_{\mc{K}}"]
    \end{tikzcd}\]
    We call $\Gamma_{\mc{K}}$ the $\mc{K}$-torsion functor and say that an object $X \in \C$ is $\mc{K}$-torsion if the counit map $\Gamma_\mc{K}X \to X$ is an equivalence. (Note that we view $\Gamma_{\mc{K}}$ as an endofunctor on $\C$ by postcomposing with the inclusion.) 
\end{chunk}

\begin{chunk}
    Given $\mc{S} \subseteq \C$, we write \[\mc{S}^\perp \coloneqq \{X \in \C \mid \Hom(S,X) \simeq 0 \text{ for all $S \in \mc{S}$}\}\] and define $L_\K\C \coloneqq  \left(\Gamma_\K\C\right)^\perp$ and $\Lambda_\K\C 
    \coloneqq \left(L_\K\C\right)^\perp$. We obtain adjunctions
    \[\begin{tikzcd}L_\K\C \ar[r, hookrightarrow, yshift=-1mm] & \C \ar[l, yshift=1mm, "L_\mc{K}"']\end{tikzcd} \quad \text{and} \quad \begin{tikzcd}\Lambda_\K\C \ar[r, hookrightarrow, yshift=-1mm] & \C \ar[l, yshift=1mm, "\Lambda_\mc{K}"']\end{tikzcd}\]
    and say that an object $X \in \C$ is $\mc{K}$-local (resp., $\mc{K}$-complete) if the unit map $X \to L_\mc{K}X$ (resp., $X \to \Lambda_\mc{K}X$) is an equivalence. Moreover, the inclusion $L_\K\C \hookrightarrow \C$ also has a right adjoint $V_\K$.
\end{chunk}

These functors satisfy the following properties:
\begin{thm}\label{thm:localduality}
    Let $\C$ be a rigidly-compactly generated tt-category and $\mc{K}$ be a proxy-small set of objects of $\C$. 
    \begin{enumerate}
        \item There are natural triangles \[\Gamma_\mc{K}X \to X \to L_\mc{K}X \quad \text{and} \quad V_\K X \to X \to \Lambda_\K X\] for all $X \in \C$.
        \item The functors $\Gamma_\mc{K}$ and $L_\mc{K}$ are smashing; i.e., there are natural isomorphisms $\Gamma_\mc{K}X \simeq \Gamma_\mc{K}\1 \otimes X$ and $L_\mc{K}X \simeq L_\mc{K}\1 \otimes X$ for all $X \in \C$.
        \item Viewed as endofunctors on $\C$, the $\mc{K}$-torsion functor is internally left adjoint to the $\mc{K}$-completion functor; i.e., there is a natural isomorphism \[\iHom(\Gamma_\mc{K}X, Y) \simeq \iHom(X, \Lambda_\mc{K}Y)\] for all $X,Y \in \C$. In particular, $\Lambda_\K \simeq \iHom(\Gamma_\mc{K}\1,-)$.
        \item (\emph{The MGM equivalence)} The natural maps $\Gamma_\mc{K} \to \Gamma_\mc{K}\Lambda_\mc{K}$ and $\Lambda_\mc{K}\Gamma_\mc{K} \to \Lambda_\mc{K}$ are isomorphisms. Therefore, \[\begin{tikzcd}\Gamma_\K\C \ar[r, yshift=-1mm, "\Lambda_\mc{K}"'] & \Lambda_\K\C \ar[l, yshift=1mm, "\Gamma_\mc{K}"'] \end{tikzcd}\] is an equivalence of categories.
    \end{enumerate}
\end{thm}
\begin{proof}
    This is the content of \cite[Theorem 3.3.5]{HPS}, noting that the proxy-small assumption implies that there is a set of compacts $\mc{W}$ such that $\Loctensor{\mc{K}} = \Loctensor{\mc{W}}$ so that we may work with the set of compacts $\mc{W}$ instead.
\end{proof}

\begin{chunk}\label{monoidalcompletetorsion}
    Being the local objects of a Bousfield localisation, the category $\Lambda_\K\C$ of $\mc{K}$-complete objects is closed symmetric monoidal, with tensor product $X \widehat{\otimes} Y = \Lambda_\K(X \otimes Y)$, internal hom the same as that in $\C$, and tensor unit $\Lambda_\K\1$, see \cite[Theorem 3.5.2]{HPS}. By the MGM-equivalence, the category of $\mc{K}$-torsion objects $\Gamma_\K\C$ inherits a closed symmetric monoidal structure from $\Lambda_\K\C$, with tensor product the same as in $\C$, internal hom $\Gamma_\mc{K}\iHom(-,-)$, and unit object $\Gamma_\mc{K}\1$. Therefore $\Gammaf\colon \C \to \Gammaf\C$ and $\Lambdaf\colon \C \to \Lambdaf\C$ are closed, strong symmetric monoidal functors. Moreover, note that these make $\Gamma_\K\colon \Lambda_\K\C \to \Gamma_\K\C$ and $\Lambda_\K\colon \Gamma_\K\C \to \Lambda_\K\C$ into strong monoidal equivalences.
\end{chunk}

\begin{rem}
    It is important to note that $\Gamma_\K\C$ and $\Lambda_\K\C$ are both compactly generated tt-categories, but not rigidly so. For example, $\Gamma_\K\1 \in \Gamma_\K\C$ is rigid (since it is the tensor unit), but is not compact in general. This is in contrast to $L_\K\C$ which is rigidly-compactly generated since $L\colon \C \to L_\K\C$ is a smashing localisation and hence preserves compact objects.
\end{rem}

In order to state the behaviour of torsion and completion under geometric functors and their adjoints, we introduce the following definition which is inspired by \cite[Lemma 2.22(2)]{BCHV}. Recall from \cref{chunk:geometric} that a geometric functor $f^*\colon \C \to \D$ fits into an adjoint triple $f^* \dashv f_* \dashv f^!$.
\begin{defn}
    Let $f^*\colon \C \to \D$ be a geometric functor, and $\K$ be a set of objects of $\D$. We say that $\K$ is \emph{proxy-small relative to $f^*$} if there exists a set of compact objects $\W$ of $\C$ such that: 
    \begin{enumerate}
        \item $f_*\K$ is proxy-small in $\C$ with witness $\W$, and,
        \item $\K$ is proxy-small in $\D$ with witness $f^*\W$.
    \end{enumerate}
\end{defn}

\begin{rem}\label{rem:psrelloc}
    If $\K$ is proxy-small relative to $f^*$, then $\Loctensor{\mc{K}} = \Loctensor{f^*f_*\mc{K}}$ by applying $f^*$ to $\Loctensor{\mc{W}} = \Loctensor{f_*\mc{K}}$.
\end{rem}

\begin{ex}\label{ex:normalisation} 
    Let $S \to R \to k$ be maps of commutative algebras in an enhanced rigidly-compactly generated tt-category. Suppose that $S \to R$ is a \emph{normalisation}, meaning that $k$ and $R$ are both compact $S$-modules. Write $f^*$ for extension of scalars along $S \to R$. It is straightforward to check from the definition that $k$ is proxy-small as an $R$-module with witness $R \otimes_S k$. In particular, $k$ is proxy-small relative to $f^*$, taking $\W=\{k\} \subseteq \Mod{}(S)^c$.
\end{ex}

In general, we do not know whether $\K$ being proxy-small relative to $f^*$ is independent of the choice of witness $\W$. However, if $\K \subseteq \D^c$, then it is independent of this choice:
\begin{lem}\label{lem:independentofwitness}
    Let $f^*\colon \C \to \D$ be a geometric functor, and $\K$ be a set of objects of $\D^c$ such that $f_*\K$ is proxy-small in $\C$. Then $\K$ is proxy-small relative to $f^*$ if and only if $\thicktensor{\K} = \Loctensor{f^*f_*\K} \cap \D^c$.
\end{lem}
\begin{proof}
    For the forward direction, if $\K$ is proxy-small relative to $f^*$, then we have $\Loctensor{f^*f_*\K} \cap \D^c = \Loctensor{\K} \cap \D^c$ by \cref{rem:psrelloc}. Since $\K \subseteq \D^c$, we have $\Loctensor{\K} \cap \D^c = \thicktensor{\K}$ by the Thomason localisation theorem (e.g., see \cite[Theorem 7.1]{Greenleesbalmer}), so the forward implication holds. For the reverse implication, write $\W \subseteq \C^c$ for a witness for the proxy-smallness of $f_*\K$. We have $\Loctensor{f^*f_*\K} \cap \D^c = \Loctensor{f^*\W} \cap \D^c = \thicktensor{f^*\W}$ again using the Thomason localisation theorem, so $\thicktensor{f^*\W} = \thicktensor{\K}$. From this, one easily checks that $\K$ is proxy-small relative to $f^*$.
\end{proof}

We now show that relative proxy-smallness ensures that the torsion, localisation, and completion functors behave well under geometric functors and their adjoints.
There are various related results in the literature, for example~\cite[Lemma 3.1]{Shamir}, but we have been unable to locate the form that we require as given below.
\begin{prop}\label{prop:changeofbase}
    Let $f^*\colon \C \to \D$ be a proxy-small geometric functor, and $\K$ be a set of objects of $\D$ which is proxy-small relative to $f^*$. Then there are natural isomorphisms:
    \begin{enumerate}
    \item $f_*\Gamma_{\K}Y \simeq \Gamma_{f_*\K}f_*Y$ for all $Y \in \D$;
    \item $\Gamma_{\K}f^*X \simeq f^*\Gamma_{f_*\K}X$ for all $X \in \C$;
    \item $\Lambda_{\K}f^{!}X \simeq f^{!}\Lambda_{f_*\K}X$ for all $X \in \C$;
    \item $f_*\Lambda_{\K}Y \simeq \Lambda_{f_*\K}f_*Y$ for all $Y \in \D$.
    \end{enumerate}
\end{prop}
\begin{proof}
    For (1), we show that $f_*\Gamma_{\K}X$ has the universal property of the $f_*\K$-torsion of $f_*X$. This amounts to showing that:
    \begin{enumerate}
        \item[(i)] $f_*\Gamma_{\K}X \in \Loctensor{f_*\K}$;
        \item[(ii)] the canonical map $\Hom(Z,f_*\Gamma_{\K}X) \to \Hom(Z,f_*X)$ is an equivalence for all $Z \in \Loctensor{f_*\K}$.
    \end{enumerate}
    For (i), we have $\Gamma_{\K}X \in \Loctensor{\K} = \Loctensor{f^*f_*\K} = \Loc(\D \otimes f^*f_*\K)$ by relative proxy-smallness (\cref{rem:psrelloc}). Therefore \[f_*\Gamma_{\K}X \in \Loc(f_*(\D \otimes f^*f_*\K)) = \Loc(f_*(\D) \otimes f_*(\K)) \subseteq \Loctensor{f_*\K}\] where the equality holds by the projection formula \cref{projformula}.

    Now for (ii), let $Z \in \Loctensor{f_*\K}$. We may identify the map
    \[\Hom(Z, f_*\Gamma_\K X) \to \Hom(Z, f_*X)\] with \[\Hom(f^*Z, \Gamma_\K X) \to \Hom(f^*Z,X)\] by adjunction. Since $Z \in \Loctensor{f_*\K}$, we have $f^*Z \in \Loctensor{f^*f_*\K} = \Loctensor{\K}$ by relative proxy-smallness (\cref{rem:psrelloc}). Therefore the latter map, and hence the former is an isomorphism as required, which finishes the proof of (1).

     For (2), as above we need to show that $f^*\Gamma_{f_*\K}Y$ has the universal property of $\Gamma_\K f^*Y$; i.e., that $f^*\Gamma_{f_*\K}Y \in \Loctensor{\K}$ and that for the canonical map $\alpha\colon f^*\Gamma_{f_*\K}Y \to f^*Y$, we have $\Hom(Z,\alpha)$ an isomorphism for all $Z \in \Loctensor{\K}$. Firstly, we have $f^*\Gamma_{f_*\K}Y \in \Loctensor{f^*f_*\K} = \Loctensor{\K}$ by relative proxy-smallness. Since $\K$ is proxy-small relative to $f^*$, there is a witness set $\W \subseteq \C^c$ for the proxy-smallness of $f_*\K$ such that $f^*\W$ is a witness for the proxy-smallness of $\K$. In particular, $\Loctensor{f^*\W} = \Loctensor{\K}$. Now consider the map $\alpha\colon f^*\Gamma_{f_*\K}Y \to f^*Y$. By a localising subcategory argument, $\Hom(Z,\alpha)$ is an isomorphism for all $Z \in \Loctensor{\K}$ if and only if it is an isomorphism for $Z = C \otimes f^*W$ for all $W \in \W$ and $C \in \D^c$. This is in turn equivalent to $\iHom(f^*W,\alpha)$ being an isomorphism for all $W \in \W$. Since the elements of $f^*\W$ are rigid, this is equivalent to $f^*W \otimes \alpha$ being an isomorphism for all $W \in \W$. This then follows from the strong monoidality of $f^*$ together with the fact that elements of $\W$ are $f_*\K$-torsion.

    (3) and (4) are the right adjoints to (1) and (2) respectively so follow immediately.
\end{proof}

\begin{rem}\label{rem:changeofbaseL}
    One can also deduce the behaviour of $L_\K$ and $L_{f_*\K}$ with respect to $f^*$ and $f_*$ from (1) and (2) in the previous result by using the triangle $\Gamma_\K X \to X \to L_\K X$.
\end{rem}

\subsection{Proxy-small geometric functors}
In this subsection we consider the special case of relative proxy-smallness for a geometric functor $f^*\colon \C \to \D$ when $\K = \{\1_\D\}$.  This is the setting which most of the paper will operate under, so we introduce simplified terminology for this case. Recall from \cref{chunk:geometric} that a geometric functor $f^*\colon \C \to \D$ fits into an adjoint triple $f^* \dashv f_* \dashv f^!$.

\begin{defn}\label{defn:f*proxysmall}
    Let $f^*\colon \C \to \D$ be a geometric functor. We say that $f^*$ is \emph{proxy-small} if $f_*\1_\D$ is proxy-small. We note that this is in fact equivalent to $\1_\D$ being proxy-small relative to $f^*$, see \cref{rem:automaticallyrelative} below.
\end{defn}

\begin{rem}\label{rem:automaticallyrelative}
    For any geometric functor $f^*\colon \C \to \D$ we have $\Loctensor{f^*f_*\1_\D} = \D$. Indeed, write $\eta$ and $\varepsilon$ for the unit and counit of the adjunction $(f^*,f_*)$. Since $f^*\1_\C \simeq \1_\D$, the triangle identities show that the composite 
    \[\1_\D \xrightarrow{\alpha} f^*f_*\1_\D \xrightarrow{\varepsilon} \1_\D\] is the identity where $\alpha = f^*(\eta)$. Therefore $\1_\D$ is a retract of $f^*f_*\1_\D$ proving that $\Loctensor{f^*f_*\1_\D} = \D$. As a consequence of this and \cref{lem:independentofwitness} we see that $f_*\1_\D$ is proxy-small if and only if $\1_\D$ is proxy-small relative to $f^*$. This means that when $f^*$ is a proxy-small geometric functor, we may use \cref{prop:changeofbase} with $\K = \{\1_\D\}$.
\end{rem}

\begin{rem}\label{rem:calg}
    If $R \to S$ is a map of commutative algebras in some enhanced rigidly-compactly generated tt-category $\C$, then we abuse terminology and say that $R \to S$ is proxy-small if the geometric functor $f^* = S \otimes_R -\colon \Mod{\C}(R) \to \Mod{\C}(S)$ is proxy-small. In the special case when $\C$ is the category of spectra, this is consistent with the terminology in \cite[Definition 4.14]{DGI}.
\end{rem}

The following is the key feature of proxy-small geometric functors. In order to state it, we set some notation that we will use throughout the paper.
\begin{nota}\label{nota:Gammaf}
Let $f^*\colon \C \to \D$ be a proxy-small geometric functor. We write $\Gammaf$ and $\Lambdaf$ for the torsion and completion functors associated to $f_*\1_\D$ in the sense of \cref{thm:localduality}. By \cref{prop:changeofbase}, note that $f^*$ factors as $f^*\Gammaf$, and similarly, $f^! \simeq f^! \Lambdaf$.
\end{nota}

\begin{prop}\label{Wequivalences}
    Let $f^*\colon \C \to \D$ be a proxy-small geometric functor and write $\W \subseteq \C^c$ for a witness. Let $\alpha$ be a map in $\C$. Then the following are equivalent:
    \begin{enumerate}
        \item $W \otimes \alpha$ is an isomorphism for all $W \in \mc{W}$;
        \item $\Gammaf\alpha$ is an isomorphism;
        \item $\Lambdaf\alpha$ is an isomorphism;
        \item $f^*\alpha$ is an isomorphism;
        \item $f^{!}\alpha$ is an isomorphism.
    \end{enumerate}
\end{prop}
\begin{proof}
    That (1) and (2) are equivalent follows from the facts that $\Gammaf$ is smashing and elements of $\W$ are $f_*\1_\D$-torsion. The equivalence between (2) and (3) is then immediate from the MGM-equivalence (\cref{thm:localduality}). 
    For (2) $\Rightarrow$ (4), if $\Gammaf\alpha$ is an isomorphism, then $f^*\alpha \simeq f^*\Gammaf\alpha$ is an isomorphism by \cref{prop:changeofbase}. Similarly, (3) $\Rightarrow$ (5) holds since we have $f^{!}\alpha \simeq f^{!}\Lambdaf\alpha$ by \cref{prop:changeofbase}. 
    
    To see that (4) $\Rightarrow$ (1), we have $f_*\1 _\D \otimes \alpha \simeq f_*f^*\alpha$ by the projection formula \cref{projformula}. As $\mc{W} \subseteq \thicktensor{f_*\1_\D}$, by a thick subcategory argument it then follows that $W \otimes \alpha$ is an isomorphism for all $W \in \mc{W}$.
    Finally, for (5) $\Rightarrow$ (1), we have 
    \[\Hom(X, \iHom(f_*\1_\D, \alpha)) \simeq \Hom(X, f_*f^!\alpha) \simeq \Hom(f^*X, f^{!}\alpha)\] for any $X \in \C$, by \cref{internaladjunction2} and adjunction. As such, $\iHom(f_*\1_\D, \alpha)$ is an isomorphism. Since $\W \subseteq \thicktensor{f_*\1_\D}$ we see that $\iHom(W,\alpha)$ is an isomorphism for each $W \in \W$. Now since each $W \in \W$ is rigid, $W$ is a retract of $W \otimes DW \otimes W$. Therefore $W \otimes \alpha$ is a retract of $W \otimes \iHom(W,\alpha) \otimes W$, and hence is an isomorphism.
\end{proof}

\begin{rem}
    We say that a map $\alpha$ in $\C$ is an \emph{$f^*$-isomorphism} if $f^*\alpha$ is an isomorphism, and similarly for $f^!$, $\Gammaf$, and $\Lambdaf$.
\end{rem}

\subsection{Conservativity}
We end this section with some observations about conservativity which we will need later on. Let $f^*\colon \C \to \D$ be a geometric functor. Recall that $f_*\colon \D \to \C$ is conservative if for any $Y \in \D$, $f_*Y \simeq 0$ implies $Y \simeq 0$. In the following we make use of $\infty$-categorical enhancements in order to consider the category of modules over a commutative algebra object. Recall our conventions regarding enhancements from~\cref{chunk:enhanced}.
\begin{prop}\label{cor:conservative}
    Let $f^*\colon \C \to \D$ be a geometric functor. The following are equivalent:
    \begin{enumerate}
        \item $f_*$ is conservative;
        \item $\thick(f^*(\C^c)) = \D^c$.
    \end{enumerate}
    If moreover $f^*$ is an enhanced geometric functor, these are also equivalent to:
    \begin{enumerate}
        \item[(3)] there is a symmetric monoidal equivalence $\overline{f}^*\colon \Mod{\C}(f_*\1_\D) \xrightarrow{\sim} \D$ making the diagram 
        \[ 
            \begin{tikzcd}[row sep=1cm]
            \Mod{\C}(f_*\1_\D) \ar[r, "\overline{f}^*"] & \D \\
            \C \ar[u, "F"] \ar[ur, "f^*"'] &
            \end{tikzcd}    
        \]
        commute, where $F = f_*\1_\D \otimes -$ denotes the free functor.
    \end{enumerate}
\end{prop}
\begin{proof}
    (1) $\Leftrightarrow$ (2) follows from a simple compact generation argument. Now suppose that $f^*$ is an enhanced geometric functor. The implication (1) $\Rightarrow$ (3) follows from \cite[Proposition 5.29]{MNN}. Finally for (3) $\Rightarrow$ (1) we have $f_* = U \circ \overline{f}_*$ where $\overline{f}_*$ denotes the inverse to $\overline{f}^*$ and $U$ is the forgetful functor $\Mod{\C}(f_*\1_\D) \to \C$ which is right adjoint to $F$. Since $U$ is conservative and $\overline{f}_*$ is an equivalence the claim follows.
\end{proof}

This has the following key consequence:
\begin{lem}\label{flowerconservative}
    Let $f^*\colon\C\to\D$ be a geometric functor whose right adjoint $f_*$ is conservative. 
    For any $\mc{X}\subseteq\D$, 
    \[f_*(\thicktensor{\mc{X}})\subseteq\thicktensor{f_*\mc{X}} \quad \text{and} \quad f_*(\Loctensor{\mc{X}})\subseteq\Loctensor{f_*\mc{X}}.\]
\end{lem}
\begin{proof}
    Note that $\thicktensor{\mc{X}}=\thick(\mc{X}\otimes\D^c)=\thick(\mc{X}\otimes f^*(\C^c))$ where the latter equality holds by conservativity (see \cref{cor:conservative}(2)). As $f_*$ is an exact functor, we therefore have \[f_*\thick(\mc{X} \otimes f^*(\C^c)) \subseteq\thick(f_*(\mc{X}\otimes f^*(\C^c))) = \thick(f_*\mc{X} \otimes \C^c) = \thicktensor{f_*\mc{X}}\] by the projection formula (\ref{projformula}) as claimed. The statement with localising $\otimes$-ideals follows similarly since $f_*$ is coproduct-preserving.
\end{proof}

\section{Torsion rigidity and Grothendieck duality}
In this section, given a proxy-small geometric functor $f^*\colon \C \to \D$ we characterise the objects $M$ of $\C$ such that $\Gammaf M$ is rigid in the torsion category $\Gammaf\C$ as those for which $f^*M$ is rigid in $\D$, see \cref{f*rigid}. We call such objects \emph{$f^*$-rigid}. We show that such objects satisfy many structural properties, most notably, Grothendieck duality. We also explain how $f^*$-rigidity is closely related to local dualisability in the strata of tt-categories as investigated in~\cite{BIKP}. The substance of the main results in this section is contained in \cref{subsec:rigidandGrothendieck}. However, we first require some coherence results which are proved for the careful reader in \cref{subsec:coherences}.

\subsection{Coherences}\label{subsec:coherences}
In order to prove these coherence results, we begin by setting some terminology and notation. 
\begin{chunk}
    For clarity and brevity, we refer to an adjoint pair of morphisms \[(M\otimes X\to Y)\longleftrightarrow (X\to\iHom(M,Y))\] in a closed symmetric monoidal category as $M$-\emph{adjuncts} of each other. For example, the evaluation map \[\ev_{M,N}\colon M\otimes\iHom(M,N)\to N\] is the $M$-adjunct of $\mrm{id}_{\iHom(M,N)}$. 
\end{chunk}

We first let $f^*\colon\C\to\D$ be a strong symmetric monoidal functor between closed symmetric monoidal categories, and refer the reader to \cite[Section III]{LM(M)S} for further details of the discussion here. We choose this notation to suggest that $f^*$ is a geometric functor but do not actively assume so here - the following applies to other monoidal functors such as $\Gammaf\colon \C \to \Gammaf\C$. 

\begin{chunk}\label{chunk:alphaiso}
    Consider the natural map
    \[\alpha_{M,N}\colon f^*\iHom(M,N) \to \iHom(f^*M,f^*N)\]
    constructed as the $f^*M$-adjunct of \[f^*M \otimes f^*\iHom(M,N) \simeq f^*(M \otimes \iHom(M,N)) \xrightarrow{f^*(\ev_{M,N})} f^*N\]
    
    This is an isomorphism provided that $M$ is rigid. In particular, if $M$ is rigid, then
    \[\alpha_{M,\1}\colon f^*(DM) \xrightarrow{\sim} D(f^*M)\] is a natural isomorphism.
\end{chunk}

\begin{chunk}\label{chunk:rigidspreserved} An important observation is that rigid objects are preserved by $f^*$. Indeed, we note that the evaluations fit into the following commutative diagram:
\begin{equation}\label{diagram:evmonoidal}
\begin{tikzcd}[row sep=1cm, column sep=1cm]
	{f^*\iHom(M,N)\otimes f^*M} & {f^*(\iHom(M,N)\otimes M)} \\
	{\iHom(f^*M,f^*N)\otimes f^*M} & f^*N.
	\arrow["\simeq", from=1-1, to=1-2]
	\arrow["{\alpha_{M,N}\otimes f^*M}"', from=1-1, to=2-1]
	\arrow["{f^*(\ev_{M,N})}", from=1-2, to=2-2]
	\arrow["{\ev_{f^*M,f^*N}}"', from=2-1, to=2-2]
\end{tikzcd}
\end{equation}
    
    As a consequence, the natural transformation \[\rho_{M,N}\colon DM\otimes N\to \iHom(M,N)\] 
    constructed as the $M$-adjunct of
    \[M \otimes DM \otimes N \xrightarrow{\mrm{ev}_{M,\1} \otimes N} \1_\C \otimes N \simeq N\]
    fits into a commutative diagram
     \begin{equation}\label{diagram:rhomonoidal}
    \begin{tikzcd}[row sep=1cm, column sep=1cm]
        f^*(DM) \otimes f^*N \ar[r, "\simeq"] \ar[d, "\alpha_{M,\1} \otimes f^*N"'] & f^*(DM \otimes N) \ar[r, "f^*(\rho_{M,N})"] & f^*\iHom(M,N) \ar[d, "\alpha_{M,N}"] \\
        D(f^*M) \otimes f^*N \ar[rr, "\rho_{f^*M,f^*N}"'] & & \iHom(f^*M,f^*N).
    \end{tikzcd}
    \end{equation}
    These coherences are recorded in ~\cite{FHM} as diagrams (1.7) and (1.8).
    In particular, by taking $M = N$ in \cref{diagram:rhomonoidal}, it follows from \cref{chunk:alphaiso} that if $M$ is rigid, then so is $f^*M$. 
\end{chunk}

\begin{chunk}\label{coherenceHypothesis}
We now prove some further coherence results. We henceforth assume that the strong symmetric monoidal functor $f^*$ has a right adjoint $f_*$, noting that $f_*$ is then automatically lax symmetric monoidal (the pair $(f^*,f_*)$ is a \emph{monoidal adjunction} between closed symmetric monoidal categories). Once again we are evoking notation for geometric functors but do not make that assumption. For such a general monoidal adjunction, we note that the internal adjunction transformation \[a_{M,Y}\colon\iHom(M,f_*Y)\xrightarrow{\simeq}f_*\iHom(f^*M,Y)\]
of \cref{internaladjunction} remains an isomorphism, however the projection formula need not hold, in that the natural transformation \[p_{M,Y}\colon M\otimes f_*Y\to f_*(f^*M\otimes Y)\] of \cref{projformula} need not be an isomorphism as it is for geometric functors.
\end{chunk}

We wish to show the compatibility of this adjunction with the tensor-hom adjunctions, recorded as \cref{cor:tensorhominteraction}, but we shall first prove a special case.

\begin{lem}\label{lem:evAdjunction}
    Let $(f^*,f_*)$ be as in \cref{coherenceHypothesis}, $M\in\C$, and $Y\in\D$. The diagram
    \begin{equation}\label{evcoherence}
    \begin{tikzcd}[row sep=1cm, column sep=1cm]
        \iHom(M,f_*Y)\otimes M \ar[r, "a \otimes M", "\simeq"'] \ar[d, "{\ev_{M, f_*Y}}"'] & f_*\iHom(f^*M,Y) \otimes M \ar[d, "p"] \\
        f_*Y & f_*(\iHom(f^*M,Y) \otimes f^*M) \ar[l, "{f_*(\ev_{f^*M,Y})}"]
    \end{tikzcd}
    \end{equation}
    commutes.
\end{lem}
\begin{proof}
    We shall verify that the ``right-down-left'' composite has the same adjunct as $\ev_{M,f_*Y}$ under $(f^*,f_*)$. To do so, we consider the following diagram:
    \begin{equation}\label{evAdjunctionPf}
        \begin{tikzcd}[row sep=1cm, column sep=1cm]
    	{f^*\iHom(M,f_*Y)\otimes f^*M} & {f^*(\iHom(M,f_*Y)\otimes M)} \\
    	{\iHom(f^*M,f^*f_*Y)\otimes f^*M} & {f^*f_*Y} \\
    	{\iHom(f^*M,Y)\otimes f^*M} & Y
    	\arrow["\simeq", from=1-1, to=1-2]
    	\arrow["{\alpha_{M,f_*Y}\otimes f^*M}"', from=1-1, to=2-1]
    	\arrow["{f^*(\ev_{M,f_*Y})}", from=1-2, to=2-2]
    	\arrow["{\ev_{f^*M,f^*f_*Y}}"', from=2-1, to=2-2]
    	\arrow["{\iHom(f^*M,\varepsilon_Y)\otimes f^*M}"', from=2-1, to=3-1]
    	\arrow["{\varepsilon_Y}", from=2-2, to=3-2]
    	\arrow["{\ev_{f^*M,Y}}"', from=3-1, to=3-2]
    \end{tikzcd}
    \end{equation}
    
    The upper square commutes by \cref{diagram:evmonoidal}, and the lower square by naturality of $\ev_{f^*M,-}$.

    Note that the right hand composite in \cref{evAdjunctionPf} is precisely the adjunct of $\ev_{M,f_*Y}$, and so we shall show the ``right-down-left'' composite of \cref{evcoherence} has adjunct given by the ``left-down-down-right'' composite $f^*(\iHom(M,f_*Y)\otimes Y)\to Y$ in \cref{evAdjunctionPf}.

    Firstly, to identify the adjunct of $p\circ (a\otimes M)$ we use a general fact that follows easily from definitions: if $\gamma\colon X\to f_*Y$ is a morphism with adjunct $\gamma^\sharp\colon f^*X \to Y$, then for $Z\in\C$, the adjunct of the composite \[X\otimes Z\xrightarrow{\gamma\otimes Z}f_*Y\otimes Z\xrightarrow{p} f_*(Y\otimes f^*Z)\] is the composite \[f^*(X\otimes Z)\simeq f^*X\otimes f^*Z\xrightarrow{\gamma^\sharp\otimes f^*Z} Y\otimes f^*Z\]

    We apply this to $\gamma = a\colon\iHom(M,f_*Y)\to f_*\iHom(f^*M,Y)$ which is defined as the adjunct of \[f^*\iHom(M,f_*Y)\xrightarrow{\alpha}\iHom(f^*M,f^*f_*Y)\xrightarrow{\iHom(f^*M,\varepsilon)}\iHom(f^*M,Y)\] so that $p\circ(a\otimes M)$ has adjunct given by the ``left-down-down'' composite in \cref{evAdjunctionPf}. 
    
    From the definition of adjuncts, given maps $\gamma\colon X \to f_*Y$ and $\delta\colon Y \to Z$, it follows that the adjunct of $f_*(\delta)\circ \gamma\colon X \to f_*Z$ is given by the composite
    \[f^*X \xrightarrow{\gamma^\sharp} Y \xrightarrow{\delta} Z.\] Applying this to $\gamma = p \circ (a \otimes M)$ and $\delta = \mrm{ev}_{f^*M,Y}$, we therefore deduce that the ``right-down-left'' composite of \cref{evcoherence} has adjunct given by the ``left-down-down-right'' composite $f^*(\iHom(M,f_*Y)\otimes Y)\to Y$ in \cref{evAdjunctionPf} as required.
\end{proof}

\begin{cor}\label{cor:tensorhominteraction}
    Let $(f^*,f_*)$ be as in \cref{coherenceHypothesis}. For a morphism $\gamma\colon Y\to\iHom(f^*M,X)$ in $\D$, let $\overline{\gamma}\colon f^*M\otimes Y\to X$ be its $f^*M$-adjunct.
    The $M$-adjunct of the composite \[f_*Y\xrightarrow{f_*(\gamma)} f_*\iHom(f^*M,X)\xrightarrow{a^{-1}}\iHom(M,f_*X)\] is the composite \[M\otimes f_*Y\xrightarrow{p} f_*(f^*M\otimes Y)\xrightarrow{f_*(\overline{\gamma})}f_*X.\]
\end{cor}
\begin{proof}
    Consider the following diagram \[\begin{tikzcd}[row sep=1cm, column sep=1cm]
	{M\otimes f_*Y} & {f_*(f^*M\otimes Y)} \\
	{M\otimes f_*\iHom(f^*M,X)} & {f_*(f^*M\otimes\iHom(f^*M,X))} \\
	{M\otimes\iHom(M,f_*X)} & {f_*X}
	\arrow["p", from=1-1, to=1-2]
	\arrow["{M\otimes f_*(\gamma)}"', from=1-1, to=2-1]
	\arrow["{f_*(f^*M\otimes\gamma)}", from=1-2, to=2-2]
	\arrow["p", from=2-1, to=2-2]
	\arrow["{M\otimes a^{-1}}"', from=2-1, to=3-1]
	\arrow["{f_*(\ev)}", from=2-2, to=3-2]
	\arrow["\ev"', from=3-1, to=3-2]
\end{tikzcd}\] where the upper square commutes by naturality of $p$ and the lower by \cref{lem:evAdjunction}. The entire left-bottom composite is the $M$-adjunct of $a^{-1}\circ f_*(\gamma)$. The right-hand vertical composite is precisely $f_*$ applied to the $f^*M$-adjunct of $\gamma$, and thus we identify the entire composite with $f_*(\overline\gamma)\circ p$ as required.
\end{proof}

\begin{chunk}\label{chunk:nudefn}
    We can use this to prove the coherence of other monoidal constructions. Consider the natural transformation \[\nu_{M,X,N}\colon\iHom(M,X)\otimes N\to\iHom(M,X\otimes N)\] constructed as the $M$-adjunct of \[\ev_{M,X}\otimes N\colon M\otimes\iHom(M,X)\otimes N\to X\otimes N.\]
\end{chunk}
 The proof of our characterisation of $f^*$-rigid objects (see \cref{f*rigid}) will require the following coherence property of $\nu$:
\begin{lem}\label{lem:nudiagramcommutes}
    Let $(f^*,f_*)$ be as in \cref{coherenceHypothesis}. For all $M,N\in\C$ and $Y\in \D$, the following diagram commutes:
    \begin{equation}\label{nudiagram}
        \begin{tikzcd}[column sep=0.6cm]
	{\iHom(M,f_*Y)\otimes N} & & {\iHom(M,f_*Y\otimes N)} \\
	{f_*\iHom(f^*M,Y)\otimes N} & & \\
	{f_*(\iHom(f^*M,Y)\otimes f^*N)} &
    {f_*\iHom(f^*M,Y\otimes f^*N)} & {\iHom(M,f_*(Y\otimes f^*N))}.
	\arrow["{\nu}", from=1-1, to=1-3]
	\arrow["\simeq", "{a\otimes N}"', from=1-1, to=2-1]
	\arrow["{\iHom(M,p)}", from=1-3, to=3-3]
	\arrow["p"', from=2-1, to=3-1]
	\arrow["{f_*(\nu)}"', from=3-1, to=3-2]
	\arrow["\simeq", "{a^{-1}}"', from=3-2, to=3-3]
    \end{tikzcd}
    \end{equation}
\end{lem}
\begin{proof}
    We shall show that the two diagonal composites \[\iHom(M,f_*Y)\otimes N\to\iHom(M,f_*(Y\otimes f^*N))\] coincide by showing that their $M$-adjuncts agree.
    We identify first the top-right composite. Since $\nu_{M,f_*Y,N}$ is the $M$-adjunct of $\ev_{M,f_*Y}\otimes N$, it follows by naturality that $p_{N,Y}\circ (\ev_{M,f_*Y}\otimes N)$ has $M$-adjunct $\iHom(M,p_{N,Y})\circ\nu_{M,f_*Y,N}$, which is precisely the top-right composite. It thus suffices to show that the adjunct of the left-bottom composite is $p_{N,Y}\circ (\ev_{M,f_*Y}\otimes N)$.

    To do so, we first use \cref{cor:tensorhominteraction} with $\gamma=\nu_{f^*M,Y,f^*N}$, which shows that the bottom composite has $M$-adjunct given by \[f_*(\ev_{f^*M,Y}\otimes f^*N)\circ p_{M,\iHom(f^*M,Y) \otimes f^*N}\]
    Thus, by naturality of the adjunction, the adjunct of the entire left-bottom composite of \cref{nudiagram} is the rightmost composite in the following diagram.
    \[\begin{tikzcd}
	& {M\otimes\iHom(M,f_*Y)\otimes N} \\
	& {M\otimes f_*\iHom(f^*M,Y)\otimes N} \\
	& {M\otimes f_*(\iHom(f^*M,Y)\otimes f^*N)} \\
	{f_*(f^*M\otimes \iHom(f^*M,Y))\otimes N} & {f_*(f^*M\otimes \iHom(f^*M,Y)\otimes f^*N)} \\
	{f_*(f^*M\otimes \iHom(f^*M,Y))\otimes N} & {f_*(Y\otimes f^*N)}
	\arrow["{M\otimes a\otimes N}", from=1-2, to=2-2]
	\arrow["{M\otimes p}", from=2-2, to=3-2]
	\arrow["{p\otimes N}"', from=2-2, to=4-1, to path={[pos=0.25] -| (\tikztotarget) \tikztonodes}]
	\arrow["p", from=3-2, to=4-2]
	\arrow["p", from=4-1, to=4-2]
	\arrow["{f_*(\ev)\otimes N}"', from=4-1, to=5-1]
	\arrow["{f_*(\ev \otimes f^*N)}", from=4-2, to=5-2]
	\arrow["p"', from=5-1, to=5-2]
\end{tikzcd}\]
where the middle square is an associativity property of the projection formula, and the lower square is by its naturality. We thus see that the adjunct of the left-bottom composite of \cref{nudiagram} is \[p_{N,Y} \circ (f_*(\ev_{f^*M,Y})\circ p_{M,\iHom(f^*M,Y)} \circ (M\otimes a_{M,Y}))\otimes N.\] By \cref{evcoherence} we identify the bracketed term as $\ev_{M,f_*Y}$, and so the adjuncts of both composites are indeed $p_{N,Y}\circ (\ev_{M,f_*Y}\otimes N)$.
\end{proof}

\subsection{Characterising $f^*$-rigid objects and Grothendieck duality}\label{subsec:rigidandGrothendieck}
We now return to our standard setting of geometric functors between rigidly-compactly generated tt-categories, see \cref{chunk:geometric}. The following is the key result of this section, from which we will derive several important consequences, including Grothendieck duality for $f^*$-rigid objects (\cref{f*compactisos}) which we will build on in the following sections.
\begin{thm}\label{f*rigid}
    Let $f^*\colon\C\to\D$ be a proxy-small geometric functor. Let $M\in\C$.
    The following are equivalent:
        \begin{enumerate}[label=(\arabic*)]
        \item the map $\rho_{M,N}\colon DM\otimes N\to\iHom(M,N)$ is an $f^*$-isomorphism for all $N \in \C$, i.e., $f^*(\rho_{M,N})$ is an isomorphism for all $N \in \C$;
        \item $\Gammaf M\in\Gammaf\C$ is rigid;
        \item $\Lambdaf M\in\Lambdaf\C$ is rigid;
        \item $f^*M\in\D$ is rigid.
        \end{enumerate}
        Moreover, these are implied by $\Gammaf M \in \thick(\Gammaf(\C^c))$.
\end{thm}
\begin{proof}
    We first prove (1) $\Leftrightarrow$ (3). Write $D_{\Lambdaf}$ for the functional dual in $\Lambdaf\C$. Applied to the strong monoidal functor $\Lambdaf\colon \C \to \Lambdaf\C$, \eqref{diagram:rhomonoidal} yields a commutative diagram
    \[
    \begin{tikzcd}[row sep=1cm, column sep=1.2cm]
        \Lambdaf(DM) \widehat{\otimes} \Lambdaf N \ar[r, "\simeq"] \ar[d, "\alpha_{M,\1} \widehat{\otimes} \Lambdaf N"'] & \Lambdaf (DM \otimes N) \ar[r, "\Lambdaf(\rho_{M,N})"] & \Lambdaf\iHom(M,N) \ar[d, "\alpha_{M,N}"] \\
        D_{\Lambdaf}(\Lambdaf M) \widehat{\otimes} \Lambdaf N \ar[rr, "\rho_{\Lambdaf M, \Lambdaf N}"'] & & \iHom(\Lambdaf M, \Lambdaf N)
    \end{tikzcd}
    \]
    for any $N \in \C$.
    Since $\Lambdaf\colon \C \to \Lambdaf\C$ is moreover \emph{closed} monoidal, the map $\alpha_{M,N}$ is an isomorphism for all $M,N \in \C$. Therefore $\Lambdaf(\rho_{M,N})$ is an isomorphism if and only if $\rho_{\Lambdaf M, \Lambdaf N}$ is an isomorphism (i.e., $\Lambdaf M \in \Lambdaf\C$ is rigid). Now by \cref{Wequivalences}, $\Lambdaf(\rho_{M,N})$ is an isomorphism if and only if $f^*(\rho_{M,N})$ is an isomorphism. This proves that (1) is equivalent to (3). Conditions (2) and (3) are equivalent by the MGM equivalence (see \cref{thm:localduality} and \cref{monoidalcompletetorsion}).

    For (2) $\Rightarrow$ (4), we have that $f^*\colon \C \to \D$ factorises as \[\C \xrightarrow{\Gammaf} \Gammaf\C \xrightarrow{f^*} \D\] by \cref{prop:changeofbase}. Since $f^*\colon \Gammaf\C \to \D$ is strong monoidal, and $\Gammaf M$ is rigid by assumption, it follows that $f^*M$ is rigid, as strong monoidal functors preserve rigid objects, see \cref{chunk:rigidspreserved}.

    To conclude the proof, we prove that (4) $\Rightarrow$ (1). Recall the natural transformation \[\nu_{M,X,N}\colon\iHom(M,X)\otimes N\to\iHom(M,X\otimes N)\] as defined in \cref{chunk:nudefn}, and let $\mc{X}$ be the subcategory of $X \in \C$ such that $\nu_{M,X,N}$ is an equivalence for all $N \in \C$. Note that for $C\in\C$ rigid, by naturality $\nu_{M,X\otimes C,N}$ is identified under the obvious isomorphisms with $C \otimes \nu_{M,X,N}$. Thus we see that $\mc{X}$ is a thick $\otimes$-ideal of $\C$.
    
    We show that $f_*Y\in\mc{X}$ for all $Y\in\D$. To do so, we note \cref{lem:nudiagramcommutes} gives chains of natural isomorphisms (since in this setting of geometric functors, $p$ is an isomorphism) such that the diagram \[\begin{tikzcd}[column sep = 2.5cm]
	{\iHom(M,f_*Y)\otimes N} & {\iHom(M,f_*Y\otimes N)} \\
	{f_*(\iHom(f^*M,Y)\otimes f^*N)} & {f_*(\iHom(f^*M,Y\otimes f^*N)}
	\arrow["{\nu_{M,f_*Y,N}}", from=1-1, to=1-2]
	\arrow["\simeq"', from=1-1, to=2-1]
	\arrow["\simeq", from=1-2, to=2-2]
	\arrow["{f_*(\nu_{f^*M,Y,f^*N})}"', from=2-1, to=2-2]
\end{tikzcd}\] commutes. Since $f^*M$ is rigid, $\nu_{f^*M,Y,f^*N}$ is an isomorphism for all $N$, and therefore by the above diagram we see that $\nu_{M,f_*Y,N}$ is also an isomorphism so $f_*Y\in\mc{X}$ as claimed.
    
    Since $\mc{X}$ is a thick $\otimes$-ideal, it follows that $\W\subseteq\mc{X}$ where $\W$ is a witness for proxy-smallness as $\W \subseteq \thicktensor{f_*\1_\D}$. Since each $W \in \W$ is rigid, we have that $\nu_{M,W,Y}$ corresponds under natural isomorphisms to $W \otimes \nu_{M,\1_\C,Y}$. As $\nu_{M,\1_\C,Y}=\rho_{M,Y}$, this proves that $W \otimes \rho_{M,Y}$ is an isomorphism for all $W \in \W$ and $Y \in \C$, which by \cref{Wequivalences} gives (1).

    For the final statement, since $\Gammaf\colon \C \to \Gammaf\C$ is strong monoidal, any object in $\Gammaf(\C^c)$ is rigid in $\Gammaf\C$, so the claim follows by a thick subcategory argument.
\end{proof}

\begin{rem}\label{rem:notallequivalent}
We note that the condition that $\Gammaf M \in \thick(\Gammaf(\C^c))$ is not equivalent to the other conditions in the previous theorem in general, see \cref{ex:Knrigid}.
\end{rem}

We now give a couple of examples explaining the relation of \cref{f*rigid} to the problem of identifying the rigid objects in the ``strata'' of tt-categories. Recall that, in general, the category $\Gammaf\C$ is not rigidly-compactly generated, and as such it is of interest to characterise its rigid objects.

\begin{ex}\label{ex:BIKPRigid} Benson--Iyengar--Krause--Pevtsova~\cite{BIKP} undertook this in commutative algebra, where the category of interest $\Gammaf\C$ is the category of $\p$-local $\p$-torsion objects $\Gamma_\p \mathsf{D}(R_\p)$ in $\mathsf{D}(R)$, where $R$ is a commutative Noetherian ring, and $\p$ is a prime ideal of $R$. More precisely, we take $f^*$ to be the extension of scalars along the ring map $R_\p \to \kappa(\p)$ where $\kappa(\p) = R_\p/\p R_\p$ is the residue field at $\p$. This is proxy-small by \cref{rem:localringps} so \cref{f*rigid} applies. In this setting, using the numbering of \cref{f*rigid}, Benson--Iyengar--Krause--Pevtsova prove that (2)-(4) are equivalent, and that in this case, they are moreover equivalent to $\Gammaf M \in \thick(\Gammaf(\C^c))$. As such, \cref{f*rigid} gives a uniform tt-perspective on the results of \cite{BIKP}.
\end{ex}

\begin{ex}\label{ex:Knrigid}
We now explain how to apply the previous result to characterise the rigid objects in the category of $K(n)$-local spectra. Morally, one would like to apply the result to the functor $K(n) \otimes_{L_nS^0} -\colon L_n\Sp \to \Mod{}(K(n))$ where $L_nS^0$ denotes the $E(n)$-local sphere.
However, there are well known issues here: $K(n)$ does not admit an $E_2$-structure (or even a homotopy commutative ring structure at $p=2$). Nonetheless we explain how to treat this as a geometric functor for odd primes, purely at the level of homotopy categories. 

Since $K(n)$ is a field, there is a triangulated equivalence between $\mathsf{D}(K(n))$ and $\mathsf{D}(K(n)_*)$ given by taking homotopy. However, note that this cannot be promoted to an equivalence of enhancements, see \cite[\S 5.2]{IRW} or \cite[Remark 2.6]{SchwedeShipley}. Since $K(n)_* = \mathbb{F}_p[v_n^{\pm 1}]$ where $|v_n| = 2(p^n-1)$ is a graded field, note that $\mathsf{D}(K(n)_*)$ is moreover equivalent to the category of graded $K(n)_*$-modules, and hence the latter admits a tensor-triangulated structure. At odd primes, the functor $K(n)_*(-)\colon h(L_n\Sp) \to \mathrm{grMod}(K(n)_*)$ is then a geometric functor: it is clearly exact and coproduct-preserving, and is strong symmetric monoidal since $K(n)$ is homotopy commutative. (At $p=2$ this functor is not \emph{symmetric} monoidal; one can still use similar ideas to the proof of \cref{f*rigid} in this case but we will not go into details here.
Instead, we direct the reader to \cite[Theorem 8.6]{HS666}.) 

The geometric functor $K(n)_*(-)\colon h(L_n\Sp) \to \mathrm{grMod}(K(n)_*)$ is proxy-small: by \cite[Proposition 2.12]{DGIGrossHopkins}, $K(n)_*$ is proxy-small in $L_n\Sp$ with witness $L_nF(n)$ for a finite type $n$ spectrum $F(n)$, and since any module in the derived category of a graded field is proxy-small with witness given by any finite dimensional module, the claim follows. Note that the associated completion functor $\Lambdaf$ is the $K(n)$-localisation so that $\Lambdaf(h(L_n\Sp)) = h(L_{K(n)}\Sp)$. As such applying \cref{f*rigid} shows that a $K(n)$-local spectrum $M$ is rigid if and only if $K(n)_*M$ is finite dimensional, thus recovering part of \cite[Theorem 8.6]{HS666}. Moreover, there exists a rigid $K(1)$-local spectrum $Z$ such that $Z$ is not in $\thick(L_{K(1)}S^0)$, see \cite[\S 15.1]{HS666}, thus justifying the claim in \cref{rem:notallequivalent}.
\end{ex}

We now move onto the consequences of \cref{f*rigid}, and show that $f^*$-rigidity of $M\in\C$ is usually sufficient for $M$ to enjoy properties in relation to $f^*$ that one would typically only expect if $M$ itself were rigid. These properties will allow us to prove Grothendieck duality for $f^*$-rigid objects below.

\begin{cor}\label{f*reflexive}
    Suppose $f^*\colon \C \to \D$ is a proxy-small geometric functor and let $M \in \C$. If $f^*M$ is rigid, then the canonical map \[f^*(d_M)\colon f^*M \to f^*(D^2M)\] is an isomorphism.
\end{cor}
\begin{proof}
    By \cref{Wequivalences} it suffices to prove that $\Lambdaf (d_M)$ is an isomorphism. Since $f^*M$ is rigid in $\D$, by \cref{f*rigid} we have that $\Lambdaf M$ is rigid in $\Lambdaf\C$ and hence the natural map $d'_{\Lambdaf M}\colon \Lambdaf M \to D_{\Lambdaf}^2(\Lambdaf M)$ is an isomorphism, where $D_{\Lambdaf}$ denotes the functional dual in $\Lambdaf\C$. Recall from \cref{monoidalcompletetorsion} that the internal hom in complete modules is computed in the ambient category and that the tensor unit is $\Lambdaf \1$. Since $\Lambdaf\colon \C \to \Lambdaf\C$ is closed monoidal, there is a natural isomorphism $D^2_{\Lambdaf}(\Lambdaf(-)) \simeq \Lambdaf D^2(-)$ making the diagram
    \[\begin{tikzcd}[row sep=1cm, column sep=1cm]
        \Lambdaf M \ar[r, "\Lambdaf(d_M)"] \ar[dr, "d'_{\Lambdaf M}"', "\simeq"] & \Lambdaf D^2 M \ar[d, "\simeq"] \\
        & D_{\Lambdaf}^2(\Lambdaf M) 
    \end{tikzcd}\]
    commute. We see from this diagram that $\Lambdaf(d_M)$ is an isomorphism as claimed.
\end{proof}

\begin{cor}\label{f*compactclosed}
    Suppose $f^*\colon \C \to \D$ is a proxy-small geometric functor and let $M \in \C$. Recall the natural map $\alpha_{M,N}\colon f^*\iHom(M,N) \to \iHom(f^*M,f^*N)$ from \cref{chunk:alphaiso}. If $f^*M$ is rigid, then for all $N\in\C$ there are natural isomorphisms:
    \begin{equation}\label{f*closed1}
        \alpha_{M,N}\colon f^*\iHom(M,N) \xrightarrow{\sim} \iHom(f^*M,f^*N)
    \end{equation}
    \begin{equation}\label{f*duals}
        \alpha_{M,\1}\colon f^*(DM) \xrightarrow{\sim} D(f^*M).
    \end{equation}
\end{cor}
\begin{proof} 
Recall from \cref{prop:changeofbase} that $f^*\colon \C \to \D$ admits a factorisation \[\C\xrightarrow{\Gammaf}\Gammaf\C\xrightarrow{f^*_t}\D\] where $f^*_t$ is the restriction of $f^*$ to $\Gammaf\C$. We note that $f^*_t$ is also strong symmetric monoidal. Thus, the construction of \cref{chunk:alphaiso} composes, i.e., $\alpha_{M,N}$ is the composition \[f^*_t(\Gammaf\iHom_\C(M,N))\xrightarrow{f^*_t(\gamma_{M,N})}f^*_t\iHom_{\Gammaf\C}(\Gammaf M,\Gammaf N)\xrightarrow{\beta_{\Gammaf M,\Gammaf N}}\iHom_\D(f^*_t\Gammaf M, f^*_t\Gammaf N)\] where $\beta$ and $\gamma$ are the natural transformations analogous to $\alpha$ associated to $f^*_t$ and $\Gammaf$. 
Since $\Gammaf\colon \C \to \Gammaf\C$ is closed monoidal, $\gamma$ is a natural isomorphism, and hence $f^*_t(\gamma_{M,N})$ is an isomorphism. 
Therefore, to prove \cref{f*closed1} it suffices to prove that $\beta_{\Gammaf M, \Gammaf N}$ is an isomorphism. Since $f^*M$ being rigid implies $\Gammaf M\in\Gammaf\C$ is rigid by \cref{f*rigid}, this is immediate from \cref{chunk:alphaiso}. Finally, \cref{f*duals} is the case $N=\1_\C$ in \cref{f*closed1}.  
\end{proof}

In addition to the convenient features of $f^*$-rigid objects proved above, we now prove that such objects satisfy Grothendieck duality. We set $\omega_f \coloneqq  f^{!}(\1_\C)$; in \cite{BDS} this is called the \emph{relative dualising object}.
\begin{thm}\label{f*compactisos}
    Let $f^*\colon \C \to \D$ be a proxy-small geometric functor, and $M \in \C$. If $f^*M$ is rigid, then there is a natural isomorphism
    \begin{equation}\label{strongGN}
    f^{!}X \otimes f^*M \simeq f^{!}(X \otimes M)
    \end{equation}
    for all $X \in \C$. In particular, if $f^*M$ is rigid there is an isomorphism
    \begin{equation}\label{GN}
    \omega_f \otimes f^*M \simeq f^{!}M.
    \end{equation}
\end{thm}
\begin{proof}
    We have natural isomorphisms
     \[f^{!}X \otimes f^*M \simeq \iHom(D(f^*M), f^{!}X) \simeq \iHom(f^*(DM), f^{!}X) \simeq f^{!}\iHom(DM,X)\]
    using the rigidity of $f^*M$, \cref{f*duals}, and then \cref{f1ofhom}. To conclude it therefore suffices to show that the canonical map
    \[M \otimes X \xrightarrow{d_M \otimes X} D^2M \otimes X \xrightarrow{\rho_{DM,N}} \iHom(DM,X)\]
    is an $f^{!}$-isomorphism, where $\rho_{DM,N}$ is as in \cref{chunk:rigidspreserved}. By  \cref{Wequivalences}, this is equivalent to checking that it is an $f^*$-isomorphism. 
    
    Note that since $f^*M$ is rigid, it follows that $f^*(DM)$ is also rigid (since it is isomorphic to $D(f^*M)$ by \cref{f*duals}). Now consider the following diagram in which the left hand vertical is the map we want to show is an isomorphism. We have annotated some maps as isomorphisms together with references justifying this.
    \[
    \begin{tikzcd}[row sep=1.5cm, column sep=1.7cm]
        f^*(M \otimes X) \ar[d, "f^*(d_M \otimes X)"'] \ar[r, "\simeq"] & f^*M \otimes f^*X \ar[d, "f^*(d_M) \otimes f^*X", "\substack{\simeq \\ \\(\ref{f*reflexive})}"'] \\
        f^*(D^2M \otimes X) \ar[d, "f^*(\rho_{DM,N})"'] \ar[r, "\simeq"] & f^*(D^2M) \otimes f^*X \ar[r,"\alpha_{DM,\1} \otimes f^*X", "\simeq~~(\ref{f*duals})"'] & D(f^*DM) \otimes f^*X \ar[d, "\rho_{f^*DM,f^*X}", "\simeq"'] \\
        f^*\iHom(DM,X) \ar[rr, "\alpha_{DM,X}"', "\simeq~~(\ref{f*closed1})"] & & \iHom(f^*DM,f^*X)
    \end{tikzcd}
    \]
    Note that the diagram commutes: the top square by the strong monoidality of $f^*$, and the bottom rectangle by \eqref{diagram:rhomonoidal}. Therefore, we see that the claimed map is an isomorphism as required.
    This completes the proof of \cref{strongGN}. Taking $X = \1_\C$ in \cref{strongGN} gives \cref{GN}.
\end{proof}

\begin{rem}
    The isomorphism \cref{GN} establishes Grothendieck duality for all $M \in \C$ such that $f^*M$ is rigid in $\D$. This complements results of Balmer--Dell'Ambrogio--Sanders~\cite[Theorem 1.7]{BDS} and Fausk--Hu--May~\cite[Proposition 5.4]{FHM}. Balmer--Dell'Ambrogio--Sanders prove that Grothendieck duality holds for all $M$ if and only if $f_*$ preserves compacts. The previous theorem shows that by weakening the assumption that $f_*$ preserves compacts to the hypothesis that $f^*$ is proxy-small, one still obtains Grothendieck duality on the subcategory of objects $M$ for which $f^*M$ is rigid. Fausk--Hu--May prove that Grothendieck duality holds for $M$ provided that $M$ is rigid, and under the proxy-smallness assumption, we extend this to the $M$ for which $f^*M$ is rigid. As such, the previous theorem may be seen as bridging the gap between the results of \cite{BDS} and \cite{FHM}.
\end{rem}

\begin{rem}
    We note that Grothendieck duality \cref{GN} for $f^*$-rigid objects can fail without the proxy-smallness assumption. Indeed, take $f^*\colon \mathsf{D}(\Z)\to \mathsf{D}(\Q)$ to be the extension of scalars $\Q \otimes_\Z -$, and $M=\Q/\Z$. Note that $f_*$ is the restriction of scalars. We see that $\Q$ is not proxy-small in $\mathsf{D}(\Z)$: $\thicktensor{\Q}$ can only contain elements of $f_*(\mathsf{D}(\Q))$, and the only element of this which is compact in $\mathsf{D}(\Z)$ is the zero object. We see that $f^*M = \Q \otimes_\Z (\Q/\Z) \simeq 0$ (and so in particular is rigid). However, $\omega_f \otimes f^!M \simeq \Hom_\Z(\Q,\Z) \otimes \Hom_\Z(\Q, \Q/\Z)$ is non-zero. More generally, fix a rigidly-compactly generated tt-category $\C$ and a proxy-small set $\K$ in $\C$, and take $f^*$ to be the localisation functor $L\colon\C \to L_\K\C$ and $M = \Gamma\1$. One can run the same counterexample in this setting too. (The explicit example given above is then the case $\C = \mathsf{D}(\Z)$ and $\K = \{\Z/p \mid \text{$p$ prime}\}$).
\end{rem}

\section{Torsion invertibility}
In this section, given a proxy-small geometric functor $f^*\colon \C \to \D$ we characterise the objects $M$ of $\C$ such that $f^*M$ is invertible in $\D$, see \cref{f*invertible}. This result plays a key role in the next section in our study and classification of Matlis dualising objects.
\subsection{The Picard group}
Firstly we recall some preliminaries.
\begin{defn}
    Let $(\C,\otimes,\1)$ be a symmetric monoidal category. An object $M \in \C$ is \emph{invertible} if the functor $M\otimes-$ is an autoequivalence of $\C$ (or, equivalently, there exists $N \in \C$ such that $M\otimes N\simeq\1$). We write $\Pic(\C)$ for the collection of isomorphism classes of invertible objects. This forms an abelian group under $\otimes$ called the \emph{Picard group}.
\end{defn}

\begin{chunk}\label{chunk:picard} We note the following properties of invertible objects:
\begin{enumerate}[label=(\roman*)]
    \item Invertible objects satisfy 2-out-of-3: if $X \otimes Y \simeq Z$, then if any two of $\{X,Y,Z\}$ are invertible then so is the third.
    \item If $\C$ is closed symmetric monoidal, then all invertible objects are rigid. Hence, for any invertible $M\in\C$, $DM\otimes M\simeq\iHom(M,M)$, which by invertibility is $\iHom(\1,\1)\simeq\1$. That is, the inverse $N$ of $M$ is isomorphic to $DM$. Under this isomorphism, the map $M \otimes N \to \1$ corresponds to the evaluation map, and therefore an object $M \in \C$ is invertible if and only if the evaluation map $\ev_{M,\1}\colon DM \otimes M \to \1$ is an isomorphism. Throughout this section we shall write $\ev_M=\ev_{M,\1}$.
    \item If $f^*\colon \C \to \D$ is a strong symmetric monoidal functor between symmetric monoidal categories, then $f^*$ preserves invertible objects and hence induces a group homomorphism $\Pic(\C)\to\Pic(\D)$.
\end{enumerate}
\end{chunk}

We now focus our attention on the Picard groups of tt-categories.

\begin{chunk}\label{picardoftentrivial}
    In tt-categories, any suspension $\Sigma^n\1$ of the unit is invertible. Hence, it is rare for $\Pic(\C)$ to be the trivial group, requiring an equivalence $\Sigma\1\simeq\1$ (i.e. the graded endomorphism ring $\mrm{End}_*(\1)=\Hom(\Sigma^*\1,\1)$ being $1$-periodic). This is typically obstructed by the graded-commutativity of $\mrm{End}_*(\1)$ but is not impossible in characteristic 2. For example, take $\C$ to the stable category of $\mbb{F}_2C_2$-modules, so that $\mrm{End}_*(\1)$ is the Tate cohomology ring $\hat{H}^{-*}(C_2;\mbb{F}_2)\simeq\mbb{F}_2[t,t^{-1}]$ where $t$ has degree 1. The short exact sequence \[0 \to \mbb{F}_2 \to \mbb{F}_2C_2 \to \mbb{F}_2 \to 0\] shows that $\Sigma \1 \simeq \1$, so the Picard group is the trivial group in this case.
    Nonetheless, in general we do not expect $\Pic(\C)$ to be trivial, and we reserve the term for ``trivially non-trivial'' Picard groups:
\end{chunk}

\begin{defn}\label{defn:trivialpicard}
    Let $\C$ be a tt-category. We say $\C$ \emph{has trivial Picard group} if the homomorphism \[\Z\to\Pic(\C)\] given by $n\mapsto\Sigma^n\1$ is surjective, that is, if all invertible objects are suspensions of the unit.
\end{defn}

\begin{exs}\label{ex:trivialPic}
    Many tt-categories of interest have trivial Picard group:
    \begin{enumerate}[label=(\roman*)]
        \item Let $R$ be a graded commutative Noetherian local ring. Then the Picard group of $\msf{D}(R)$ is trivial. (Moreover, the Picard group of the abelian category $\mrm{grMod}(R)$ of graded $R$-modules is also trivial.)
        \item Let $R$ be a commutative ring spectrum. Baker--Richter~\cite{BR} construct a monomorphism
        \[\Pic(\mrm{grMod}(\pi_*R)) \to \Pic(\msf{D}(R))\] and prove that it is an isomorphism if $R$ is connective or even-periodic. In conjunction with the previous example, this shows that the derived categories of many commutative ring spectra have trivial Picard group.
        \item The Picard group of spectra is trivial, and is isomorphic to $\mbb{Z}$, see \cite{HMS}.
    \end{enumerate}
\end{exs}

Despite the previous example, there are also many instances where the Picard group is non-trivial: this is one key motivation for developing a theory of Gorenstein geometric functors and duality beyond the setting covered by \cite{DGI}.
\begin{exs}\label{exs:nontrivialPic}\leavevmode
    \begin{enumerate}[label=(\roman*)]
        \item For a compact Lie group $G$, all representation spheres are invertible in the category $\Sp_G$ of genuine $G$-spectra, giving a group homomorphism $RO(G)\to\Pic(\Sp_G)$. In contrast to the non-equivariant situation, this map is not an isomorphism in general. For simplicity, let us now assume that $G$ is finite; one can modify the following discussion appropriately by taking the topology on the space of conjugacy classes of closed subgroups into account, see \cite{FLMpicard}. By taking geometric fixed-points at each conjugacy class $(H)$ of subgroups, one produces a strong symmetric monoidal functor $\Sp_G\to\prod_{(H)}\Sp$. In \cite{tDP}, tom Dieck and Petrie identify the kernel of the corresponding map on Picard groups \[\Pic(\Sp_G)\xrightarrow{\mrm{deg}}\prod_{(H)}\Z\] as $\Pic(A(G))$, the Picard group of the Burnside ring. In general, the cokernel of the map is non-trivial. We direct the reader to work of Krause~\cite{KrausePicard}, which identifies the image in several cases.
        \item Let $k$ be a field of characteristic $p$ and $G$ be a finite $p$-group. The Picard group of the stable module category $\mrm{StMod}(kG)$ is the group of endotrivial modules, and is not trivial (in the sense of \cref{defn:trivialpicard}) in general. For instance, if $p=2$ and $k$ contains cube roots of unity then for the quaternion group $Q_8$, the Picard group is isomorphic to $\Z/4 \oplus \Z/2$ with the $\Z/4$ component generated by $\Sigma\1$, see \cite[Theorem 6.3]{CarlsonThevenaz}. The same holds for generalised quaternion groups $Q_{2^{n+1}}$ for $n \geq 3$ also, see \cite[Theorem 6.5]{CarlsonThevenaz}.
    \end{enumerate}
\end{exs}

\begin{rem}
    It is important for our purposes to note that given a geometric functor $f^*\colon \C \to \D$, even if $\C$ and $\D$ both have trivial Picard group, the category $\Gammaf\C$ need not have trivial Picard group, see \cref{rem:Picfactorisation}.
\end{rem}

\subsection{Invertible torsion and complete objects}
We now consider the Picard groups for the torsion and complete categories associated to a proxy-small geometric functor $f^*\colon\C\to\D$. Recall that the MGM equivalence \cref{monoidalcompletetorsion} is monoidal, so it induces an isomorphism $\Pic(\Gammaf\C)\simeq\Pic(\Lambdaf\C)$ of groups. In contrast to the examples given in the previous section, $\Gammaf\C$ is not generally rigidly-compactly generated. In particular, the unit is not typically compact and thus the Picard group does not consist of compact objects.

Even if $\C$ has trivial Picard group (in the sense of \cref{defn:trivialpicard}), it need not be the case that $\Gammaf\C$ has trivial Picard group, see \cref{rem:Picfactorisation}. Indeed, under the proxy-smallness assumption we have an analogous result to \cref{f*rigid}, that invertible objects in the torsion/complete categories are classified by their images under $f^*$.

\begin{thm}\label{f*invertible}
    Let $f^*\colon \C \to \D$ be a proxy-small geometric functor and let $M \in \C$.
    The following are equivalent:
        \begin{enumerate}[label=(\arabic*)]
        \item the evaluation map $\ev_M\colon M\otimes DM\to\1_\C$ is an $f^*$-isomorphism, i.e., $f^*(\ev_M)$ is an isomorphism;
        \item $\Gammaf M\in\Gammaf\C$ is invertible;
        \item $\Lambdaf M\in\Lambdaf\C$ is invertible;
        \item $f^*M\in\D$ is invertible.
    \end{enumerate}
\end{thm}
\begin{proof}
    Recall from \cref{chunk:picard} that an object $M$ is invertible if and only if $\mrm{ev}_M\colon DM \otimes M \to \1$ is an isomorphism. The proof is then analogous to \cref{f*rigid} by using the case $N = \1_\C$ of \eqref{diagram:evmonoidal} in place of \eqref{diagram:rhomonoidal}. In particular the proofs of (1) $\Leftrightarrow$ (3), (2) $\Leftrightarrow$ (3), and (2) $\Rightarrow$ (4) follow the same strategy, so we only give brief details. For (1) $\Leftrightarrow$ (3), applying \cref{diagram:evmonoidal} to $\Lambdaf\colon \C \to \Lambdaf\C$ and taking $N = \1_\C$, we have a commutative diagram
    \[\begin{tikzcd}[row sep=1cm, column sep=1.3cm]\Lambdaf(DM) \widehat{\otimes} \Lambdaf M \ar[r, "\simeq"] \ar[d, "\alpha_{M,\1} \widehat{\otimes} \Lambdaf M"'] & \Lambdaf(DM \otimes M) \ar[d, "\Lambdaf\mrm{ev}_M"] \\ 
    D_{\Lambdaf}(\Lambdaf M) \widehat{\otimes} \Lambdaf M \ar[r, "\mrm{ev}_{\Lambdaf M}"'] & \Lambdaf\1_\C.
    \end{tikzcd}\]
    
    We note that $\Lambdaf$ is closed monoidal, and so the left vertical arrow in the diagram is an isomorphism.
    Therefore, $\mrm{ev}_{\Lambdaf M}$ is an isomorphism if and only if $\mrm{ev}_M$ is a $\Lambdaf$-isomorphism, which is the same as an $f^*$-isomorphism by \cref{Wequivalences}, proving (1) $\Leftrightarrow$ (3).
    The equivalence of (2) and (3) is once again an immediate consequence of the MGM equivalence \cref{monoidalcompletetorsion}, and as in \cref{f*rigid} we prove (2) $\Rightarrow$ (4) by noting that the restriction of $f^*$ to $\Gammaf\C$ is strong monoidal, and thus preserves invertible objects. 

    To complete the proof, we show (4) $\Rightarrow$ (1). By \cref{f*duals} the natural map \[f^*(DM)\otimes f^*M\xrightarrow{\alpha_{M,\1}\otimes f^*M} D(f^*M)\otimes f^*M\] is an isomorphism, and so by the commutative diagram \eqref{diagram:evmonoidal} with $N = \1_\C$, we see that $\mathrm{ev}_{f^*M}$ being an isomorphism implies $f^*(\mrm{ev}_M)$ is one too.
\end{proof}

After \cref{f*rigid}, we discussed its implications in commutative algebra and chromatic homotopy theory. Likewise, we now discuss \cref{f*invertible} in the context of commutative algebra and chromatic homotopy theory.

\begin{ex}
    As in \cref{ex:BIKPRigid}, we consider a commutative Noetherian ring $R$ and prime $\p\in\mrm{Spec}(R)$, taking $f^*$ to be extension of scalars along $R_\p\to\kappa(\p)$. The statement of \cite[Proposition 4.4(3)]{BIKP} shows that if $\kappa(\p)\otimes_{R_\p}X\simeq\Sigma^a\kappa(\p)$ for some $X\in\msf{D}(R_\p)$ and $a\in\Z$, then $\Lambda_{\kappa(\p)} X\simeq\Sigma^a\Lambda_{\kappa(\p)}R_\p$. Since $\msf{D}(\kappa(\p))$ has trivial Picard group, this corresponds to the implication (4) $\Rightarrow$ (3) above. In fact the Picard group of complete objects in this case is trivial itself: this is an instance of \emph{automatic orientability} which we study below in \cref{section:orientations}.
\end{ex}

\begin{ex}\label{ex:GrossHopkins}
    We continue the example of $K(n)$-local spectra as in \cref{ex:Knrigid}. Hopkins--Mahowald--Sadofsky~\cite[Theorem 1.3]{HMS} show that a $K(n)$-local spectrum $X$ is invertible if and only if its $K(n)$-homology $K(n)_*X$ is one-dimensional over $K(n)_*$, also see the alternative proof of Hovey--Strickland~\cite[Theorem 14.2]{HS666}. The equivalence of (3) and (4) in \cref{f*invertible} applied to the geometric functor $K(n)_*(-)\colon h(L_n\Sp) \to \mathrm{grMod}(K(n)_*)$ recovers this characterisation (at least for odd primes, see \cref{ex:Knrigid}).
\end{ex}

\begin{rem}\label{rem:Picfactorisation}
    By the strong symmetric monoidality of the MGM equivalence \cref{monoidalcompletetorsion} it is immediate that $\Pic(\Gammaf\C)\simeq\Pic(\Lambdaf\C)$. However, \cref{f*invertible} does not imply that these groups are isomorphic to $\Pic(\D)$. By \cref{prop:changeofbase}, $f^*\colon \C \to \D$ factors as 
    \[\C \xrightarrow{\Gammaf} \Gammaf\C \xrightarrow{f^*_t} \D\] where $f^*_t$ is the restriction of $f^*$ to $\Gammaf\C$. Note that $f^*_t$ is again strong symmetric monoidal, so we obtain a corresponding factorisation of the induced map on Picard groups
\[\begin{tikzcd}[row sep=1cm, column sep=1cm]
	{\Pic(\C)} \\
	{\Pic(\Gammaf\C)} & {\Pic(\D).}
	\arrow["\Gammaf"', from=1-1, to=2-1]
	\arrow["f^*", from=1-1, to=2-2]
	\arrow["f^*_t"', from=2-1, to=2-2]
\end{tikzcd}\]
The vertical map may be far from surjective, and the horizontal map is not injective in general. The latter is discussed in more detail in \cref{section:orientations} from the perspective of Morita theory, in particular see \cref{rem:kernel}. 

As an illustrative example, we consider the case of \cref{ex:Knrigid,ex:GrossHopkins} for $n \geq 1$, where the above diagram corresponds to
\[\begin{tikzcd}[row sep=1cm, column sep=1cm]
	{\Pic(L_n\Sp)} \\
	{\Pic(L_{K(n)}\Sp)} & {\Pic(\Mod{}(K(n)_*)).}
	\arrow[from=1-1, to=2-1]
	\arrow[from=1-1, to=2-2]
	\arrow[from=2-1, to=2-2]
\end{tikzcd}\]
It is clear that $\Pic(\Mod{}(K(n)_*)) = \Z/2(p^n-1)$ as $K(n)_* = \mathbb{F}_p[v_n^{\pm 1}]$ with $|v_n|=2(p^n-1)$.
The horizontal map has kernel an infinite pro-$p$-group, see \cite[Proposition 14.3]{HS666} and \cite[\S 9]{HMS}, and hence is far from injective. Moreover, if $2p-2>n^2+n$ then $\Pic(L_n\Sp) = \Z$ by \cite{HoveySadofsky}, and we see that the vertical map is not surjective from the description of $\Pic(L_{K(n)}\Sp)$ given in \cite[Theorem B]{BSSW}. In particular, we see that it is possible for both $\C$ and $\D$ have trivial Picard group in the sense of \cref{defn:trivialpicard}, but $\Gammaf\C$ to not have trivial Picard group.
\end{rem}

\begin{rem}\label{rem:reflexivefalse}
    Both \cref{f*rigid} and \cref{f*invertible} are results of the form ``$\Gammaf X$ has property $P$ in $\Gammaf\C$ if and only if $f^*X$ has property $P$ in $\D$'', where $P$ is some property of objects in tt-categories. It is therefore natural to ask for which other properties this statement holds. In particular, one can ask about reflexivity, but here the answer is negative: 
    \begin{enumerate}[label=(\roman*)]
        \item We first show that there exists a proxy-small geometric functor $f^*\colon \C \to \D$ and an object $X \in \C$ for which $f^*X\in\D$ is reflexive but $\Gammaf X \in \Gammaf\C$ is not. Indeed, take any commutative Noetherian local ring $(R,\m,k)$ which is not Gorenstein, and let $f^*\colon \msf{D}(R) \to \msf{D}(k)$ be extension of scalars along $R\to k$. For simplicity, we shall assume $R$ is Artinian, so that $\Gammaf\msf{D}(R)=\msf{D}(R)$. Concretely one could take $R=k[x,y]/(x^2,y^2,xy)$.
        Although $\mrm{Tor}^R_*(k,k)$ is infinite dimensional, it is finite in each degree, and so $k\otimes_R k=f^*k$ is reflexive in $\msf{D}(k)$. So it remains to show that $k$ is not reflexive in $\msf{D}(R)$. Indeed since $R$ is not Gorenstein, $Dk = \Hom_R(k,R)$ has infinite homology over $k$, and so as an $R$-module is an infinite direct sum of suspensions of $k$, say $Dk\simeq\bigoplus_i\Sigma^{d_i}k$. Dualising again we see that $D^2k\simeq\prod_i\bigoplus_j\Sigma^{d_j-d_i}k$, which is not isomorphic to $k$ itself.
        We note that the same argument applies without the Artinian assumption (since $Dk$ and $D^2k$ are automatically torsion) and in fact the reflexivity of $k$ characterises Gorenstein local rings. An alternative, more general argument for this can be found below in \cref{cor-Goriffreflexive}.
        \item We show that there exists a proxy-small geometric functor $f^*\colon \C \to \D$ and an object $X \in \C$ for which $\Gammaf X\in\Gammaf \C$ is reflexive but $f^*X\in\D$ is not in \cref{ex:reflexivetot}. 
    \end{enumerate}
\end{rem}

\section{Matlis dualising objects and the Gorenstein property}
In this section, we propose a notion of a Gorenstein geometric functor generalising the Gorenstein ring spectra of Dwyer--Greenlees--Iyengar~\cite{DGI}. Using this we introduce the notion of Matlis dualising objects, and show that they satisfy an abstract form of Matlis duality. We characterise these Matlis dualising objects in terms of the Picard group of torsion objects, and we also relate Matlis dualising objects to external dualising objects as studied in Balmer--Dell'Ambrogio--Sanders~\cite{BDS}.

\subsection{Gorenstein geometric functors}
Recall from \cref{chunk:geometric} that a geometric functor $f^*\colon \C \to \D$ fits into an adjoint triple $f^* \dashv f_* \dashv f^!$. First we make one of the central definitions of this section. We motivate this via the following remark:
\begin{rem}\label{rem:Gorenstein}
    We recall the notion of Gorenstein ring spectra as introduced by Dwyer--Greenlees--Iyengar~\cite{DGI}. A proxy-small map $R\to k$ of commutative ring spectra is called \emph{Gorenstein} if there is an equivalence $\Hom_R(k,R)\simeq\Sigma^{a}k$ for some integer $a$ (the \emph{shift}). For the principal examples of $k$ in loc. cit., namely fields and the sphere spectrum, the Picard group of $k$ is trivial (see \cref{ex:trivialPic}), and so in this case we can fully describe this property in our setting. Namely, let $f^*\colon\msf{D}(R)\to\msf{D}(k)$ be the extension of scalars functor. Provided $\Pic(\msf{D}(k))$ is trivial, the Gorenstein property is equivalent to $\omega_f = f^!R$ being in $\Pic(\msf{D}(k))$. This naturally leads to a generalisation of this Gorenstein property in tt-geometry as given in the next definition. For instance, in the realm of equivariant ring spectra, one does not expect integer shifts, for example see \cref{ex:KR} below. Despite the generalisation, we will demonstrate that this definition retains the distinctive features of Gorenstein ring spectra whilst being flexible enough to cover a wide range of settings.
\end{rem}

\begin{defn}\label{defn:Gorenstein}
    A geometric functor $f^*\colon\C\to\D$ is \emph{Gorenstein} if $\omega_f\in\Pic(\D)$.
\end{defn}

\begin{exs}\label{exs:Gorenstein}
    As a sanity check, we see that the definition above is compatible with the standard definitions in commutative algebra and algebraic geometry:
    \begin{enumerate}[label=(\roman*)]
        \item Let $(R,\m,k)$ be a commutative Noetherian local ring and consider the extension of scalars functor $f^*\colon \msf{D}(R) \to \msf{D}(k)$ along $R \to k$. Then $R$ is Gorenstein if and only if the geometric functor $f^*$ is Gorenstein. Indeed, the Picard group of $\msf{D}(k)$ is trivial, so $f^*$ is Gorenstein if and only if $\omega_f = \Hom_R(k,R)$ is a shift of $k$ which is equivalent to $\mrm{Ext}^*_R(k,R)$ being $k$ in a single degree, that is, to $R$ being a Gorenstein ring.
        \item Let $f\colon X \to \mrm{Spec}(k)$ be a projective variety over a field $k$. We show that $X$ is Gorenstein if and only if the pullback functor $f^*\colon \msf{D}(k) \to \msf{D}_\mrm{qc}(X)$ is a Gorenstein geometric functor. By~\cite[Example 6.14]{BDS}, $\omega_f$ is a dualising complex for $X$ (in the sense that $\iHom(-,\omega_f)$ gives an antiequivalence on $\msf{D}^b(\mrm{coh}(X))$), and $X$ is Gorenstein if and only if $X$ has an invertible dualising complex. The reverse direction is then immediate from this. For the converse, if $X$ is Gorenstein then the structure sheaf is a dualising complex (\cite[p299]{Hartshorne}) so by the uniqueness of dualising complexes (\cite[Lemma 3.9]{Neeman}) we have that $\omega_f \simeq \mc{O}_X \otimes l$ for some invertible $l$, so that $\omega_f$ is invertible as required. 
    \end{enumerate}
\end{exs}

\begin{rem}
    Similarly to \cref{rem:calg}, we say that a map of commutative algebras $f\colon R \to S$ is Gorenstein if the associated geometric functor given by extension of scalars is Gorenstein.
\end{rem}

This definition also allows for more exotic notions of Gorenstein, for example as a condition for ring $G$-spectra.
\begin{ex}\label{ex:KR}
    Let $k\mbb{R}$ be the connective version of Atiyah's Real $K$-theory. This is a commutative algebra in $\Sp_Q$, where $Q$ is of order 2. Dugger~\cite[Theorem 1.5]{DuggerkR} shows that there is a cofibre sequence \[\Sigma^{1+\sigma} k\mbb{R}\xrightarrow{\bar{v}}k\mbb{R}\xrightarrow{f} H\underline{\Z}\] where $\sigma\in RO(Q)$ is the non-trivial one-dimensional representation and $H\underline{\Z}$ is the constant Mackey functor. In the language of this paper, the above cofibre sequence shows that $\omega_f\coloneqq f^{!}(k\mbb{R}) \simeq \Sigma^{-2-\sigma}H\underline{\Z}$, which is an invertible $H\underline{\Z}$-module, and hence $f\colon k\mbb{R} \to H\underline{\Z}$ is Gorenstein. Similar examples for topological modular forms with level structures may be found in~\cite{Dimitar}.
\end{ex}

\begin{rem}
    Let $f\colon R \to k$ be a map of commutative algebras in $G$-spectra. In analogy with the definition of Dwyer--Greenlees--Iyengar for ring spectra as discussed in \cref{rem:Gorenstein}, another definition of $f$ being Gorenstein that one might take is that $\omega_f = \Hom_R(k,R) \simeq \Sigma^\alpha k$ for some $\alpha \in RO(G)$. Note that this is a more restrictive definition then $f$ being Gorenstein in our sense (\cref{defn:Gorenstein}). Whilst this would cover the case of the previous example, there may be more invertible objects in $\Mod{\Sp_G}(k)$ than just the shifts of $k$ by the representation spheres, as discussed in \cref{exs:nontrivialPic}. The framework developed in this paper suggests that allowing all invertible objects rather than just the representation spheres still allows for powerful duality results under a less restrictive definition of Gorenstein. 
\end{rem}

\subsection{Matlis dualising objects}
In algebra, Gorenstein rings are precisely those for which the ring itself is a dualising module, but it is also fruitful to study dualising complexes over more general rings. Motivated by this, we abstract the notion of Gorenstein geometric functors to that of Matlis dualising objects.

\begin{defn}\label{defn:dualising}
    Let $f^*\colon \C \to \D$ be a geometric functor. An object $M \in \C$ is \emph{Matlis dualising} if $f^{!}M \in \Pic(\D)$.
\end{defn}

\begin{ex}\label{ex:injhull}
    Let $(R,\m,k)$ be a commutative Noetherian local ring, and let $f^*\colon \msf{D}(R) \to \msf{D}(k)$ be extension of scalars along $R \to k$. Then the injective hull $E(k)$ of the residue field is a Matlis dualising object as $f^!(E(k)) = \Hom_R(k,E(k)) \simeq k$. If $R$ is a $k$-algebra, recall that $E(k) \simeq \Hom_k(R,k)$, so that $\Hom_k(R,k)$ is a Matlis dualising object. More generally, the existence of a section to a geometric functor $f^*$ always yields a Matlis dualising object, and this will play a key role in the theory of Gorenstein duality as demonstrated in \cref{sec:Gorduality}.
\end{ex}

\begin{chunk}
    In order to prove the main result of this section, we first need some auxiliary results. Let $f^*\colon \C \to \D$ be a geometric functor and $M \in \C$. We write $D_M \coloneqq \iHom(-,M)$ for the $M$-dual, and note that there is a natural transformation \[\phi^M_X\colon X \to D^2_M(X)\] for all $X \in \C$, defined as the adjunct of the evaluation map $\mrm{ev}_{X,M}\colon X \otimes \iHom(X,M) \to M$. Likewise, we define $\psi^M_Y\colon Y \to D^2_{f^{!}M}Y$ for all $Y \in \D$. These fit into the following commutative diagram by \cite[Theorem 5.15]{BDS}:
    \begin{equation}\label{diagram:flowerdual}
    \begin{tikzcd}[column sep=1cm, row sep=1cm]
        f_*Y \ar[r,"f_*\psi^M_Y"] \ar[d, "\phi^M_{f_*Y}"'] & f_*D_{f^{!}M}^2(Y) \ar[d, "\simeq"] \\
        D_M^2(f_*Y) \ar[r, "\simeq"'] & D_M f_* D_{f^{!}M}(Y)
    \end{tikzcd}
    \end{equation}
    where the indicated natural isomorphisms come from \cref{internaladjunction2}.
\end{chunk}
\begin{prop}\label{Mduality}
    Let $f^*\colon \C \to \D$ be a geometric functor, and let $M \in \C$ be Matlis dualising. Then the natural map $\phi_X^M\colon X \to D^2_M(X)$ is an isomorphism for all $X \in \thicktensor{f_*(\D^c)}$. Hence, there is an equivalence of categories
    \[D_M\colon \thicktensor{f_*(\D^c)}^\op \xrightarrow{\simeq} \thicktensor{f_*(\D^c)}.\]
\end{prop}
\begin{proof}
    Let $Y \in \D^c$. By considering the commutative diagram \eqref{diagram:flowerdual}, we see that $\phi^M_{f_*Y}$ is a natural isomorphism if and only if $f_*\psi_Y^{M}$ is an isomorphism. Since $Y$ is rigid and $f^!M$ is invertible, the latter is an isomorphism, and hence $\phi^M_{f_*Y}$ is an isomorphism.
    If $C \in \C^c$, then for all $X \in \C$, the map $\varphi^M_{C \otimes X}$ may be identified (up to natural isomorphism) with $C \otimes \varphi^M_X$. As such, the collection of $X$ for which $\varphi^M_X$ is an isomorphism is a thick $\otimes$-ideal, and hence $\phi^M_X$ is an isomorphism for all $X \in \thicktensor{f_*(\D^c)}$. 

    To conclude the equivalence of categories it suffices to show that if $X \in \thicktensor{f_*(\D^c)}$ then $D_M(X) \in \thicktensor{f_*(\D^c)}$. By a thick subcategory argument we reduce to showing that if $Y \in \D^c$ then $D_M(f_*Y) \in \thicktensor{f_*(\D^c)}$. By \cref{internaladjunction2} we have
    \[\iHom(f_*Y, M) \simeq f_*\iHom(Y, f^!M) \simeq f_*(DY \otimes f^!M).\] Since $M$ is Matlis dualising, $f^!M$ is rigid, so we conclude.
\end{proof}

\begin{ex}\label{ex:Matlisdualitycomplexes}
    Let $(R,\m,k)$ be a commutative Noetherian local ring, and let $f^*\colon \msf{D}(R) \to \msf{D}(k)$ be the extension of scalars along $R \to k$. Then the injective hull $E(k)$ is a Matlis dualising object by \cref{ex:injhull} and the previous result recovers Matlis duality for complexes, see \cref{ex:injectivehull} for more details.
\end{ex}

\begin{rem}
    The previous result provides one justification for the terminology of Matlis \emph{dualising}: such objects indeed provide an antiequivalence on a suitable full subcategory. Using the previous result, we will compare our notion of Matlis dualising objects with the external dualising objects of \cite{BDS} in \cref{sec:external}, see \cref{prop-Matlisdualvsextdual} and \cref{rem:compareBDS} for more details.
\end{rem}

\begin{cor}\label{f*Mduality}
    Let $f^*\colon \C \to \D$ be a proxy-small geometric functor, and let $M \in \C$ be Matlis dualising. Then the natural map $\phi^M_X\colon X \to D^2_M(X)$ is an $f^*$-isomorphism for all $X \in \C^c$.
\end{cor}
\begin{proof}
    Write $\mc{W}$ for a witness for the proxy-smallness of $f_*\1_\D$. By \cref{Mduality}, $\varphi^M_W$ is an isomorphism for all $W\in\mc{W} \subseteq \thicktensor{f_*\1_\D}$. For $X \in \C^c$, we have that $W \otimes \varphi^M_{X} = \varphi^M_{W \otimes X} = X \otimes \varphi^M_{W}$. Since $\varphi^M_{W}$ is an isomorphism, we see that $W \otimes \varphi^M_{X}$ is an isomorphism for all $W \in \W$, and hence $\varphi^M_X$ is an $f^*$-isomorphism by \cref{Wequivalences}.
\end{proof}

\subsection{Characterising Matlis dualising objects} 
In this subsection we prove the main result of this section, characterising Matlis dualising objects for Gorenstein geometric functors in terms of the Picard group of complete objects. In order to state this, we need to recall some background on pure semisimple triangulated categories.
\begin{chunk}\label{puresemisimple}
    A compactly generated triangulated category $\D$ is said to be \emph{pure semisimple} if every object of $\D$ can be written as a coproduct of compact objects, equivalently, every object in $\D$ is pure injective. In this case, the full subcategory $\D^c$ is a Krull-Schmidt category so that any compact object of $\D$ admits a decomposition into a finite coproduct of indecomposables which is unique up to permutation. For the purposes of this paper, it is convenient to think of pure semisimple triangulated categories as being like residue fields; indeed any tt-field is pure semisimple by definition. This fits into the general philosophy of thinking of proxy-small geometric functors $f^*\colon \C \to \D$ as being passage to residue fields. However, the definition is a little more flexible than this: for example, the derived categories of finite products of fields and the stable module categories of group algebras of finite representation type are pure semisimple~\cite[Corollaries 12.5 and 12.26]{Beil}. The following property is the key feature of pure semisimplicity which we require:
\end{chunk}

\begin{lem}\label{lem:dualcompact}
    Let $\D$ be a pure semisimple, compactly generated tt-category and let $X \in \D$. If $DX$ is compact, then $X$ is compact. 
\end{lem}
\begin{proof}
    Since $\D$ is pure semisimple we may write $X \simeq \oplus_I C_i$ as a coproduct of compact objects. Since $DX \simeq \prod_I D(C_i)$ is a product of non-zero compact objects, we see that $I$ must be finite as $\D^c$ is a Krull-Schmidt category.
\end{proof}

\begin{rem}
    The previous result fails for general tt-categories: for example, in the stable homotopy category, Lin's theorem shows that $D(\mbb{F}_p) \simeq 0$ hence compact, whereas $\mbb{F}_p$ is not compact.
\end{rem}

We can now state the main result of this section. Recall from \cref{defn:Gorenstein} that $f^*\colon \C \to \D$ is Gorenstein if $\omega_f$ is invertible in $\D$.
    \begin{thm}\label{thm:dualising}
    Let $f^*\colon \C \to \D$ be a proxy-small geometric functor and assume that $\D$ is pure semisimple. If $f^*$ is Gorenstein, the following are equivalent for $M \in \C$:
    \begin{enumerate}
        \item $M$ is Matlis dualising (i.e., $f^{!}M \in \Pic(\D)$);
        \item $f^*M \in \Pic(\D)$;
        \item $\Gammaf M \in \Pic(\Gammaf\C)$;
        \item $\Lambdaf M \in \Pic(\Lambdaf\C)$.
    \end{enumerate}
\end{thm}
\begin{proof}
Conditions (2)-(4) are equivalent by \cref{f*invertible}. For (2) implies (1), since $f^*M$ is invertible and hence rigid, we have that $\omega_f \otimes f^*M \simeq f^{!}M$ by Grothendieck duality \cref{GN}. Therefore, since both $\omega_f$ and $f^*M$ are invertible, we deduce that $f^{!}M$ is also invertible.

Conversely, suppose that (1) holds, i.e., $f^{!}M$ is invertible. By \cref{f*Mduality} together with \cref{Wequivalences}, we have that $\omega_f = f^{!}\1_\C \to f^{!}D^2_M(\1_\C) = f^{!}\iHom(M,M)$ is an isomorphism. We therefore have \[\omega_f \simeq f^{!}\iHom(M,M) \simeq \iHom(f^*M, f^{!}M) \simeq D(f^*M) \otimes f^{!}M\] where the second isomorphism uses \cref{f1ofhom}, and the third uses that $f^{!}M$ is invertible and hence rigid. By the 2-out-of-3 property for invertible objects, we have that $D(f^*M) \in \Pic(\D)$. In particular $D(f^*M)$ is rigid, and so $f^*M$ is rigid by \cref{lem:dualcompact}. Thus Grothendieck duality $f^!M\simeq \omega_f \otimes f^*M$ holds \cref{GN} and $f^*M$ is invertible by the 2-out-of-3 property of invertible objects.
\end{proof}

\begin{rem}
    The proof of the previous theorem shows that the implications
    \[(2) \Rightarrow (1) \Rightarrow D(f^*M) \in \Pic(\D)\]
    hold without the hypothesis that $\D$ is pure semisimple.
\end{rem}

\subsection{Relation to external dualising objects}\label{sec:external}
In this section we relate Matlis dualising objects to the external dualising objects of \cite{BDS}.
\begin{chunk}
    In \cite[Definition 5.1 and \S 7]{BDS}, an object $M \in \C$ is said to be an \emph{external dualising object} for $\C_0 \subseteq \C$ if and only if $D_M \coloneqq \iHom(-,M)\colon \C_0^\op \xrightarrow{\sim} \C_0$ is an equivalence of categories. Note that the word ``external'' refers to the fact that $M$ can lie outside of the subcategory $\C_0$ which it dualises. \cref{Mduality} shows that if $M \in \C$ is a Matlis dualising object, then it is an external dualising object for $\thicktensor{f_*(\D^c)}$. We now investigate when the converse implication holds. The following results show that Matlis dualising objects are a tightening of the notion of external dualising objects. Moreover, the two notions agree in the typical motivating cases; that is, when the geometric functor $f^*\colon \C \to \D$ is given by passage to some sort of residue field, i.e., when $\D$ is pure semisimple, see \cref{cor-Matlisdualvsextdual}.
\end{chunk}

\begin{prop}\label{prop-Matlisdualvsextdual}
    Let $f^*\colon \C \to \D$ be a geometric functor and $M \in \C$. Suppose that $f_*$ is conservative. Then the following are equivalent:
    \begin{enumerate}
        \item $M$ is Matlis dualising;
        \item $f^!M$ is rigid and $M$ is an external dualising object for $\thicktensor{f_*(\D^c)}$;
        \item $f^{!}M$ is rigid and $\psi_{\1_\D}^{M}\colon \1_\D \to \iHom(f^{!}M, f^{!}M)$ is an isomorphism.
    \end{enumerate}
\end{prop}
\begin{proof}
    The implication (1) $\Rightarrow$ (2) holds by \cref{Mduality}. For (2) $\Rightarrow$ (3), by assumption $\phi^M_X\colon X \to D^2_M(X)$ is an isomorphism for all $X \in \thicktensor{f_*(\D^c)}$, in particular, for $X = f_*\1_\D$. By the commutative diagram \eqref{diagram:flowerdual}, we therefore deduce that \[f_*(\psi_{\1_\D}^{M})\colon f_*\1_\D \to f_*\iHom(f^{!}M, f^{!}M)\] is an isomorphism. We then conclude by conservativity of $f_*$. For (3) $\Rightarrow$ (1),
    as $f^{!}M$ is rigid, we have that $D(f^{!}M) \otimes f^{!}M \simeq \1_\D$, so $M$ is Matlis dualising.
\end{proof}

We can use the above result to provide a generalisation of \cite[Theorem 1.9, Proposition 4.1]{BDS} which does not require Grothendieck--Neeman duality.
\begin{cor}\label{cor-Goriff}
    Let $f^*\colon \C \to \D$ be a geometric functor and suppose that $f_*$ is conservative. Then $f^*$ is Gorenstein (i.e., $\omega_f$ is invertible) if and only if $\omega_f$ is compact and $f_*\1_\D$ is reflexive. 
\end{cor}
\begin{proof}
    Take $M = \1_\C$ in  \cref{prop-Matlisdualvsextdual}.
\end{proof}

\begin{ex}\label{ex:notequivalent}
     We now show that \cref{prop-Matlisdualvsextdual} is sharp, in that there are external dualising objects which are not Matlis dualising. Take $\C = \msf{D}(R)$ to be the derived category of a complete local Noetherian ring $(R,\m,k)$ which is not Artinian, and let $f^*$ be the identity functor. By a thick subcategory argument, one sees that $M \in \msf{D}(R)$ is an external dualising object for $\thicktensor{\msf{D}(R)^c} = \msf{D}(R)^c$ if and only if $\Hom_R(M,M) \simeq R$. For example, the injective hull $E(k)$ of the residue field is an external dualising object, but it is not compact (if it were it would be finitely generated, but there are no finitely generated injectives over non-Artinian rings) and hence not invertible, so it is not Matlis dualising. 
\end{ex}

Despite the previous example, in our motivating setting, i.e., when $\D$ is pure semisimple, we see that Matlis dualising and external dualising objects do agree.
\begin{cor}\label{cor-Matlisdualvsextdual}
    Let $f^*\colon \C \to \D$ be a geometric functor. Suppose that $\D$ is pure semisimple and $f_*$ is conservative. Then $M$ is Matlis dualising if and only if it is an external dualising object for $\thicktensor{f_*(\D^c)}$.
\end{cor}
\begin{proof}
    By \cref{prop-Matlisdualvsextdual} it suffices to show that if $M$ is an external dualising object for $\thicktensor{f_*(\D^c)}$ then $f^{!}M$ is rigid. 
    Since $\D$ is pure-semisimple, $f^{!}M$ can be written as a coproduct of compact objects $\oplus_{i \in I}A_i$. Since $M$ is an external dualising object, $\1_\D \simeq \iHom(f^!M, f^!M)$ and so we have \[\1_\D \simeq \prod_{i \in I} \bigoplus_{j \in I} \iHom(A_i, A_j).\] 
    Since $\D^c$ is Krull-Schmidt, the indexing set $I$ must be finite, and hence $f^{!}M$ must be compact as required. 
\end{proof}

We also obtain the following refinement of \cref{cor-Goriff}. 
\begin{cor}\label{cor-Goriffreflexive}
    Let $f^*\colon \C \to \D$ a geometric functor. Suppose that $\D$ is pure semisimple and $f_*$ is conservative. Then $f^*$ is Gorenstein if and only if $f_*\1_\D$ is reflexive.
\end{cor}
\begin{proof}
    This is the case $M = \1_\C$ of \cref{cor-Matlisdualvsextdual}.
\end{proof}

\begin{rem}
    The trichotomy theorem of Balmer--Dell'Ambrogio--Sanders \cite[Corollary 1.13]{BDS} states that a geometric functor $f^*\colon \C \to \D$ fits into an adjoint triple, an adjoint quintuple, or an infinite chain of adjoints. The previous result shows that if $\D$ is pure semisimple and $f_*$ is conservative, this collapses to a \emph{di}chotomy theorem: either $f^*$ fits into an adjoint triple, or an infinite chain of adjoints. Indeed, the existence of the adjoint quintuple is equivalent to $f_*$ preserving compacts~\cite[Theorem 1.7]{BDS}, and if this holds, then $f_*\1_\D$ is rigid, hence reflexive. Therefore by the previous corollary, $f^*$ is Gorenstein (i.e., $\omega_f$ is invertible), and hence there are in fact infinitely many adjoints by \cite[Theorem 1.9]{BDS}.
    More concretely, the previous corollary can also be viewed as a general form of the fact that regular rings are Gorenstein. Specialising to the case of ring spectra, this shows that the hypotheses on connectivity and dc-completeness in \cite[Proposition 8.13]{DGI} can be removed.
\end{rem}

We now connect our results with \cite[Theorem 7.1]{BDS}.
\begin{rem}\label{rem:compareBDS}
In loc. cit., given a geometric functor $f^*\colon \C \to \D$ they pick an object $Y$ in $\D$ which dualises the compacts (i.e., such that $X \to D^2_YX$ is an isomorphism for all $X \in \D^c$), and prove that any \emph{Matlis lift} $M$ of $Y$ (that is, any $M \in \C$ with $f^!M \simeq Y$), is an external dualising object for $\thicktensor{f_*(\D^c)}$. By considering \cref{prop-Matlisdualvsextdual} in the case when $f^*$ is the identity functor on $\D$, we deduce that such a $Y$ is precisely an element of the Picard group of $\D$. We therefore see that our notion of Matlis dualising objects is compatible with the Matlis lifts of \cite{BDS}.
\end{rem}

\section{Orientations and Morita theory}\label{section:orientations}
In this section, given a proxy-small geometric functor $f^*\colon \C \to \D$ we classify the Matlis lifts of objects of $\D$ via Morita theory, building on results of Dwyer--Greenlees--Iyengar in the setting of ring spectra~\cite{DGI}. We single out a particular class of Matlis lifts of $\omega_f$ called the orientable Matlis lifts which satisfy a local duality theorem and then provide conditions for all Matlis lifts to be orientable.
\subsection{Orientability and local duality}
In this subsection we introduce the key terminology of orientable Matlis lifts and show that they satisfy a local duality theorem by construction. Recall from \cref{chunk:geometric} that a geometric functor $f^*\colon \C \to \D$ fits into an adjoint triple $f^* \dashv f_* \dashv f^!$.
\begin{defn}\label{defn:Matlislift}
    Let $f^*\colon \C \to \D$ be a geometric functor, and let $Y \in \D$. We say that an object $X \in \C$ is a \emph{Matlis lift} of $Y$ if $f^!X \simeq Y$.
\end{defn}

The object $\omega_f \coloneqq f^!(\1_\C)$ plays a distinguished role (e.g., through Grothendieck duality), so we are particularly interested in Matlis lifts of $\omega_f$. Moreover, if $f^*$ is Gorenstein then any Matlis lift of $\omega_f$ is a Matlis dualising object. The following definition singles out a particular class of these Matlis lifts of interest.
\begin{defn}\label{defn:orientableMatlislift}
    Let $f^*\colon \C \to \D$ be a proxy-small geometric functor, and let $M$ be a Matlis lift of $\omega_f$. We say that $M$ is an \emph{orientable Matlis lift of $\omega_f$} if $\Gammaf M \simeq \Gammaf\1_\C$. 
\end{defn}

\begin{rem}
We warn the reader that \cref{Wequivalences} does \emph{not} imply that any Matlis lift of $\omega_f$ is orientable. \cref{Wequivalences} ensures that a morphism $X\to Y$ is an $f^!$-isomorphism if and only if it is a $\Gammaf$-isomorphism, however here we only have an abstract isomorphism $f^!M \simeq \omega_f$ which need not be the image under $f^!$ of a morphism in $\C$.
\end{rem}

By construction, orientable Matlis lifts satisfy an abstract version of Grothendieck local duality:
\begin{prop}\label{Matlisliftofomega}
    Let $f^*\colon \C \to \D$ be a proxy-small geometric functor. If $M$ is an orientable Matlis lift of $\omega_f$, then for all $X \in \C$, there is an isomorphism
    \[\iHom(\Gammaf X,M) \simeq \iHom(X,\Lambdaf\1_\C).\]
\end{prop}
\begin{proof}
    Since $M$ is an orientable Matlis lift of $\omega_f$, we have that $\Lambdaf M \simeq \Lambdaf \1_\C$ by the MGM equivalence. Therefore \[\iHom(\Gammaf X, M) \simeq \iHom(X, \Lambdaf M) \simeq \iHom(X, \Lambdaf \1_\C)\] for all $X \in \C$ as claimed.
\end{proof}

\begin{ex}
    In the context of local algebra, the previous result yields the usual statement of Grothendieck local duality for Gorenstein local rings, that for any $R$-module $X$, \[\Hom_R(H_\m^{\mrm{dim}(R)-i}(X), E(k)) \simeq \mrm{Ext}^i_R(X, R_\m^\wedge),\] see \cref{ex:injectivehull} for more details.
\end{ex}

In the following subsection we investigate when Matlis lifts of $\omega_f$ are automatically orientable in terms of Morita theory, see \cref{cor:AOandMorita}.

\subsection{Matlis lifts via Morita theory}
A Matlis lift $X$ of an object $Y\in\D$, should be considered as additional structure on $Y$. By \cref{internaladjunction2}, there is an equivalence $f_*f^!X\simeq\iHom(f_*\1_\D,X)$, giving $f_*Y \simeq f_*f^!X$ a right action by the endomorphisms of $f_*\1_\D$ in $\C$. The rest of this section concerns how these right actions detect and classify Matlis lifts. This is inspired by \cite[\S 6]{DGI}; however, one subtlety in this tt-context as opposed to \cite[\S6]{DGI}, is that $f_*$ is not necessarily conservative. Therefore, we should not expect endomorphisms of $f_*\1_\D$ to detect anything that can't be detected by $f_*$ itself, and so we require a weaker notion of Matlis lifting. For good behaviour of module categories, we will often assume that $f^*$ is enhanced throughout this section. Note that since $f^*$ is strong symmetric monoidal, its right adjoint $f_*$ is lax symmetric monoidal, and hence any object in the image of $f_*$ admits a canonical $f_*\1_\D$-module structure.

\begin{defn}
    Let $f^*\colon \C \to \D$ be an enhanced geometric functor and $Y \in \D$. A \emph{weak Matlis lift} of $Y$ is an object $X\in\C$ such that there is an equivalence $f_*f^!X \simeq f_*Y$ of $f_*\1_\D$-modules.
\end{defn}

\begin{rem}\label{rem:strongvsweaklifts}
    If $f_*$ is conservative, then it induces an equivalence $\Mod{\C}(f_*\1_\D)\simeq\D$, see \cref{cor:conservative}. Therefore, in this case being a weak Matlis lift is identical to being a ``strong'' Matlis lift as in \cref{defn:Matlislift}.
\end{rem}

\begin{chunk}\label{endoaction}
    Consider the endomorphism algebra $\Ef\coloneqq \iHom(f_*\1_\D, f_*\1_\D)$. This is an algebra internal to $\C$, and the natural action of $f_*\1_\D$ on itself gives a morphism $f_*\1_\D\to \Ef$ of algebras in $\C$. Furthermore, the endomorphism action gives $f_*\1_\D$ a natural left $\Ef$-module structure.
\end{chunk}

\begin{ex}\label{ex:EM}
    Let $R$ be a commutative ring spectrum and $X$ a based topological space. Let $C^*(X;R)$ be the function spectrum $F(\Sigma^\infty X, R)$, which is a commutative ring spectrum via the diagonal on $X$. Let $f^*$ be extension of scalars along $f\colon C^*(X;R)\to R$. Provided the $R$-based Eilenberg--Moore spectral sequence for $X$ converges, $\Ef=\Hom_{C^*(X;R)}(R,R)\simeq C_*(\Omega X;R)$, with multiplication given by loop concatenation.
\end{ex}

\begin{rem}
    Despite the existence of an algebra map $f_*\1_\D\to \Ef$, the endomorphism algebra $\Ef$ does not necessarily come from an algebra internal to $\D$, even if $f_*$ is conservative. For example, let $f^*$ be extension of scalars along the map of commutative ring spectra $ku\to H\Z$: by ~\cite[Theorem 5.14]{IRW} the endomorphism algebra $\Ef = \Hom_{ku}(H\Z,H\Z)$ does not admit an $H\Z$-algebra structure.
\end{rem}

\begin{chunk}
    In general $\Ef$ is not commutative and we write $\Mod{\C}(\Ef)$ for the category of \emph{right} $\Ef$-modules. We note this therefore does not admit a tensor-triangulated structure. As a triangulated category, however, it is compactly generated by the set $\C^c\otimes \Ef=\{X\otimes \Ef:X\in\C^c \}$.
\end{chunk}

\begin{lem}\label{lem:EModTorsion}
    Considered as underlying objects of $\C$, right $\Ef$-modules are $f_*\1_\D$-torsion.
\end{lem}

\begin{proof}
    Since $\C^c\otimes \Ef$ generates $\Mod{\C}(\Ef)$, the underlying object in $\C$ of any $\Ef$-module is in $\Loc(\C^c\otimes \Ef)=\Loctensor{\Ef}$. Thus, it suffices to show that $\Ef$ is $f_*\1_\D$-torsion. Since $\Ef$ is an $f_*\1_\D$-module, we conclude by an application of \cref{flowerconservative}.
\end{proof}

\begin{chunk}\label{chunk:Morita}
We now consider the Morita theory relating $\C$ and $\Mod{\C}(\Ef)$. Given any $X \in \C$, $\iHom(f_*\1_\D,X)$ admits a canonical right $\Ef$-module structure coming from the left $\Ef$-action on $f_*\1_\D$. As such we obtain an adjunction
\[
\begin{tikzcd}[column sep=2.5cm]
	{\Mod{\C}(\Ef)} & \C.
	\arrow["{-\otimes_{\Ef}f_*\1_\D}", shift left, from=1-1, to=1-2]
	\arrow["{\iHom(f_*\1_\D,-)}", shift left, from=1-2, to=1-1]
\end{tikzcd}
\]
By standard category theory, this adjunction restricts to an equivalence between the full subcategories of objects for which the unit and counit are equivalences. We now characterise such objects.
\end{chunk}

We firstly deal with the counit.
\begin{defn}
    An object $X\in\C$ is \emph{effectively constructible} if the counit map of the adjunction in \cref{chunk:Morita} \[\varepsilon_X\colon\iHom(f_*\1_\D,X)\otimes_{\Ef} f_*\1_\D\to X\]
    is an isomorphism.
\end{defn}

\begin{prop}\label{identifycounit}
    Let $f^*\colon \C \to \D$ be a proxy-small, enhanced geometric functor. For all $X \in \C$, the counit map $\varepsilon_X$ is isomorphic to the canonical map $\Gammaf X \to X$. In particular, $X$ is effectively constructible if and only if it is $f_*\1_\D$-torsion.
\end{prop}
\begin{proof}
    It suffices to show that the domain of $\varepsilon_X$ is $f_*\1_\D$-torsion and that $\Gammaf(\varepsilon_X)$ is an isomorphism. For the former, since $\iHom(f_*\1_\D, X) \in \Mod{\C}(\Ef)$, we have $\iHom(f_*\1_\D,X) \in \Loc(\C^c \otimes \Ef)$, so by applying $- \otimes_{\Ef} f_*\1_\D$ we deduce that the domain of $\varepsilon_X$ is in $\Loctensor{f_*\1_\D}$, and hence is $f_*\1_\D$-torsion.

    For the latter, let $\mc{Y}$ be the full subcategory of all $Y\in\C$ for which the natural map \[\mu_Y\colon\iHom(f_*\1_\D,X)\otimes_{\Ef}\iHom(Y,f_*\1_\D)\to\iHom(Y,X)\] defined as the adjunct of the composition $\varepsilon_X\circ\big(\iHom(f_*\1_\D,X)\otimes_{\Ef}\mrm{ev}_{Y,f_*\1_\D}\big)$ is an isomorphism.

    The subcategory $\mc{Y}$ is localising, and for $C\in\C^c$ there are natural isomorphisms identifying $\mu_{C\otimes Y}$ with $\mu_Y \otimes DC$, so $\mc{Y}$ is moreover a localising $\otimes$-ideal. By definition $f_*\1_\D\in\mc{Y}$ and so $\Gammaf\C\subseteq\mc{Y}$. In particular, writing $\W$ for a witness for the proxy-smallness of $f_*\1_\D$, we have $\W\subseteq\mc{Y}$, and hence $W\otimes\mu_{\1_\C} = \mu_{W}$ is an isomorphism for each $W \in \W$. Since $\mu_{\1_\C}=\varepsilon_X$, this shows that $\Gammaf\varepsilon_X$ is an isomorphism by \cref{Wequivalences}.
\end{proof}

Having understood the objects for which the counit map is an isomorphism, we now set terminology for the unit map. Recall from \cref{endoaction} that there is a map of algebras $f_*\1_\D \to \Ef$ in $\C$, and hence a restriction functor $\res^{\Ef}_{f_*\1_\D}\colon \Mod{\C}(\Ef) \to \Mod{\C}(f_*\1_\D)$.
\begin{defn}\label{defn:effective}
    Let $Z\in\Mod{\C}(f_*\1_\D)$. An \emph{$\Ef$-lift} of $Z$ is a right $\Ef$-module $\wt{Z}$ such that $\res^{\Ef}_{f_*\1_\D}(\wt{Z})=Z$. An $\Ef$-lift $\wt{Z}$ of $Z$ is \emph{effective} if the unit \[\eta_{\wt{Z}}\colon\wt{Z}\to\iHom(f_*\1_\D,\wt{Z}\otimes_{\Ef}f_*\1_\D)\]
    of the adjunction of \cref{chunk:Morita} is an isomorphism. More generally, we write $\EffLift_f$ for the full subcategory of $\Mod{\C}(\Ef)$ consisting of the objects $Y$ for which the unit map $\eta_{Y}$ is an isomorphism -  that is, the objects $Y$ which are effective lifts for their restrictions.
\end{defn}

\begin{rem}
    In \cite{DGI} effective lifts are called lifts of \emph{Matlis type}. We change this terminology to avoid confusion with Matlis lifts, and to compare to the notion of effective constructibility.
\end{rem}

By combining the above definitions and \cref{identifycounit} with the adjunction of \cref{chunk:Morita} we obtain the following corollary:
\begin{cor}\label{cor:Moritaequivalence}
    Let $f^*\colon \C \to \D$ be a proxy-small, enhanced geometric functor. There is an equivalence of categories
\[
\begin{tikzcd}[column sep=2.5cm]
	{\EffLift_f} \ar[r, phantom, "\simeq"] & \Gammaf\C.
	\arrow["{-\otimes_{\Ef}f_*\1_\D}", yshift=2mm, from=1-1, to=1-2] 
	\arrow["{\iHom(f_*\1_\D,-)}", yshift=-2mm, from=1-2, to=1-1]
\end{tikzcd}
\] 
\end{cor}

From this we can formally deduce the following classification of weak Matlis lifts in terms of the action of $\Ef$.

\begin{prop}\label{prop:classifylifts}
    Let $f^*\colon \C \to \D$ be a proxy-small, enhanced geometric functor and let $Y \in \D$. There is a bijection:
    \[
    \begin{tikzcd}[ampersand replacement=\&, column sep=4cm]
                {\begin{Bmatrix}
                    \text{effective $\Ef$-lifts}\\
                    \text{of } f_*Y
                \end{Bmatrix}} \ar[r, yshift=1.5mm, "{\wt{f_*Y} \mapsto \wt{f_*Y} \otimes_{\Ef} f_*\1_\D}"] \ar[r, phantom, "\simeq"]
                \&
                {\begin{Bmatrix}
                    \text{torsion weak} \\
                    \text{Matlis lifts of } Y
                \end{Bmatrix}} \ar[l, yshift=-1.5mm, "{ \iHom(f_*\1_\D,X) \mapsfrom X}"]
                \end{tikzcd}
        \]
\end{prop}
\begin{proof}
    By \cref{cor:Moritaequivalence}, it suffices to show that given an effective $\Ef$-lift $\wt{f_*Y}$ of $f_*Y$, we have that $\wt{f_*Y} \otimes_{\Ef} f_*\1_\D$ is a weak Matlis lift of $Y$. Since $\wt{f_*Y}$ is an effective $\Ef$-lift, we have an isomorphism
    \[\wt{f_*Y} \simeq \iHom(f_*\1_\D, \wt{f_*Y} \otimes_{\Ef} f_*\1_\D)\]
    of right $\Ef$-modules. Since the restriction of $\wt{f_*Y}$ to an $f_*\1_\D$-module is precisely $f_*Y$, the above isomorphism together with \cref{internaladjunction2} shows that $\wt{f_*Y} \otimes_{\Ef} f_*\1_\D$ is a weak Matlis lift of $Y$.
\end{proof}

\begin{rem}
    Note that the previous result only classifies \emph{torsion} Matlis lifts. This is indeed all we should expect: since $f^! \simeq f^!\Gammaf$ by \cref{prop:changeofbase,Wequivalences}, given any $X \in \C$, we have that $X$ is a Matlis lift for $Y \in \D$ if and only if $\Gammaf X$ is a Matlis lift for $Y$. 
\end{rem}

The object $\omega_f \coloneqq f^!\1_\C$ plays a distinguished role in the theory of duality in tt-categories. As such, we now specialise to this case and investigate the Matlis lifts of $\omega_f$. Recall the definition of orientable Matlis lifts of $\omega_f$ from \cref{defn:orientableMatlislift}. 
\begin{defn}\label{defn:autoorient}
    Let $f^*\colon \C \to \D$ be a proxy-small, Gorenstein, geometric functor. We say that $f^*$ is \emph{automatically orientable} if every Matlis lift of $\omega_f$ is orientable.
\end{defn}

We can characterise automatically orientable geometric functors in two ways: in terms of Morita theory, and in terms of the Picard group.
\begin{cor}\label{cor:AOandMorita}
    Let $f^*\colon \C \to \D$ be a proxy-small, Gorenstein, enhanced geometric functor with $f_*$ conservative. Then $f^*$ is automatically orientable if and only if there is a unique effective $\Ef$-lift of $f_*(\omega_f)$.
\end{cor}
\begin{proof}
    Firstly, note that since $f_*$ is conservative, Matlis lifts and weak Matlis lifts are the same by \cref{rem:strongvsweaklifts}. Since $\Gammaf\1_\C$ is a torsion Matlis lift of $\omega_f$, we conclude by applying \cref{prop:classifylifts}.
\end{proof}

We now give the characterisation in terms of the Picard group.
\begin{prop}\label{prop:AOandPic}
    Let $f^*\colon \C \to \D$ be a proxy-small, Gorenstein, geometric functor. If $f^*$ is automatically orientable then the group homomorphism \[f^*\colon \Pic(\Gammaf\C) \to \Pic(\D)\] is injective. If $\D$ is moreover pure semisimple then the converse implication also holds.
\end{prop}
\begin{proof}
    If $f^*M \simeq \1_\D$ then by Grothendieck duality \cref{GN}, it follows that $M$ is a Matlis lift of $\omega_f$ so we conclude that the group homomorphism $f^*\colon \Pic(\Gammaf\C) \to \Pic(\D)$ is injective by the definition of orientability. For the converse, we now assume that $\D$ is pure semisimple. Suppose that $M$ is a Matlis lift of $\omega_f$. Since $\omega_f$ is invertible as $f^*$ is Gorenstein, $M$ is therefore a Matlis dualising object, so by \cref{thm:dualising}, we have that $f^*M$ is invertible and hence rigid. Therefore by Grothendieck duality \cref{GN} and the invertibility of $\omega_f$, we have that $f^*M \simeq \1_\D$. Since $f^*\colon \C \to \D$ factors over $\Gammaf\C$ by \cref{prop:changeofbase} the claim follows by injectivity. 
\end{proof}

\begin{rem}\label{rem:kernel}
    One can also connect the two approaches to characterising automatic orientability as follows. Suppose that $f^*\colon \C \to \D$ is a proxy-small, Gorenstein, enhanced geometric functor, where $\D$ is pure semisimple and $f_*$ is conservative. Then the kernel of the group homomorphism $f^*\colon \Pic(\Gammaf\C) \to \Pic(\D)$ is in bijection with the set of effective $\Ef$-lifts of $f_*(\omega_f)$. This follows by combining \cref{thm:dualising} with \cref{prop:classifylifts}.
\end{rem}

\section{Gorenstein duality}\label{sec:Gorduality}
In this section we develop a theory of Gorenstein duality for proxy-small geometric functors. We do this in the setting when the geometric functor $f^*$ admits a geometric section, meaning that it generalises the case when $f^*$ is extension of scalars along a ring map $R \to k$ where $R$ is moreover a $k$-algebra. 

\begin{chunk}\label{chunk:geometricsection}
    More precisely, we say that $f^*\colon \C \to \D$ \emph{admits a geometric section} if there is a geometric functor $\eta^*\colon \D \to \C$ such that $f^*\eta^* \simeq \mrm{id}_\D$. Therefore diagrammatically we have
    \[\begin{tikzcd}[column sep=2cm]
        \D \ar[rr, bend left, "\mrm{id}_\D", yshift=2.5mm]  \ar[r, "\eta^*", yshift=2.5mm] \ar[r, "\eta^{!}"', yshift=-2.5mm] & \C \ar[r, "f^*", yshift=2.5mm] \ar[r, "f^{!}"', yshift=-2.5mm] \ar[l, "\eta_*" description] & \D \ar[l, "f_*" description]
    \end{tikzcd}\]
    and we note that it follows that $\eta_*f_* \simeq \mrm{id}_\D$ and $f^!\eta^! \simeq \mrm{id}_\D$.
    In analogy with $\omega_f \coloneqq f^{!}(\1_\C)$, we set $\omega_{\eta} \coloneqq \eta^!(\1_\D)$. Notice that if $f^*$ has a geometric section $\eta^*$, then both $f_*$ and $\eta^*$ are conservative: $f_*X \simeq 0$ implies $X \simeq \eta_*f_*X \simeq 0$ and similarly for $\eta^*$.
\end{chunk}

\begin{exs}\label{exs:geometricsection}\leavevmode
    \begin{enumerate}[label=(\roman*)]
        \item \label{item:algebra} Let $\C$ be an enhanced rigidly-compactly generated tt-category, and let $f\colon R \to S$ be a map of commutative algebras in $\C$. If $R$ is an $S$-algebra with unit map $\eta\colon S \to R$, then extension of scalars along $\eta$ is a geometric section for extension of scalars along the augmentation $f$.
        \item \label{item:resinf} Let $G$ be a compact Lie group, and let $f^* = \mrm{res}_1^G\colon \Sp_G \to \Sp$ be the restriction functor to the trivial group. This has a geometric section given by the inflation functor $\eta^* = \mrm{inf}_1^G$.
    \end{enumerate}
\end{exs}

\begin{rem}
    Suppose that $f^*\colon \C \to \D$ is a geometric functor with a geometric section $\eta^*$. Since $f_*$ is then conservative as noted in \cref{chunk:geometricsection}, the functor $f^*$ is naturally equivalent to extension of scalars along $\1_\C \to f_*\1_\D$ by \cref{cor:conservative}. Despite this, not every geometric functor $f^*$ with geometric section $\eta^*$ can be expressed (up to equivalence) in the setting of \cref{exs:geometricsection}\ref{item:algebra}. Indeed, if this were true, then $\eta_*$ would be conservative but this is false in \cref{exs:geometricsection}\ref{item:resinf} where $\eta_* = (-)^G$ is the categorical $G$-fixed points functor which is not conservative. 
\end{rem}

The geometric section $\eta^*$ yields a canonical Matlis dualising object for $f^*$, namely $\omega_\eta$. Informally, Gorenstein duality is the assertion that this Matlis dualising object $\omega_\eta$ is suitably compatible with the Picard group of $\D$. This is made precise as follows:

\begin{defn}\label{defn:Gorduality}
    A geometric functor $f^*\colon\C\to\D$ with a geometric section $\eta^*\colon \D \to \C$ is said to have \emph{Gorenstein duality} if there is an $X\in\Pic(\D)$ such that $\Gammaf\omega_{\eta}\simeq\Gammaf\eta^*(X)$.
\end{defn}

\begin{ex}
    Let $(R,\m,k)$ be a commutative Noetherian local ring of Krull dimension $d$, and suppose that it is an algebra over its residue field $k$. We let $f^*\colon \msf{D}(R) \to \msf{D}(k)$ be extension of scalars along $R\to k$, and $\eta^*$ be extension along $k \to R$. Since $\omega_
    \eta$ is the $k$-dual $\Hom_k(R,k)$ which is the injective hull $E(k)$ of $k$, and noting that $E(k)$ is already torsion, Gorenstein duality amounts to the statement that the local cohomology satisfies \[H_\m^*(R)\simeq\Sigma^{-d}E(k).\] In this setting, such an isomorphism in fact characterises Gorenstein rings, see \cref{ex:injectivehull} for more details. Indeed, this example motivates the terminology of Gorenstein duality in \cref{defn:Gorduality}.
\end{ex}

\begin{ex}\label{ex:Poincare}
    For $Y$ a finite, connected CW-complex, and $R$ a commutative local self-injective ring, consider the commutative ring spectrum of cochains $C^*(Y;R)$. The geometric functor associated to the augmentation $C^*(Y;R) \to R$ as in \cref{exs:geometricsection}\ref{item:algebra} has Gorenstein duality if and only if $Y$ has Poincar\'e duality with respect to $R$, see \cite[Remark 2.11]{BCHV}.
    The assumption that $R$ is local means that $\Pic(\msf{D}(R))$ is trivial, and so Gorenstein duality concerns an integer shift. If $R$ is not local, this need not occur. For example, if $Y=\mbb{RP}^2$, then $Y$ has Poincar\'e duality of dimension $0$ over $\Q$, and of dimension $2$ over $\mbb{F}_2$. Taking $R=\Q\times\mbb{F}_2$, we see that $C^*(Y;R)$ has Gorenstein duality, with the $X$ as in \cref{defn:Gorduality} being  $\Q\times\Sigma^{-2}\mbb{F}_2$.
\end{ex}

We observe that Gorenstein duality implies the Gorenstein property itself.

\begin{prop}\label{GordualityimpliesGor}
    Suppose $f^*\colon\C\to\D$ is a proxy-small geometric functor with geometric section $\eta^*$. If $f^*$ has Gorenstein duality, then $f^*$ is Gorenstein.
\end{prop}
\begin{proof}
    Applying the MGM-equivalence to the Gorenstein duality statement we have that $\Lambdaf\omega_\eta \simeq \Lambdaf\eta^*(X)$ for some $X \in \Pic(\D)$. Applying $f^!$ to this together with \cref{prop:changeofbase}, we get \[\1_\D\simeq f^{!}(\omega_{\eta}) \simeq f^{!}\eta^*(X).\] Since $f^*\eta^*(X)\simeq X$ is rigid, by Grothendieck duality \cref{GN} it follows that $\1_\D\simeq\omega_f\otimes X$. Therefore $\omega_f$ is invertible with inverse $X$.
\end{proof}

\begin{rem}\label{rem:GordualityX}
The proof of the previous result shows that Gorenstein duality is in fact equivalent to $\Gammaf\omega_\eta \simeq \Gammaf\eta^*(\omega_f^{-1})$, that is, one may always take $X = \omega_f^{-1}$ in \cref{defn:Gorduality}.
\end{rem}

We note in general that the converse to \cref{GordualityimpliesGor} does not hold, and we relate the equivalence of the two conditions to the orientability of the Matlis lift $\eta^!(\omega_f)$ of $\omega_f$ in the sense of \cref{defn:orientableMatlislift}.
\begin{prop}\label{dualityIffOrientable}
     Let $f^*\colon \C \to \D$ be a proxy-small, Gorenstein geometric functor with a geometric section $\eta^*$. Then $f^*$ has Gorenstein duality if and only if $\eta^!(\omega_f)$ is an orientable Matlis lift of $\omega_f$.
\end{prop}
\begin{proof}
    We note that since $\omega_f\in \D$ is invertible, it is rigid, and so Grothendieck duality shows that $\omega_\eta\otimes\eta^*(\omega_f)\simeq\eta^!(\omega_f)$, see~\cite[Proposition 5.4]{FHM}. (Note that Grothendieck duality holds here without the proxy-smallness assumption on $\eta^*$ since $\omega_f$ is rigid and not just $\eta^*$-rigid.) Therefore the statement of orientability of $\eta^!(\omega_f)$ is precisely \[\Gammaf(\eta^*(\omega_f)\otimes\omega_\eta)\simeq\Gammaf\1_\C\] which by invertibility of $\omega_f$ is equivalent to the Gorenstein duality statement \[\Gammaf\omega_\eta\simeq\Gammaf\eta^*(\omega_f^{-1}).\]
    Since the $X$ in the statement of Gorenstein duality must always be the inverse to $\omega_f$ (\cref{rem:GordualityX}), this shows the two conditions are always equivalent.
\end{proof}

This is the primary usage of automatic orientability: it ensures that the Gorenstein condition and Gorenstein duality are always equivalent.
\begin{cor}\label{thm:Gordualityalgebras}
    Let $f^*\colon \C \to \D$ be a proxy-small geometric functor with a geometric section $\eta^*$, and suppose that $f^*$ is automatically orientable. Then $f^*$ is Gorenstein if and only if it has Gorenstein duality.
\end{cor}
\begin{proof}
    If $f^*$ has Gorenstein duality, then it is Gorenstein by \cref{GordualityimpliesGor}. Conversely, automatic orientability applies to $\eta^!(\omega_f)$ and so the Gorenstein property implies Gorenstein duality by \cref{dualityIffOrientable}.
\end{proof}

\begin{rem}\label{rem:orientabilityneeded?}
    We summarise the implications between the Gorenstein property, Gorenstein duality, and automatic orientability in the following diagram:
     \[\begin{tikzcd}[column sep=2cm]
        \begin{matrix}
            \text{Gorenstein and} \\
            \text{automatically orientable}
        \end{matrix} \arrow[r, Rightarrow, yshift=0.2cm, "(\ref{thm:Gordualityalgebras})"] & 
            \begin{matrix} \text{Gorenstein} \\ \text{duality} \end{matrix} \arrow[negated, Rightarrow, yshift=-0.2cm, "(\ref{ex:resinf})" {yshift=-5pt}]{l} \arrow[r, Rightarrow, yshift=0.2cm, "(\ref{GordualityimpliesGor})"] 
            & 
                \text{Gorenstein} \arrow[negated, Rightarrow, yshift=-0.2cm, "({\ref{ex:resinf},\,\ref{Gorbutnoduality}})" {yshift=-5pt}]{l}
    \end{tikzcd}\]
\end{rem}

We end this section by giving some illustrative examples of Gorenstein duality, which in particular justify the failure of the indicated implications above.

\begin{ex}\label{ex:resinf}
    Let $G$ be a compact Lie group and $f^*=\res^G_1\colon\Sp_G\to\Sp$ be the restriction functor, with geometric section $\eta^*=\inf^G_1\colon\Sp\to\Sp_G$ as in \cref{exs:geometricsection}\ref{item:resinf}. We know that the Wirthm\"uller isomorphism means $f^*$ is Gorenstein with $\omega_f=S^{-d}$ where $d=\dim G$, and we explore Gorenstein duality in this situation. The associated torsion category is the category of free $G$-spectra, and $\Gammaf=EG_+\wedge-$.

    The first observation to make is that $f^*$ is not automatically orientable in general: for instance if $G$ is finite of order $2$ and $\sigma$ denotes the non-trivial one dimensional representation, then the invertible free $G$-spectra $EG_+\wedge S^1$ and $EG_+\wedge S^\sigma$ restrict to the same non-equivariant spectrum $S^1$, but are distinct as free $G$-spectra (distinguished for example by their $0$th integral Bredon cohomology). Therefore, $f^*$ is not automatically orientable by \cref{prop:AOandPic}. 
    
    Gorenstein duality as in \cref{defn:Gorduality}, would take the form of an isomorphism between $EG_+\wedge\omega_\eta$ and $\Sigma^{-d}EG_+$, but does not immediately follow from the fact that $f^*$ is Gorenstein due to the failure of automatic orientability.
    In general, the Matlis dualising object $\omega_\eta$ is difficult to describe fully. Sanders~\cite[(3.17)]{SandersCompactness} has studied this in greater generality, and in particular shows there in an equivalence \begin{equation}\label{SandersIso} EG_+\wedge\omega_\eta\simeq EG_+\wedge S^{-\mrm{ad}(G)},
    \end{equation}
    taking $N=G$ in loc. cit., where $\mrm{ad}(G)$ is the adjoint representation of $G$. Therefore Gorenstein duality is dependent on the adjoint representation:
    \begin{enumerate}
        \item If $\mrm{ad}(G)$ is trivial (for example, when $G$ is finite or abelian), then $EG_+\wedge\omega_\eta\simeq\Sigma^{-d}EG_+$ and so $f^*$ indeed has Gorenstein duality.
        \item If $\mrm{ad}(G)$ is non-trivial, $S^{-\mrm{ad}(G)}$ does not lie in the image of $\eta^*$ and so Gorenstein duality need not occur. For example, if $G=O(2)$ then $\mrm{ad}(G)$ is the non-trivial one-dimensional representation $\sigma$, and one can check that $EG_+\wedge S^1\not\simeq EG_+ \wedge S^\sigma$. Therefore the restriction functor $\Sp_G\to\Sp$ does not have Gorenstein duality in this case, although it is Gorenstein.
    \end{enumerate}
    These cases highlight two important phenomena:
    \begin{enumerate}[label=(\roman*)]
        \item Combining (1) with the fact that $f^*$ is not automatically orientable, we see that Gorenstein duality can hold without automatic orientability. In other words, $\eta^!(\omega_f)$ may happen to be an orientable Matlis lift without all Matlis lifts being orientable.
        \item From (2), we see that Gorenstein geometric functors need not have Gorenstein duality. 
    \end{enumerate}
\end{ex}

The following example gives an alternative case where Gorenstein duality can fail despite the geometric functor being Gorenstein.
\begin{ex}\label{Gorbutnoduality}
Consider the map $f\colon C^*(\mathbb{RP}^2; \mathbb{Z}/4) 
\to \mathbb{Z}/4$ of commutative ring spectra. This is Gorenstein, but does not have Gorenstein duality since $\mathbb{RP}^2$ does not have Poincar\'e duality with respect to $\mathbb{Z}/4$, see \cref{ex:Poincare} and \cite[Example 11.4(ii)]{BraveNewAlgebra}.
\end{ex}

\section{Ascent and descent}\label{sec:descent}
In this section, we prove various ascent and descent results which explain the behaviour of Matlis dualising objects along composites of geometric functors. We will use these in the following section to construct various examples of Gorenstein geometric functors. Throughout this section, we work with geometric functors $g^*\colon \B \to \C$ and $f^*\colon \C \to \D$. For orientation, this means that we have adjunctions
\[\begin{tikzcd}[column sep=2cm]
        \B \ar[r, "g^*", yshift=2.5mm] \ar[r, "g^{!}"', yshift=-2.5mm] & \C \ar[r, "f^*", yshift=2.5mm] \ar[r, "f^{!}"', yshift=-2.5mm] \ar[l, "g_*" description] & \D \ar[l, "f_*" description]
    \end{tikzcd}\]
with left adjoints denoted on top.

\begin{prop}\label{ascentdescentdualising}
    Let $g^*\colon \B \to \C$ and $f^*\colon \C \to \D$ be geometric functors and suppose that $f^*$ is proxy-small. Let $M \in \B$ be a Matlis dualising object for $g^*$. Then $f^*$ is Gorenstein if and only if $M$ is Matlis dualising for $f^*g^*$.
\end{prop}
\begin{proof}
    Since $M$ is Matlis dualising for $g^*$, we have that $g^!M$ is invertible and hence $f^*g^!M$ is invertible (thus rigid). Therefore by Grothendieck duality ~\cite[Proposition 5.4]{FHM}, we have an isomorphism
    \begin{equation}\label{eq:arithmeticshifts}
        f^!g^!M \simeq \omega_f \otimes f^*g^!M
    \end{equation}
    showing that $f^!g^!M$ is invertible if and only if $\omega_f$ is invertible. (Note that Grothendieck duality holds here without a proxy-smallness assumption on $f^*$ since $g^!M$ is rigid and not just $f^*$-rigid.)
\end{proof}

We write $\omega_{gf} \coloneqq f^!g^!(\1_\B)$ in line with our geometric notation.
\begin{cor}\label{ascentdescentGorenstein}
    Let $g^*\colon \B \to \C$ and $f^*\colon \C \to \D$ be geometric functors and suppose that $f^*$ is proxy-small. Suppose that $g^*$ is Gorenstein. Then there is an isomorphism \[\omega_{gf} \simeq \omega_f \otimes f^*\omega_g\] and thus $f^*$ is Gorenstein if and only if $f^*g^*$ is Gorenstein.
\end{cor}
\begin{proof}
    This is the case $M = \1_\B$ of \cref{ascentdescentdualising} together with \cref{eq:arithmeticshifts}.
\end{proof}

\begin{rem}
    In general the converse to \cref{ascentdescentGorenstein} fails: namely, if $f^*$ and $f^*g^*$ are Gorenstein, then $g^*$ need not be Gorenstein. Indeed, suppose that $(R,\m,k)$ is a Gorenstein local ring which is also a $k$-algebra, and set $g^*$ to be extension of scalars along $k \to R$ and $f^*$ to be extension of scalars along $R \to k$. Then $f^*$ is Gorenstein since $R$ is Gorenstein, and $f^*g^* = \mrm{id}$ is Gorenstein trivially. However, $g^*$ is not Gorenstein, since $g^!k = \Hom_k(R,k)$ is not invertible in general. 
\end{rem}

Despite this, a partial converse does exist as a consequence of \cref{thm:dualising}. In order to state it we require the following auxiliary lemma.
\begin{lem}\label{fupperconservative}
    Let $f^*\colon \C \to \D$ be a proxy-small geometric functor. The following conditions are equivalent:
    \begin{enumerate}
        \item $f^*$ is conservative;
        \item $\Loctensor{f_*\1_\D} = \C$;
        \item $\Gammaf$ is the identity functor;
        \item $f^!$ is conservative.
    \end{enumerate}
\end{lem}
\begin{proof}
    That (2) and (3) are equivalent is immediate from the definition of $\Gammaf$. The remaining implications follow from \cref{Wequivalences}.
\end{proof}

\begin{rem}
     The previous result extends \cite[Proposition 13.33]{BCHS} from the case where $f_*\1_\D$ is compact, to the proxy-small setting. This further narrows down the possible counterexamples to \cite[Question 21.6]{BCHS}.
\end{rem}

\begin{prop}\label{prop:ascentdescentconverse}
    Let $g^*\colon \B \to \C$ and $f^*\colon \C \to \D$ be geometric functors and suppose that $f^*$ is proxy-small and conservative, and that $\D$ is pure semisimple. If $f^*$ and $f^*g^*$ are Gorenstein then $g^*$ is Gorenstein.
\end{prop}
\begin{proof}
    Since $f^*g^*$ is Gorenstein, $\omega_{gf} \simeq f^!(\omega_g)$ is invertible, and hence $\Gammaf(\omega_g) \in \Pic(\Gammaf\C)$ by \cref{thm:dualising}. 
    As $f^*$ is conservative, the functor $\Gammaf$ is the identity by \cref{fupperconservative}. We therefore have that $\omega_g \in \Pic(\C)$ as required.
\end{proof}

\begin{rem}
    The previous result shows that one can often detect the Gorenstein property by applying a conservative Gorenstein functor, see \cref{FormalDGAGor} and \cref{rem:noneqGor} for examples.
\end{rem}

We end this section by explaining how automatic orientability composes.

\begin{prop}\label{prop:AOimpliesPicmonic}
    Let $g^*\colon \B \to \C$ and $f^*\colon \C \to \D$ be proxy-small geometric functors, and moreover that $f_*\1_\D$ is proxy-small relative to $g^*$. If $g^*$ is Gorenstein and has automatic orientability, then there is an injective group homomorphism \[\Pic(\Gamma_{\!gf}\B) \to \Pic(\Gammaf\C)\]
    given by $X \mapsto g^*X$.
\end{prop}
\begin{proof}
    Since $f_*\1_\D$ is proxy-small relative to $g^*$, by \cref{prop:changeofbase} we have a natural isomorphism $\Gammaf g^* \simeq g^* \Gamma_{\!gf}$. This shows that the assignment $X \mapsto g^*X$ does define a group homomorphism as in the statement. 
    
    We now prove injectivity, so suppose that $g^* \Gamma_{\!gf}X \simeq \Gammaf \1_\C$. By the isomorphism $\Gammaf g^* \simeq g^* \Gamma_{\!gf}$ we obtain that
    \[g^* \Gamma_{\!gf} X \simeq \Gammaf g^*\1_\B \simeq g^* \Gamma_{\!gf} \1_\B\] and since $g^*$ has automatic orientability (see \cref{prop:AOandPic}) we deduce that $\Gamma_{\!g} \Gamma_{\!gf} X \simeq \Gamma_{\!g} \Gamma_{\!gf} \1_\B$. Since $g^*$ is proxy-small, we have $\Gamma_{\!g} g_* \simeq g_*$ by \cref{prop:changeofbase}. Therefore
    \[\Gamma_{\!gf}\B = \Loctensor{g_*f_*\1_\D} \subseteq \Gamma_{\!g}\B = \Loctensor{g_*\1_\C}.\]
    Thus $\Gamma_{\!g} \Gamma_{\!gf} \simeq \Gamma_{\!gf}$, and we conclude that $\Gamma_{\!gf} X \simeq \Gamma_{\!gf}\1_\B$, proving injectivity.
\end{proof}

\begin{cor}\label{descentAO}
    Let $g^*\colon \B \to \C$ and $f^*\colon \C \to \D$ be proxy-small geometric functors, and moreover that $f_*\1_\D$ is proxy-small relative to $g^*$ and $\D$ is pure semisimple. If $f^*$ and $g^*$ are Gorenstein and have automatic orientability, then $f^*g^*$ has automatic orientability.
\end{cor}
\begin{proof}
    Consider the composite 
    \[\Pic(\Gamma_{\!gf}\B) \xrightarrow{g^*} \Pic(\Gammaf\C) \xrightarrow{f^*} \Pic(\D).\] The first map is injective by \cref{prop:AOimpliesPicmonic} and the latter map is injective as $f^*$ has automatic orientability (see \cref{prop:AOandPic}). Therefore the composite is injective, and hence $f^*g^*$ has automatic orientability by applying \cref{prop:AOandPic} again.
\end{proof}

\section{Examples}
In this section we explore some examples and applications of the theory developed in this paper.

\subsection{Commutative algebra}\label{ex:injectivehull}
    Let $(R,\m,k)$ be a commutative Noetherian local ring, and consider the geometric functor $f^*\colon \msf{D}(R) \to \msf{D}(k)$ given by extension of scalars along the ring map $R \to k$. This is proxy-small by \cref{rem:localringps}. The associated torsion and completion functors compute the local (co)homology at the maximal ideal $\m$: for any $R$-module $X$ we have 
    \[H_n(\Gammaf X) \simeq H_\m^{-n}(X) \quad \text{and} \quad H_n(\Lambdaf X) \simeq H^\m_n(X),\] see \cite[\S 6]{DwyerGreenlees}. In particular, $\Lambdaf R \simeq R_\m^\wedge$ is the $\m$-adic completion of $R$. 
    
    The injective hull of the residue field, $E(k)$, is a Matlis dualising object since $f^{!}(E(k)) = \Hom_R(k,E(k)) \simeq k$. \cref{Mduality} then asserts that $X \to D^2_{E(k)}(X)$ is an isomorphism for all $X \in \thicktensor{k}$, that is, for all $X \in \msf{D}(R)$ which consist of finite length modules and have bounded homology. As such, this recovers the usual form of Matlis duality for complexes. 
    
    Now suppose that $R$ is moreover Gorenstein (and note that this is equivalent to $f^*$ being Gorenstein, see \cref{exs:Gorenstein}). We firstly note that $f^*$ is automatically orientable. This can be deduced in several ways. One can use the viewpoint of Morita theory as in \cref{section:orientations}; see \cite[Proposition 3.9]{DGI} or \cite[Proposition 18.1]{GreenleesCRM}. Alternatively, by \cite[Proposition 4.4(3)]{BIKP}, $k \otimes_R M$ is a shift of $k$ if and only if $\Lambdaf M$ is isomorphic to some shift of $R_\m^\wedge$. By the MGM equivalence, we see that $f^*\colon \Pic(\Gammaf\msf{D}(R)) \to \Pic(\msf{D}(k))$ is injective, and hence $f^*$ is automatically orientable by \cref{prop:AOandPic}.
    
    The duality statement of \cref{Matlisliftofomega} implies that \[\Hom_R(\Gammaf X, E(k)) \simeq \Sigma^{-\mrm{dim}(R)}\Hom_R(X, R_\m^\wedge).\]
    In particular, if $X$ is an $R$-module, applying $H_{-i}$ to the above gives an isomorphism
    \[\Hom_R(H_\m^{\mrm{dim}(R)-i}(X), E(k)) \simeq \mrm{Ext}^i_R(X, R_\m^\wedge)\]
    so that we recover the usual Grothendieck local duality theorem for Gorenstein rings.

    Now suppose that $(R,\m,k)$ is moreover a $k$-algebra. The functor $f^*$ has a geometric section $\eta^*$ given by extension of scalars along $k \to R$. Then the Gorenstein duality statement of \cref{thm:Gordualityalgebras} yields the calculation that \[H_{\m}^*(R) \simeq \Sigma^{-\mrm{dim}(R)}E(k)
    \] since $\omega_\eta = \Hom_k(R,k) = E(k)$ is always $\m$-power torsion.

\subsection{Totalisation and formal DG algebras}\label{sec:totalisation}
    In this section we consider a particular geometric functor between two reasonable notions of derived categories over a graded ring. In particular we use the results of this paper to show algebraic and homotopical versions of the Gorenstein property coincide for formal DG algebras; the framework developed in this paper provides a mechanisation of techniques suggested in \cite[Remark A.7]{Nucleus}.
    
    \begin{chunk}
        Let $R$ be a graded commutative ring. There are two rigidly-compactly generated tt-categories associated to $R$:
    \begin{itemize}
        \item $\DCh(R)$: the derived category of chain complexes of graded $R$-modules;
        \item $\DDG(R)$: the derived category of differential graded (DG) $R$-modules, where $R$ is viewed as a DG algebra with zero differential.
    \end{itemize}
    Note that the objects in the former are \emph{bi}graded, whereas objects in the latter are singly graded. The tensor unit in both categories is $R$ viewed as an object in the appropriate way. To avoid confusion, we write $\1_\mrm{Ch}$ and $\1_\mrm{DG}$ to indicate this.
    \end{chunk}

    \begin{chunk}
        The \emph{totalisation functor} $T^*\colon\DCh(R)\to\DDG(R)$ sends a chain complex $C=(C_{ij})$ to the DG module with underlying graded module $(T^*(C))_n=\bigoplus_{i+j=n}C_{ij}$ with componentwise differentials. Note this is indeed well-defined on derived categories, since the homology of $T^*(C)$ for $C\in\DCh(R)$ is the totalisation of $H_{**}(C)$. The functor $T^*$ preserves coproducts, and is strong symmetric monoidal, so that $T^*$ is a geometric functor.
    \end{chunk}

    \begin{chunk}\label{chunk:totps}
        The right adjoint $T_*\colon \DDG(R) \to \DCh(R)$ of $T^*$ can be described as follows: given a DG module $(M,d)$, $T_*(M)$ is the chain complex \[\dots\to\Sigma^1M\xrightarrow{d}M\xrightarrow{d}\Sigma^{-1}M\to\dots\]
        In particular, $T_*(\1_\mrm{DG})$ is the chain complex \[(\dots\to\Sigma R\xrightarrow{0}R\xrightarrow{0}\Sigma^{-1}R\to\dots)=\bigoplus_{n\in\Z}\Sigma^{n,-n}R\] which contains $R = \1_\mrm{Ch}$ as a summand. Therefore $T^*$ is proxy-small with witness $\1_\mrm{Ch}$ and $\Loc(T_*(\1_\mrm{DG})) = \DCh(R)$, so by \cref{fupperconservative} the totalisation functor $T^*$ is conservative and the associated torsion functor $\Gamma_T$ is the identity.
    \end{chunk}
    
    \begin{chunk}\label{chunk:totomega}
    The further right adjoint $T^!\colon \DCh(R) \to \DDG(R)$ is the direct product totalisation \[(T^!(C))_n=\prod_{i+j=n}C_{ij}\] with componentwise differentials. Therefore, $\omega_T=T^!(R)$ is simply a copy of $R$ as a DG module, i.e., $\omega_T \simeq \1_\mrm{DG}$. Hence $T^*$ is  a Gorenstein geometric functor.
    \end{chunk}

\begin{lem}\label{totalCompact}
    Let $C \in \DCh(R)$. Then $C\in\DCh(R)$ is rigid if and only if $T^*(C) \in \DDG(R)$ is rigid.
\end{lem}
\begin{proof}
    By \cref{f*rigid}, we have that $T^*(C)$ is rigid if and only if $\Gamma_T C$ is rigid in $\Gamma_T\DCh(R)$, but $\Gamma_T$ is the identity functor in this case, see \cref{chunk:totps}.
\end{proof}

\begin{chunk}
    Let $A$ be a commutative DG algebra with $H_*A$ a local Noetherian graded ring and write $k$ for the residue field of $H_*A$. We say that $A$ is \emph{homotopically regular} if $k$ is a compact DG $A$-module, and \emph{algebraically regular} if $H_*A$ is regular as a local graded ring (i.e., $k$ has finite projective dimension over $H_*A$). Say that $A$ is \emph{formal as an augmented DGA}, if there is a zig-zag of quasi-isomorphisms $A \simeq H_*A$ of DG-algebras, which are compatible with the augmentations $A \to k$ and $H_*A \to k$. Note that that the triangulated equivalence $\msf{D}(A) \simeq \DDG(H_*A)$ induced by extension and restriction of scalars then sends $k$ to $k$. As such, \cref{totalCompact} shows that if $A$ is formal as an augmented DGA, then it is homotopically regular if and only if it is algebraically regular.
\end{chunk}

\begin{chunk}
    We say that $A$ is \emph{homotopically Gorenstein} if the induced extension of scalars functor $\DDG(A)\to\DDG(k)$ is a Gorenstein geometric functor. Since $\Pic(\DDG(k))$ is trivial, this is equivalent to $A$ being Gorenstein as a ring spectrum in the sense of \cite{DGI}, namely that $\Hom_A(k,A) \simeq \Sigma^i k$ for some $i \in \Z$. We say that $A$ is \emph{algebraically Gorenstein} if $H_*A$ is Gorenstein as a local graded ring. Using the ascent and descent properties of Gorenstein geometric functors proved above, we can now prove that these notions coincide for suitably formal DGAs:
\end{chunk}

\begin{prop}\label{FormalDGAGor}
    Let $A$ be a commutative DG algebra with $H_*A$ a local Noetherian graded ring. If $A$ is formal as an augmented DGA, then $A$ is homotopically Gorenstein if and only if it is algebraically Gorenstein.
\end{prop}
\begin{proof}
    Write $R=H_*A$ so that there is a triangulated equivalence $\msf{D}(A)\simeq\DDG(R)$ by formality. There are two functors associated to the map $f\colon R \to k$ given by extension of scalars: $f^*_\mrm{Ch}\colon \DCh(R)\to\DCh(k)$ and $f^*_\mrm{DG}\colon \DDG(R)\to\DDG(k)$. 
    Notice that $f^*_\mrm{Ch}$ is a Gorenstein geometric functor if and only if $A$ is algebraically Gorenstein, and that $f^*_\mrm{DG}$ is a Gorenstein geometric functor if and only if $A$ is homotopically Gorenstein. So it suffices to show that $f^*_\mrm{Ch}$ is Gorenstein if and only if $f^*_\mrm{DG}$ is Gorenstein. 
    
    The construction of the geometric functor $T^*$ is natural in maps of graded rings so that we have a commutative square \[\begin{tikzcd}
	{\DCh(R)} & {\DCh(k)} \\
	{\DDG(R)} & {\DDG(k)}
	\arrow["{f^*_\mrm{Ch}}", from=1-1, to=1-2]
	\arrow["{T^*_R}"', from=1-1, to=2-1]
	\arrow["{T^*_k}", from=1-2, to=2-2]
	\arrow["{f^*_\mrm{DG}}"', from=2-1, to=2-2]
\end{tikzcd}\] of geometric functors. Recall from \cref{chunk:totps} and \cref{chunk:totomega} that the functors $T_R^*$ and $T_k^*$ are both proxy-small Gorenstein geometric functors.
If $f^*_\mrm{Ch}$ is Gorenstein, then by applying \cref{ascentdescentGorenstein}, the composite $T^*_k\circ f^*_\mrm{Ch} \simeq f^*_\mrm{DG}\circ T^*_R$ and hence also $f^*_\mrm{DG}$ is Gorenstein. Conversely, if $f^*_\mrm{DG}$ is Gorenstein then $T^*_k\circ f^*_\mrm{Ch}$ is also Gorenstein. Since $T^*_k$ is conservative, by applying \cref{prop:ascentdescentconverse} we conclude that $f^*_\mrm{Ch}$ is also Gorenstein. 
\end{proof}
\begin{rem}
    In general, if a DG algebra $A$ is algebraically Gorenstein, the spectral sequence 
    \begin{equation}\label{UCSS}
            E_2^{**}=\mrm{Ext}^{**}_{H_*A}(k, H_*A)\Longrightarrow H_*(\Hom_A(k,A))
    \end{equation}
    must collapse, so that $A$ is necessarily homotopically Gorenstein, and \cref{FormalDGAGor} gives a converse for formal DG algebras.
It should be noted that for ungraded rings, the $E_2$-page of \cref{UCSS} is concentrated in a single line and so must collapse. However for graded rings one cannot rely on the geometry of the spectral sequence in the same way, hence the necessity of the argument in \cref{FormalDGAGor}.
\end{rem}

We end this section by providing the promised example of a proxy-small geometric functor $f^*$ and an object $X$ for which $\Gammaf X\in\Gammaf \C$ is reflexive but $f^*X\in\D$ is not, see \cref{rem:reflexivefalse}. 
\begin{ex}\label{ex:reflexivetot}
    Consider a field $k$ and the totalisation functor $T^*\colon\DCh(k)\to\DDG(k)$. In this case $\Gamma_T\DCh(k)=\DCh(k)$, see \cref{chunk:totps}. Take $X=\bigoplus_{n\in\Z}\Sigma^{n,-n}k\in\DCh(k)$. This is reflexive in $\Gamma_T\DCh(k) = \DCh(k)$ since it is finite in every bidegree. However, $T^*X$ is concentrated in degree zero where it is $\bigoplus_{n\in\Z}k$. Thus $T^*X$ is not finite in each degree and therefore is not reflexive in $\DDG(k)$.
\end{ex} 

\subsection{Invariant rings and Watanabe's theorem}\label{WatanabeSection}
In this section we give an example of a Gorenstein geometric functor coming from polynomial invariant theory, thus illustrating the broadness of the definition in this paper. In particular, we give a new perspective on Watanabe's theorem, showing that it can be deduced as a consequence of an equivariant Gorenstein duality statement which holds more generally.

\begin{chunk}\label{invariantsSetup}
    Let $k$ be a field, $G$ a finite group with $|G|$ invertible in $k$, and $V$ a finite $kG$-module of dimension $d$. We consider the symmetric algebra $k[V]$ associated to $V$, which is a polynomial algebra acted on by $G$, and we consider the invariants $k[V]^G$. We shall use the book \cite{BensonInvariants} as a standard reference, noting in loc. cit. that $k[V]$ is instead defined to be the symmetric algebra on the dual of $V$.
    We consider the ring map $k[V]^G\to k$ induced by $V\to 0$, and in particular note that the geometric functor given by extension along this map is Gorenstein if and only if $k[V]^G$ is Gorenstein when localised at the augmentation ideal.
\end{chunk}

We recall the following theorem of Watanabe~\cite{Watanabe} describing when the invariant ring is Gorenstein, which we will give a new viewpoint on in the remainder of the section.
\begin{thm}[Watanabe]\label{Watanabe}
    Let $(k,G,V)$ be as in \cref{invariantsSetup}. If $G$ acts on $V$ with determinant $1$, then $k[V]^G\to k$ is Gorenstein.
\end{thm}
The framework of this paper allows us to extend this to a Gorenstein phenomenon for \emph{all} representations $V$, by allowing non-integer shifts by the determinant, and we then explain how to recover \cref{Watanabe} from such a description. We now give the necessary setup:

\begin{chunk}\label{DkGSetting}
    Let $k$ be a field and $G$ a finite group with $\left|G\right|$ invertible in $k$. The derived category $\msf{D}(kG)$ of the group algebra is a rigidly-compactly generated tt-category with tensor product $\otimes_k$. Note by Maschke's theorem that $\msf{D}(kG)$ is pure semisimple, and we identify the invertible objects as suspensions of one-dimensional representations of $G$, so that $\Pic(\msf{D}(kG))\simeq\Z\oplus\Hom(G,k^\times)$.
    Algebras in (an enhancement of) $\msf{D}(kG)$ are understood as DG algebras with a $G$-action, and their modules as \emph{twisted modules}, i.e., DG modules of the underlying DG algebra along with a compatible $G$-action.
\end{chunk}

We note that given an $R\in\CAlg{\msf{D}(kG)}$ the restriction functor $\mrm{res}^G_1\colon\msf{D}(kG)\to\msf{D}(k)$, induces a geometric functor $\res^G_1\colon\Mod{\msf{D}(kG)}(R)\to\Mod{\msf{D}(k)}(\res^G_1 R)$.
\begin{lem}\label{underlyingDetectsRigid}
    Let $R\in\CAlg{\msf{D}(kG)}$ and $M$ be an $R$-module. Then $M$ is rigid as an $R$-module if and only if $\res^G_1M$ is rigid as a $\res^G_1R$-module.
\end{lem}
\begin{proof}
    The restriction functor $\res^G_1\colon\Mod{\msf{D}(kG)}(R)\to\Mod{\msf{D}(k)}(\res^G_1 R)$ is conservative and has right adjoint given by $-\otimes_kkG$. Since $k$ is compact in $\msf{D}(kG)$, $\res^G_1$ is a proxy-small geometric functor. Therefore the result follows from an application of \cref{f*rigid} together with \cref{fupperconservative}.
\end{proof}

Our particular case of interest is the map $f\colon k[V]\to k$ as in \cref{invariantsSetup}, which can be viewed as a map in the category $\CAlg{\msf{D}(kG)}$, and we consider the corresponding extension of scalars functor \[f^*\colon \Mod{\msf{D}(kG)}(k[V])\to\Mod{\msf{D}(kG)}(k)= \msf{D}(kG)\]
noting that the equality holds as $k$ is the tensor unit in $\msf{D}(kG)$. This is a geometric functor, and has right adjoint $f_*$ given by restriction of scalars along $f$. Since $\res^G_1k[V]$ is a polynomial ring, \cref{underlyingDetectsRigid} immediately implies the following.

\begin{cor}
    As an object of $\Mod{\msf{D}(kG)}(k[V])$, $k$ is rigid.
\end{cor}

Our reinterpretation of Watanabe's theorem is then as follows.
\begin{prop}\label{prop:ourWatanabe}
    Let $(k,G,V)$ be as in \cref{invariantsSetup}. The extension of scalars functor $f^*$ associated to $k[V]\to k$ is Gorenstein, with $\omega_f=\Sigma^{-d}\det^{-1}$, where $\det$ is the determinant character of $V$ and $d = \dim(V)$.
\end{prop}
\begin{proof}
    We use \cite[Proposition 4.6.1]{BensonInvariants}, which in our context is the statement that \[\Hom_B(k[V],B)\simeq k[V]\otimes_k\det\] as $k[V]$-modules for some Noether normalisation $B$ of $k[V]^G$. Write $g$ for the composite $B \to k[V]^G \hookrightarrow k[V]$. Considering $g\colon B\to k[V]$ as a map of algebras in $\msf{D}(kG)$, the above isomorphism shows that $g^*$ is Gorenstein with $\omega_g=k[V]\otimes_k\det$.

    Since $B$ is a polynomial ring with trivial action, we have that $f^*g^*\colon B\to k$ is Gorenstein with $\omega_{gf}=\Sigma^{-d}k$, and so by \cref{ascentdescentGorenstein} it follows that $f^*$ is Gorenstein, and the formula\[\omega_{gf}\simeq\omega_f\otimes f^*(\omega_g)\] shows that $\omega_f\simeq\Sigma^{-d}\det^{-1}$ as required.
\end{proof}

\begin{rem}\label{rem:noneqGor}
    The content of the previous result is the identification of the group action on $\omega_f$. The fact that $f\colon k[V]\to k$ is Gorenstein follows from the fact it is true non-equivariantly, in an argument analogous to \cref{FormalDGAGor}. Indeed, we consider the commutative diagram \[\begin{tikzcd}[column sep=1.5cm]
	{\Mod{\msf{D}(kG)}(k[V])} & {\msf{D}(kG)} \\
	{\Mod{\msf{D}(k)}(\res^G_1k[V])} & {\msf{D}(k)}
	\arrow["{f^*}", from=1-1, to=1-2]
	\arrow["{\res^G_1}"', from=1-1, to=2-1]
	\arrow["{\res^G_1}", from=1-2, to=2-2]
	\arrow["{(\res^G_1f)^*}"', from=2-1, to=2-2]
\end{tikzcd}\] where both instances of $\res^G_1$ are Gorenstein and conservative. Since, in ordinary commutative algebra, polynomial rings are Gorenstein, we know $(\res^G_1f)^*$ must be also Gorenstein, and so the down-right composite is Gorenstein by \cref{ascentdescentGorenstein}. Therefore the other composite $\res^G_1\circ f^*$ must also be Gorenstein, and so $f^*$ is Gorenstein by \cref{prop:ascentdescentconverse}.
\end{rem}

\begin{rem}
    As stated, \cref{prop:ourWatanabe} does not immediately imply \cref{Watanabe}: this notion of algebras in $\msf{D}(kG)$ being Gorenstein does not in general imply any statement about their fixed points. However, if $G$ acts on $V$ with determinant $1$, the proof provides an equivalence \[\Hom_B(k[V],B)\simeq k[V]\]
    which we can use to deduce \cref{Watanabe}. Indeed, taking fixed points tells us \[\Hom_{B^G}(k[V]^G,B^G)\simeq\Hom_B(k[V],B)^G\simeq k[V]^G\] where the first equivalence comes from the fact $B$ has trivial $G$-action. This means $B^G\to k[V]^G$ is Gorenstein, and we deduce that $k[V]^G\to k$ is also Gorenstein since $B^G$ is a polynomial ring, using \cref{ascentdescentGorenstein}. 
\end{rem}

However in a general setting, we now show that Gorenstein \emph{duality} behaves well with respect to fixed points, providing an alternative way to recover \cref{Watanabe} from \cref{prop:ourWatanabe}. This in particular acts as proof of concept for using equivariant Gorenstein duality to produce non-equivariant duality statements by applying fixed points, an idea which we intend to return to in future work.

Given a commutative DG algebra $R$, we write $R_* = H_*R$ for its homology. If moreover $R$ is a commutative algebra in $\msf{D}(kG)$, as fixed points commute with homology, we write $R_*^1$ and $R_*^G$ for the homology of the fixed points.

\begin{prop}\label{prop:fixedpoints}
    Let $f\colon R\to k$ be an augmented commutative algebra in $\msf{D}(kG)$, such that $R^1_*$ is a Noetherian local graded ring with residue field $k$. 
    Suppose that $f\colon R\to k$ has Gorenstein duality with $\omega_f\simeq\Sigma^a k$ for some $a\in\Z$. Then the corresponding $G$-fixed point map $f^G\colon R^G\to k$ in $\CAlg{\msf{D}(k)}$ also has Gorenstein duality, with $\omega_{f^G}\simeq\Sigma^ak$.
\end{prop}
\begin{proof}
    Since we consider torsion functors in different categories in this proof, we indicate this with a superscript to avoid confusion. For example, $\Gamma_k^R$ denotes the $k$-torsion functor in $\Mod{\msf{D}(kG)}(R)$. The Gorenstein duality statement for $R\to k$ reads \[\Gamma^R_k R\simeq \Sigma^a\Gamma^R_k\Hom_k(R,k)\]
    and we wish to take $G$-fixed points of this statement. Indeed, we note $\Hom_k(R,k)^G\simeq\Hom_k(R^G,k)$ by Schur's lemma (or a localising subcategory argument), and so it suffices to show that fixed points commute with torsion, in that $\Gamma^{R^G}_k(M^G)\simeq(\Gamma^R_kM)^G$ for all $M \in \Mod{\msf{D}(kG)}(R)$.

    To prove this, choose a Noether normalisation of the graded algebra $R^G_*$, i.e., a polynomial subalgebra over which $R^G_*$ is finitely generated. A choice of generators gives a morphism $V\coloneqq\bigoplus_i\Sigma^{d_i}k\to R$ in $\msf{D}(kG)$, yielding a map $S'\coloneqq k[V]\to R$ of commutative algebras in $\msf{D}(kG)$. Note that $S'$ has trivial action, so we may view it as $S' = \mrm{inf}S$ for a commutative algebra $S$ in $\msf{D}(k)$, where $\mrm{inf}\colon \msf{D}(k) \to \msf{D}(kG)$ denotes the inflation functor. 
    Since $R^G_*\hookrightarrow R^1_*$ is a finite extension, it follows that $R^1_*$ is finitely generated over $S_*$, and as $S_*$ is regular it follows from \cref{underlyingDetectsRigid} that $R$ and $k$ are both compact as $\mrm{inf}S$-modules. Therefore $\mrm{inf}S \to R \to k$ is a normalisation, and so is proxy-small relative to extension of scalars along $\mrm{inf}S \to R$ with witness $Q=R\otimes_Sk$ (\cref{ex:normalisation}). 
    By adjunction, we obtain a map of commutative algebras $S \to R^G$ in $\msf{D}(kG)$, and $S \to R^G \to k$ is again a normalisation, so $k$ is proxy-small relative to $S \to R^G$ (\cref{ex:normalisation}). 
    
    Since $k = k^G$, one also checks that $k$ is proxy-small relative to $\mrm{inf}\colon \Mod{\msf{D}(k)}(S) \to \Mod{\msf{D}(kG)}(\mrm{inf}S)$, and is its own witness in both categories. Writing $\mrm{ext}_S^{R^G}$ and $\mrm{res}_S^{R^G}$ for the extension and restriction of scalars functors along $S \to R^G$, and similarly for $\mrm{inf}S \to R$, we
    therefore have the following:
    \begin{enumerate}[label=(\roman*)]
        \item $k$ is relatively proxy-small with respect to $\mrm{ext}_{\mrm{inf}S}^R\colon \Mod{\msf{D}(kG)}(\mrm{inf}S) \to \Mod{\msf{D}(kG)}(R)$;
        \item $k$ is relatively proxy-small with respect to $\mrm{ext}_{S}^{R^G}\colon \Mod{\msf{D}(k)}(S) \to \Mod{\msf{D}(k)}(R^G)$;
        \item $k$ is relatively proxy-small with respect to $\mrm{inf}\colon \Mod{\msf{D}(k)}(S) \to \Mod{\msf{D}(kG)}(\mrm{inf}S)$.
    \end{enumerate} 
    
    Also note that by construction, the following diagram commutes:
    \begin{equation}\label{eq:fixedptscommute}
        \begin{tikzcd}[column sep=2cm, row sep=1.5cm]
            \Mod{\msf{D}(kG)}(R) \ar[r, "(-)^G"] \ar[d, "\mrm{res}^{R}_{\mrm{inf}S}"'] & \Mod{\msf{D}(k)}(R^G) \ar[d, "\mrm{res}^{R^G}_S"] \\
            \Mod{\msf{D}(kG)}(\mrm{inf}S) \ar[r, "(-)^G"'] & \Mod{\msf{D}(k)}(S)
        \end{tikzcd}
    \end{equation}

    Using the above we have the following chain of isomorphisms for all $M \in \Mod{\msf{D}(kG)}(R)$:
    \begin{align*}
        \mrm{res}^{R^G}_S((\Gamma_k^R M)^G) &\simeq (\mrm{res}^{R}_{\mrm{inf}S} \Gamma_k^R M)^G &\text{by \eqref{eq:fixedptscommute}} \\
        &\simeq (\Gamma_k^{\mrm{inf}S} \mrm{res}_{\mrm{inf}S}^{R}M)^G &\text{by (i) with \cref{prop:changeofbase}} \\
        &\simeq \Gamma_k^S (\mrm{res}_{\mrm{inf}S}^R M)^G &\text{by (iii) with \cref{prop:changeofbase}} \\
        &\simeq \Gamma_k^S \mrm{res}_S^{R^G} M^G &\text{by \eqref{eq:fixedptscommute}} \\
        &\simeq \mrm{res}_S^{R^G} \Gamma_k^{R^G} M^G &\text{by (ii) with \cref{prop:changeofbase}.}
    \end{align*}
    One checks that the above isomorphism is the image of the canonical map $\Gamma_k^{R^G}(M^G) \to (\Gamma_k^R M)^G$ under $\mrm{res}_S^{R^G}$. Since $\mrm{res}_S^{R^G}$ is conservative, we conclude.
\end{proof}

\begin{rem}
    In order to use \cref{prop:fixedpoints} to recover \cref{Watanabe}, it remains to show that $f^*\colon \Mod{\msf{D}(kG)}(k[V]) \to \msf{D}(kG)$ has Gorenstein duality, rather than just being Gorenstein. We do this by showing that $f^*$ is automatically orientable, see \cref{thm:Gordualityalgebras}. To prove automatic orientability, one uses \cref{cor:AOandMorita} and shows that there is a unique $\Ef=\Hom_{k[V]}(k,k)$-module structure on any module with homology concentrated in a single degree.

    This is true non-equivariantly since $\mrm{res}^G_1\Ef$ is connective and we reduce to this case, noting $H_0\Ef\simeq k$. Indeed, take two $\Ef$-modules $M$ and $N$ whose homologies are both concentrated in degree $0$. Then, \cite[Remark 3.10]{DGI} shows the natural map \[H_0\Hom_{\res^G_1\Ef}(\res^G_1M,\res^G_1N)\to\Hom_k(H_0M,H_0N)\] is an isomorphism. The crucial point is that this map is just a restriction of the \emph{equivariant} map $H_0\Hom_{\Ef}(M,N)\to\Hom_k(H_0M,H_0N)$ and so there is an isomorphism on fixed points. If $M$ and $N$ are both $\Ef$-lifts of $k$, then the isomorphism $H_0M\simeq H_0N$ in $\Hom_k(k,k)^G$ can therefore be realised as a map $f\in H_0\Hom_{\Ef}(M,N)^G$, i.e., an $\Ef$-module equivalence $M\xrightarrow{\sim} N$.
\end{rem}

We end this section with a non-polynomial example given by Watanabe \cite[II, Example 3]{Watanabe}. A natural question one can ask is whether the fixed points of a non-polynomial Gorenstein ring are Gorenstein, and \cref{Watanabe} suggests this may occur when $G$ acts with determinant $1$ on the cotangent space $\m/\m^2$ at the augmentation ideal $\m$. The following demonstrates that this is \emph{not} the case.

\begin{ex}\label{ex:Watanabenonpoly}
    Let $k$ be a field of characteristic $\neq3$, and $\xi\in k$ be a primitive cube root of unity. We let $R=k[x,y,z]/(xy-z^2)$ with a $G=C_3$ action via $(x,y,z)\mapsto(\xi x,\xi y,\xi z)$. As noted in \cite[II, Example 3]{Watanabe}, $G$ acts on the cotangent space with determinant $1$ but $R^G$ is not Gorenstein at the origin. 

    In our context, we find that $f\colon R\to k$ is Gorenstein (see \cref{rem:noneqGor}), however $\omega_f$ is not a suspension of $k$.
    Indeed, consider the projection $g\colon S=k[x,y,z]\to R$. This is equivariant, however the cofibre sequence \[\mrm{res}^G_1 S\xrightarrow{xy-z^2} \mrm{res}^G_1 S\to \mrm{res}^G_1 R\] of $\mrm{res}^G_1 S$-modules is not, since $G$ acts on $xy-z^2$ via multiplication by $\xi^2$.
    Instead, writing $\chi$ for the one-dimensional character $1\mapsto\xi$, there is a cofibre sequence \[S\otimes_k\chi\xrightarrow{xy-z^2}S\to R.\]
    Therefore one computes $\omega_g=\Hom_S(R,S)\simeq\Sigma^{-1}\chi^2$. By \cref{prop:ourWatanabe} we know that $\omega_{gf}=\Sigma^{-3}k$ and so an application of \cref{ascentdescentGorenstein} tells us $f^*\colon\Mod{\msf{D}(kG)}(R)\to\msf{D}(kG)$ is Gorenstein with $\omega_f=\Sigma^{-2}\chi$.
\end{ex}

\begin{rem}
    The results of this section, particularly \cref{prop:fixedpoints}, highlight the exact nature of this counterexample. Namely, the fact that $k[x,y,z]/(xy-z^2)$ is not a polynomial ring means the action on the cotangent space $\m/\m^2$ does not dictate the action on $\omega_f$, in constrast to the case for \cref{prop:ourWatanabe}.
\end{rem}

\bibliographystyle{abbrv}
\bibliography{references.bib}
\end{document}